\documentclass{amsart}
\usepackage{amsmath}
\usepackage{enumerate}
\usepackage{amsmath,amsthm,amscd,amssymb}

\usepackage{latexsym}
\usepackage{upref}
\usepackage{verbatim}

\usepackage[mathscr]{eucal}
\usepackage{dsfont}

\usepackage{graphicx}
\usepackage[colorlinks,hyperindex,pagebackref,hypertex]{hyperref}

\newtheorem{theorem}{Theorem}
\newtheorem{proposition}[theorem]{Proposition}
\newtheorem{corollary}[theorem]{Corollary}
\newtheorem{lemma}[theorem]{Lemma}
\newtheorem{definition}[theorem]{Definition}
\newtheorem{remark}[theorem]{Remark}
\newtheorem{hypothesis}[theorem]{Hypothesis}

\chardef\bslash=`\\ 

\newcommand{\wh}{\widehat}

\newcommand{\whA}{T}
\newcommand{\whB}{T_{\cB}^\kappa}

\newcommand{\dA}{{\dot A}}

\newcommand{\bbR}{{\mathbb{R}}}

\newcommand{\bbC}{{\mathbb{C}}}

\newcommand{\ran}{\text{\rm{Ran}}}

\newcommand{\ti}{\tilde  }

\newcommand{\dom}{\text{\rm{Dom}}}

\newcommand{\calA}{{\mathcal A}}
\newcommand{\calB}{{\mathcal B}}
\newcommand{\calC}{{\mathcal C}}
\newcommand{\calD}{{\mathcal D}}

\newcommand{\calH}{{\mathcal H}}

\newcommand{\calR}{{\mathcal R}}

\newcommand{\linspan}{\mathrm{lin\ span}}

\newcommand{\mM}{\mathfrak M}

\newcommand{\DdA}{\dom(\dA)}

\renewcommand{\Im}{\text{\rm Im}}

\def\sM{{\mathfrak M}}   \def\sN{{\mathfrak N}}

\def\bA{{\mathbb A}}   \def\dB{{\mathbb B}}   \def\dC{{\mathbb C}}

      \def\dR{{\mathbb R}}

   \def\cB{{\mathcal B}}   
      
   \def\cH{{\mathcal H}}

\def\RE{{\rm Re\,}}
\def\Ker{{\rm Ker\,}}

\def\wh{\hat}

\def\uphar{{\upharpoonright\,}}

\DeclareMathOperator{\IM}{Im}

\newcommand{\eval}[2][\right]{\relax
  \ifx#1\right\relax \left.\fi#2#1\rvert}

\begin{document}

\title[Perturbations of Donoghue classes]{Perturbations of Donoghue classes and inverse problems for L-systems}

\author{S. Belyi}
\address{Department of Mathematics\\ Troy State University\\
Troy, AL 36082, USA\\
}
\curraddr{}
\email{sbelyi@troy.edu}


\author{E. Tsekanovski\u i}
\address{Department of Mathematics\\ Niagara University, NY 14109\\ USA}
\email{tsekanov@niagara.edu}


\dedicatory{Dedicated with  great pleasure to  K.A. Makarov on his $60^{th}$ birthday.}


\begin{abstract}
 We study linear perturbations of Donoghue classes of scalar Her\-glotz-Nevanlin\-na functions by a real parameter $Q$  and their representations as impe\-dance  of  conservative L-systems.  Perturbation classes $\sM^Q$, $\sM^Q_\kappa$, $\sM^{-1,Q}_\kappa$ are introduced  and for each class the realization theorem is stated and proved. We use a new approach that leads to explicit new formulas describing  the von Neumann  parameter of the main operator of a realizing L-system and the unimodular  one corresponding to a self-adjoint extension of the symmetric part of the main operator.
The dynamics of the presented formulas as functions of $Q$ is obtained. As a result, we substantially enhance the existing  realization theorem for scalar Herglotz-Nevanlinna functions.
 In addition, we  solve the inverse problem (with uniqueness condition) of recovering the perturbed L-system  knowing the perturbation parameter $Q$  and the corresponding non-perturbed L-system. Resolvent formulas describing  the resolvents of main operators of perturbed L-systems  are presented.
A concept of a unimodular transformation as well as conditions of transformability of one perturbed L-system into another one are discussed. Examples that illustrate the obtained results are presented.

 \end{abstract}

\maketitle

\tableofcontents

\section{Introduction}\label{s1}

Let $T$ be a non-symmetric, densely defined, and closed linear operator in a Hilbert space $\cH$ such that its resolvent set $\rho(T)$ is not empty. We also assume that
$\dom(T)\cap \dom(T^*)$ is dense and that the restriction $T|_{\dom(T)\cap \dom(T^*)}$ is a closed symmetric operator $\dA$ with finite and equal deficiency indices.
Let $\calH_+\subset\calH\subset\calH_-$ be the rigged Hilbert space associated with $\dot A$ (see the next section for details).

One of the main objectives  of the current paper is the study of  \textit{L-systems} of the form
\begin{equation}
\label{col0}
 \Theta =
\left(%
\begin{array}{ccc}
  \bA    & K & J \\
   \calH_+\subset\calH\subset\calH_- &  & E \\
\end{array}%
\right),
\end{equation}
where the \textit{state-space operator} $\bA$ is a bounded linear operator from
$\calH_+$ into $\calH_-$ such that  $\dA \subset T\subset \bA$, $\dA^* \subset T^* \subset \bA$,
$K$ is a bounded linear operator from the finite-dimensional Hilbert space $E$ into $\calH_-$,  $J=J^*=J^{-1}$ is a self-adjoint isometry  on $E$
such that $\IM\bA=KJK^*$. Due to the facts that $\calH_\pm$ is dual to $\calH_\mp$ and  {that} $\bA^*$ is a bounded linear operator from $\calH_+$ into $\calH_-$, $\RE\bA=(\bA+\bA^*)/2$ and $\IM\bA=(\bA-\bA^*)/2i$ are  well defined bounded operators from $\calH_+$  into $\calH_-$.
Note that the main operator $T$ associated with L-system $\Theta$ is uniquely determined by the state-space operator $\bA$ as its restriction onto the domain $\dom(T)=\{f\in\calH_+\mid \bA f\in\calH\}$.

Recall that an operator-valued function  given by
\begin{equation*}\label{W1}
 W_\Theta(z)=I-2iK^*(\bA-zI)^{-1}KJ,\quad z\in \rho(T),
\end{equation*}
 is called the \textit{transfer function}  of an L-system $\Theta$ and
\begin{equation*}\label{real2}
 V_\Theta(z)=i[W_\Theta(z)+I]^{-1}[W_\Theta(z)-I] =K^*(\RE\bA-zI)^{-1}K,\quad z\in\rho(T)\cap\dC_{\pm},
\end{equation*}
is called the \textit{impedance function } of $\Theta$.

This article is a part of an ongoing project studying the connections between various subclasses of Herglotz-Nevanlinna functions and conservative realizations of L-systems (see \cite{AlTs2}, \cite{AlTs3}, \cite{ABT}, \cite{BMkT}, \cite{BMkT-2}, \cite{Bro}, \cite{Lv2}, \cite{LiYa}).
The class of all Herglotz-Nevanlinna functions in a finite-dimensional Hilbert space $E$, that can be realized as impedance functions of an L-system, was described in \cite{BT3},
 \cite[Definition 6.4.1]{ABT}. In particular, it was shown in \cite{BT4} (see also \cite{ABT}) that if the  representing  measure of a Herglotz-Nevanlinna function   is unbounded on $\dR$, then  this function can be realized by an L-system for any  constant term in the function integral representation. In this paper we shift our focus to the structure of the realizing L-system and substantially enhance the realization theorem for scalar Herglotz-Nevanlinna functions with unbounded representing measure  obtained in \cite{BT3}, \cite{BT4} (for details see also \cite[Chapter 6]{ABT}).

The main results of the present paper are the following. Relying on our development in \cite{BMkT}, we introduce linear perturbations of the Donoghue classes of Herglotz-Nevanlin\-na functions $\sM$, $\sM_\kappa$, $\sM^{-1}_\kappa$ by a real constant $Q$. Note that all unperturbed classes $\sM$, $\sM_\kappa$, $\sM^{-1}_\kappa$ are different by a normalization condition specific for each class.  The new perturbed classes are denoted by  $\sM^Q$, $\sM^Q_\kappa$, $\sM^{-1,Q}_\kappa$ and for each of these classes we state and prove the realization theorem where an appearing L-system of the form \eqref{col0} has a one-dimensional input-output space. Moreover, we utilize a new approach that leads to explicit new formulas describing the modulus of the von Neumann  parameter $\kappa=\kappa(Q)$ of the main operator and the unimodular one $U=U(Q)$ associated with the real part of the state-space operator $\bA$ of the realizing L-system.
For a fixed pair of $\kappa$ and $U$ we establish a uniqueness condition for  a realizing L-system construction.  It turns out that the parameters $\kappa$ and $U$ do not depend on a particular realization of an unperturbed function but only on perturbing parameter $Q$ and  normalization condition of the representing measure. Consequently, an L-system realizing the perturbed function can  be constructed based on some realization of an unperturbed function with the help of  parameters $\kappa$ and $U$.

In the end of the paper we also consider  perturbations of given  L-systems and solve the inverse problem of recovering the perturbed L-system $\Theta^Q$ from the perturbation parameter $Q$  and a given non-perturbed L-system $\Theta$.

The paper is organized as follows.

In Section \ref{s2}  we recall definition of an L-system, its components,  transfer and impedance functions, and provide the necessary background.

In Section \ref{s3} we describe the standard Donoghue class $\sM$ of Herglotz-Nevanlinna functions and explain the set of hypotheses that we will rely upon in the remainder of the paper. Here we also introduce L-systems with one-dimensional input-output that will be used throughout the text and recall the results connecting these systems to the set of hypotheses.

In Section \ref{s4} we study the most general structure of the generalized Donoghue classes of scalar  Herglotz-Nevanlinna functions. For each of these classes we state and prove realization theorems where functions are realized  as impedance functions of L-systems.

In Section \ref{s5} we consider a simplest ``perturbed" version of the Donoghue class $\sM$ . We state and prove an ``enhanced" version of the realization theorem where the explicit formulas for the moduli of von Neumann's  parameters $\kappa=\kappa(Q)$ and the unimodular ones $U=U(Q)$ are derived.

Section \ref{s6} contains the realization treatment of the perturbed classes $\sM^Q_\kappa$ and $\sM^{-1,Q}_\kappa$ producing analogues results and formulas. It turns out that the formula for the modulus of the von Neumann parameter is an even function of the perturbing real parameter $Q$ while the formula for the unimodular parameter is an ``even function up to taking conjugate", i.e., $U(-Q)=\bar U(Q)$.

Section \ref{s10} describes an alternative realization technique that results in a universal construction of a model L-system that works for all three distinct subclasses. Moreover, the symmetric operator and state space  in this model are independent from a perturbing parameter. In addition to that, we show how an arbitrary realization of an unperturbed function can be used to construct an L-system realizing its perturbed version. This section  also contains resolvent formulas for the resolvents of main operators of perturbed L-systems from all three perturbed classes.

In Section \ref{s7} we deal with ``perturbed" L-systems. We solve the inverse problem of recovering the perturbed L-system $\Theta^Q$ from the perturbation parameter $Q$  and a given non-perturbed L-system $\Theta$.

In Section \ref{s8} we discuss a concept of a unimodular transformation of one L-system into another. We show how linear perturbations can change two L-systems that were not transformable into each other to a transformable pair. This idea is further illustrated in Example 3.

Section \ref{s9} deals with construction of perturbed L-systems out of given ones whose impedance function belongs to one of the Donoghue classes $\sM$, $\sM_\kappa$, or $\sM_\kappa^{-1}$. We also describe a unimodular transformation that changes a given L-system into the one whose impedance function only differs from the impedance of original L-system by the sign of the constant term in its integral representation.

Section \ref{s11} contains a concise summary of all realization techniques and results of the paper presented as a table. We also put forward a condition (in terms of the Liv\u sic function) for a realizing L-system to be unique under the assumption of a fixed symmetric operator and normalized deficiency vectors as well as its transfer function normalization conditions. The case when two different L-systems share the same symmetric operator and deficiency basis and realize the same function is presented in Example 2.

 The paper is concluded with  examples that illustrate all the main results and concepts. Explicit constructions of a model L-system and  the one with a given state-space operator are presented in the Appendices \ref{A2} and \ref{A1} for convenience of the reader.

\section{Preliminaries}\label{s2}

For a pair of Hilbert spaces $\calH_1$, $\calH_2$ we denote by
$[\calH_1,\calH_2]$ the set of all bounded linear operators from
$\calH_1$ to $\calH_2$. Let $\dA$ be a closed, densely defined,
symmetric operator in a Hilbert space $\calH$ with inner product
$(f,g),f,g\in\calH$. Any non-symmetric operator $T$ in $\cH$ such that
\[
\dA\subset T\subset\dA^*
\]
is called a \textit{quasi-self-adjoint extension} of $\dA$.

 Consider the rigged Hilbert space (see \cite{Ber}, \cite{BT3})
$\calH_+\subset\calH\subset\calH_- ,$ where $\calH_+ =\dom(\dA^*)$ and
\begin{equation}\label{108}
(f,g)_+ =(f,g)+(\dA^* f, \dA^*g),\;\;f,g \in \dom(A^*).
\end{equation}
Let $\calR$ be the \textit{\textrm{Riesz-Berezansky   operator}} $\calR$ (see \cite{Ber}, \cite{BT3}) which maps $\mathcal H_-$ onto $\mathcal H_+$ such
 that   $(f,g)=(f,\calR g)_+$, ($\forall f\in\calH_+$, $g\in\calH_-$) and
 $\|\calR g\|_+=\| g\|_-$.
 Note that
identifying the space conjugate to $\calH_\pm$ with $\calH_\mp$, we
get that if $\bA\in[\calH_+,\calH_-]$, then
$\bA^*\in[\calH_+,\calH_-].$
An operator $\bA\in[\calH_+,\calH_-]$ is called a \textit{self-adjoint
bi-extension} of a symmetric operator $\dA$ if $\bA=\bA^*$ and $\bA
\supset \dA$.
Let $\bA$ be a self-adjoint
bi-extension of $\dA$ and let the operator $\hat A$ in $\cH$ be defined as follows:
\[
\dom(\hat A)=\{f\in\cH_+:\hat A f\in\cH\}, \quad \hat A=\bA\uphar\dom(\hat A).
\]
The operator $\hat A$ is called the \textit{quasi-kernel} of a self-adjoint bi-extension $\bA$ (see \cite{TSh1}, \cite[Section 2.1]{ABT}).
According to the von Neumann Theorem (see \cite[Theorem 1.3.1]{ABT}) the domain of $\wh A$, a self-adjoint extension of $\dA$,  can be expressed as
\begin{equation}\label{DOMHAT}
\dom(\hat A)=\dom(\dA)\oplus(I+U)\sN_{i},
\end{equation}
where von Neumann's parameter $U$ is an isometric  operator from $\sN_i$ into $\sN_{-i}$ in $\calH$
 and $$\sN_{\pm i}=\Ker (\dA^*\mp i I)$$ are the deficiency subspaces of $\dA$.
 A self-adjoint bi-extension $\bA$ of a symmetric operator $\dA$ is called \textit{t-self-adjoint} (see \cite[Definition 3.3.5]{ABT}) if its quasi-kernel $\hat A$ is a
self-adjoint operator in $\calH$.
An operator $\bA\in[\calH_+,\calH_-]$  is called a \textit{quasi-self-adjoint bi-extension} of an operator $T$ if
$\bA\supset T\supset \dA$ and $\bA^*\supset T^*\supset\dA.$  We will be mostly interested in the following type of quasi-self-adjoint bi-extensions.
\begin{definition}[\cite{ABT}]\label{star_ext}
Let $T$ be a quasi-self-adjoint extension of $\dA$ with a nonempty resolvent set $\rho(T)$. A quasi-self-adjoint bi-extension $\bA$ of an operator $T$ is called a \textit{($*$)-extension }of $T$ if $\RE\bA$ is a
t-self-adjoint bi-extension of $\dA$.
\end{definition}
In what follows we assume that $\dA$ has equal and finite deficiency indices and will say that a quasi-self-adjoint extension $T$ of $\dA$ belongs to the
\textit{class $\Lambda(\dA)$} if $\rho(T)\ne\emptyset$, $\dom(\dA)=\dom(T)\cap\dom(T^*)$, and hence  $T$ admits $(*)$-extensions. The description of
all $(*)$-extensions via Riesz-Berezansky   operator $\calR$ can be found in \cite[Section 4.3]{ABT}.

\begin{definition}\label{d-2}
A system of equations
\[
\left\{   \begin{array}{l}
          (\bA-zI)x=KJ\varphi_-  \\
          \varphi_+=\varphi_- -2iK^* x
          \end{array}
\right.,
\]
 or an
array
\begin{equation}\label{e6-3-2}
\Theta= \begin{pmatrix} \bA&K&\ J\cr \calH_+ \subset \calH \subset
\calH_-& &E\cr \end{pmatrix}
\end{equation}
 is called an \textbf{{L-system}}   if:
\begin{enumerate}
\item[(1)] {$\mathbb  A$ is a   ($\ast $)-extension of an
operator $T$ of the class $\Lambda(\dA)$};
\item[(2)] {$J=J^\ast =J^{-1}\in [E,E],\quad \dim E < \infty $};
\item[(3)] $\IM\bA= KJK^*$, where $K\in [E,\calH_-]$, $K^*\in [\calH_+,E]$, and
$\ran(K)=\ran (\IM\bA).$
\end{enumerate}
\end{definition}
In the definition above   $\varphi_- \in E$ stands for an input vector, $\varphi_+ \in E$ is an output vector, and $x$ is a state space vector in $\calH_+$ ($\subset\calH$).
The operator $\bA$  is called the \textit{state-space operator} of the system $\Theta$, $T$ is the \textit{main operator},  $J$ is the \textit{direction operator}, and $K$ is the  \textit{channel operator}. A system $\Theta$ in \eqref{e6-3-2} is called \textit{minimal} if the operator $\dA$ is a \textit{prime} operator in $\calH$, i.e., there exists no non-trivial reducing
invariant subspace of $\calH$ on which it induces a self-adjoint operator. 
An L-system $\Theta$ defined above is \textit{conservative} in the sense explained in \cite[Section 6.3]{ABT}, \cite[Section 1.1]{Zol} and is an evolved version of linear conservative systems considered in \cite{Bro}, \cite{Lv2}, and \cite{LiYa}.

We  associate with an L-system $\Theta$ the operator-valued function
\begin{equation}\label{e6-3-3}
W_\Theta (z)=I-2iK^\ast (\mathbb  A-zI)^{-1}KJ,\quad z\in \rho (T),
\end{equation}
which is called the \textbf{transfer  function} of the L-system $\Theta$. Clearly, it follows from the system of equations in Definition \ref{d-2} that $\varphi_+=W_\Theta (z)\varphi_-$. We also consider an operator-valued function
\begin{equation}\label{e6-3-5}
V_\Theta (z) = K^\ast (\RE\bA - zI)^{-1} K, \quad z\in\rho(\hat A).
\end{equation}
It was shown in \cite{BT3}, \cite[Section 6.3]{ABT} that both \eqref{e6-3-3} and \eqref{e6-3-5} are well defined. The transfer operator-function $W_\Theta (z)$ of the system
$ \Theta $ and  operator-function $V_\Theta (z)$ of the form (\ref{e6-3-5}) are connected by the following relations valid for $\IM z\ne0$, $z\in\rho(T)$,
\begin{equation}\label{e6-3-6}
\begin{aligned}
V_\Theta (z) &= i [W_\Theta (z) + I]^{-1} [W_\Theta (z) - I] J,\\
W_\Theta(z)&=(I+iV_\Theta(z)J)^{-1}(I-iV_\Theta(z)J).
\end{aligned}
\end{equation}
Function $V_\Theta(z)$ defined by \eqref{e6-3-5} is called the \textbf{impedance function} of an L-system $ \Theta $ of the form (\ref{e6-3-2}). The class of all Herglotz-Nevanlinna functions in a finite-dimensional Hilbert space $E$, that can be realized as impedance functions of an L-system, was described in \cite{BT3}, \cite[Definition 6.4.1]{ABT}.

Two minimal L-systems
$$
\Theta_j= \begin{pmatrix} \bA_j&K_j&J\cr \calH_{+j} \subset\calH_j \subset \calH_{-j}& &E\cr
\end{pmatrix},\quad j=1,2,
$$
are called \textbf{bi-unitarily equivalent} \cite[Section 6.6]{ABT} if there exists a triplet of operators $(U_+, U, U_-)$ that isometrically maps the triplet $\calH_{+1}\subset\calH_1\subset\calH_{-1}$ onto the triplet $\calH_{+2}\subset\calH_2\subset\calH_{-2}$ respectively such that $U_+=U|_{\calH_{+1}}$ is an isometry from $\calH_{+1}$ onto $\calH_{+2}$, $U_-=(U_+^*)^{-1}$ is an isometry from $\calH_{-1}$ onto $\calH_{-2}$, and
\begin{equation}\label{167}
UT_1=T_2U,  \quad U_-\bA_1=\bA_2 U_+,\quad U_-K_1=K_2.
\end{equation}
It is shown in \cite[Theorem 6.6.10]{ABT} and its corollary that if the transfer functions $W_{\Theta_1}(z)$ and $W_{\Theta_2}(z)$ of two minimal L-systems $\Theta_1$ and $\Theta_2$ are the same on
$ z\in\rho(T_1)\cap\rho(T_2)\cap \dC_{\pm}\ne\emptyset$, then $\Theta_1$ and $\Theta_2$ are bi-unitarily equivalent and $U$ is uniquely determined.

\section{Donoghue classes and L-systems with one-dimensional input-output}\label{s3}

Suppose that $\dA$ is a closed prime 
densely defined symmetric operator with deficiency indices $(1,1)$. Assume also that $T \ne T^*$ is a  maximal dissipative extension of $\dot A$,
$$\Im(T f,f)\ge 0, \quad f\in \dom(T ).$$
Since $\dot A$ is symmetric, its dissipative extension $T$ is automatically quasi-self-adjoint \cite{ABT},
that  is,
$$
\dot A \subset T \subset \dA^*,
$$
and hence, (see \cite{BMkT})
\begin{equation}\label{parpar}
g_+-\kappa g_-\in \dom
 (T )\quad \text{for some }
|\kappa|<1.
\end{equation}
{Throughout this paper  $\kappa$ will be referred to as the \textbf{ von Neumann  parameter} of operator $T$.}

Recall that  Donoghue \cite{D}  introduced a concept of the Weyl-Titchmarsh function $M(\dot A, A)$ associated with a pair $(\dot A, A)$ by
$$M(\dot A, A)(z)=
((Az+I)(A-zI)^{-1}g_+,g_+), \quad z\in \bbC_+,
$$
$$g_+\in \Ker( \dA^*-iI),\quad \|g_+\|=1,$$
where $\dot A $ is a symmetric operator with deficiency indices $(1,1)$, and $A$ is its self-adjoint extension.
Following our earlier developments in \cite{MT10}, \cite{BMkT} we denote by $\mM$ the \textbf{Donoghue class} of all analytic mappings $M$ from $\bbC_+$ into itself  that admits the representation
 \begin{equation}\label{hernev-0}
M(z)=\int_\bbR \left
(\frac{1}{\lambda-z}-\frac{\lambda}{1+\lambda^2}\right )
d\mu,
\end{equation}
where $\mu$ is an  infinite Borel measure   and
$$
\int_\bbR\frac{d\mu(\lambda)}{1+\lambda^2}=1\,,\quad\text{equivalently,}\quad M(i)=i.
$$
It is known  \cite{D},  \cite{GT}, \cite{GMT97}, \cite {MT-S} that $M\in \mM$ if and only if $M$
can be realized  as the Weyl-Titchmarsh function $M(\dot A, A)$ associated with a pair $(\dot A, A)$.

\begin{hypothesis}\label{setup} Suppose that $\whA \ne\whA^*$  is  a maximal dissipative extension of  a symmetric operator $\dot A$  with deficiency indices $(1,1)$. Assume, in addition, that $A$ is a  self-adjoint extension of $\dot A$. Suppose,  that the deficiency elements $g_\pm\in \Ker (\dA^*\mp iI)$ are normalized, $\|g_\pm\|=1$, and chosen in such a way that
\begin{equation}\label{ddoomm14}g_+- g_-\in \dom ( A)\,\,\,\text{and}\,\,\,
g_+-\kappa g_-\in \dom (\whA )\,\,\,\text{for some }
\,\,\,|\kappa|<1.
\end{equation}
\end{hypothesis}
It is known \cite{MT-S} that if $\kappa=0$, then  quasi-self-adjoint extension $\whA $ coincides with the restriction of the adjoint operator $\dot A^*$ on
$$
\dom(\whA )=\dom(\dot A)\dot + \Ker (\dA^*-iI).
$$
Similar to Hypothesis \ref{setup} we will also consider the ``anti-Hypothesis" as follows.
\begin{hypothesis}\label{setup-1} Suppose
that $\whA \ne\whA^*$  is  a maximal
dissipative extension of  a symmetric operator $\dot A$
 with deficiency indices $(1,1)$. Assume, in addition, that
$A$ is a  self-adjoint extension of $\dot A$. Suppose,  that the
deficiency elements $g_\pm\in \Ker (\dA^*\mp iI)$ are
normalized, $\|g_\pm\|=1$, and chosen in such a way that
\begin{equation}\label{ddoomm14-1}g_++ g_-\in \dom ( A)\,\,\,\text{and}\,\,\,
g_+-\kappa g_-\in \dom (\whA )\,\,\,\text{for some }
\,\,\,|\kappa|<1.
\end{equation}
\end{hypothesis}
\noindent
\begin{remark}\label{r-12}
Without loss of generality, in what follows we assume that $\kappa$ is real and $0\le\kappa<1$: if $\kappa=|\kappa|e^{i\theta}$,
change (the basis) $g_-$ to $e^{i\theta}g_-$ in the deficiency subspace  $\Ker (\dA^*+ i I)$.
\end{remark}
This remark means the following: let
\begin{equation}\label{e-62}
\Theta= \begin{pmatrix} \bA&K&\ 1\cr \calH_+ \subset \calH \subset
\calH_-& &\dC\cr \end{pmatrix}
\end{equation}
be a minimal L-system  with one-dimensional input-output space $\dC$. 
If the main operator $T$ of $\Theta$ is parameterized with a \textit{complex} von Neumann's parameter $\kappa$ that corresponds to a chosen normalized pair of deficiency vectors $g_+$ and $g_-$, then we can change the deficiency basis as described in Remark \ref{r-12} and represent $T$ using real value of $|\kappa|$ with respect to the new deficiency basis. This procedure will change the parameter $U$ of the quasi-kernel $\hat A$ of $\RE\bA$ in  \eqref{DOMHAT} and ultimately the way $\bA$ is described (see Appendix \ref{A2}).
 \textit{Thus, for the remainder of this paper (unless otherwise is specified) we will consider L-systems  \eqref{e-62} such that $\kappa$ is real and $0\le\kappa<1$.}

\begin{definition}
We say that an L-system $\Theta$ of the form \eqref{e-62}  \textit{satisfies Hypothesis} \ref{setup} (or \ref{setup-1}) if its main operator $T$ and the quasi-kernel $\hat A$ of $\RE\bA$ satisfy the conditions of Hypothesis \ref{setup} (or \ref{setup-1}) for a fixed set of deficiency vectors of the symmetric operator $\dA$.
\end{definition}

Let $\Theta$ be a minimal L-system of the form \eqref{e-62} that satisfies the conditions of Hypothesis \ref{setup}. It is shown in \cite{BMkT} that  the impedance function $V_\Theta(z)$ can be represented as
\begin{equation}\label{e-imp-m}
    V_{\Theta}(z)=\left(\frac{1-\kappa}{1+\kappa}\right)V_{\Theta_0}(z),
\end{equation}
where $V_{\Theta_0}(z)$ is  the impedance function of an L-system $\Theta_0$ with the same set of conditions but with $\kappa_0=0$, where $\kappa_0$ is the von Neumann parameter of the main operator $T_0$ of $\Theta_0$.

Let $\Theta_1$ and $\Theta_2$ be two minimal L-system of the form \eqref{e-62}  whose components satisfy the conditions of Hypothesis \ref{setup} and Hypothesis \ref{setup-1}, respectively. Then it was proved in \cite[Lemma 5.1]{BMkT-2} that the impedance functions $V_{\Theta_1}(z)$ and $V_{\Theta_2}(z)$ admit the integral representation
\begin{equation}\label{e-60-nu-1}
V_{\Theta_{k}}(z)=\int_\bbR \left(\frac{1}{t-z}-\frac{t}{1+t^2}\right )d\mu_{k}(t),\quad k=1,2.
\end{equation}

Now let us consider a minimal L-system $\Theta$ of the form \eqref{e-62} that satisfies  Hypothesis \ref{setup}. Let also
\begin{equation}\label{e-62-alpha}
\Theta_\alpha= \begin{pmatrix} \bA_\alpha&K_\alpha&\ 1\cr \calH_+ \subset \calH \subset
\calH_-& &\dC\cr \end{pmatrix},\quad \alpha\in[0,\pi),
\end{equation}
 be a one parametric family of L-systems such that
 \begin{equation}\label{e-63-alpha}
    W_{\Theta_\alpha}(z)=W_\Theta(z)\cdot (-e^{2i\alpha}),\quad \alpha\in[0,\pi).
 \end{equation}
The existence and structure of $\Theta_\alpha$ were described in details in \cite[Section 8.3]{ABT}. In particular, it was shown that $\Theta$ and $\Theta_\alpha$ share the same main operator $T$ and that
\begin{equation}\label{e-64-alpha}
    V_{\Theta_\alpha}(z)=\frac{\cos\alpha+(\sin\alpha) V_\Theta(z)}{\sin\alpha-(\cos\alpha) V_\Theta(z)}.
\end{equation}

Let $\Theta$ be a minimal L-system $\Theta$ of the form \eqref{e-62} that satisfies Hypothesis \ref{setup}. Let also $\Theta_{\alpha}$ be a one parametric family of L-systems given by \eqref{e-62-alpha}-\eqref{e-63-alpha}.
It was shown in \cite[Theorem 5.2]{BMkT-2} that in this case the impedance function $V_{\Theta_{\alpha}}(z)$ has an integral representation
$$
V_{\Theta_{\alpha}}(z)=\int_\bbR \left(\frac{1}{t-z}-\frac{t}{1+t^2}\right )d\mu_{\alpha}(t)
$$
if and only if $\alpha=0$ or $\alpha=\pi/2$.

The next result describes the relationship between two L-systems of the form \eqref{e-62} that comply with different Hypotheses. Let
\begin{equation}\label{e-62-1}
\Theta_1= \begin{pmatrix} \bA_1&K_1&\ 1\cr \calH_+ \subset \calH \subset
\calH_-& &\dC\cr \end{pmatrix}
\end{equation}
be a minimal L-system whose main operator $T$ and the quasi-kernel $\hat A_1$ of $\RE\bA_1$ satisfy the conditions of Hypothesis \ref{setup} and let
\begin{equation}\label{e-62-2}
\Theta_2= \begin{pmatrix} \bA_2&K_2&\ 1\cr \calH_+ \subset \calH \subset
\calH_-& &\dC\cr \end{pmatrix}
\end{equation}
be another minimal L-system with the same  operators $\dA$ and $T$ as $\Theta_1$ but with the quasi-kernel $\hat A_2$ of $\RE\bA_2$ that satisfies the conditions of Hypothesis \ref{setup-1}. It was shown in \cite[Theorem 5.3]{BMkT-2} that
\begin{equation}\label{e-55-1}
    W_{\Theta_1}(z)=-W_{\Theta_2}(z),\quad z\in\dC_+\cap\rho(T),
\end{equation}
and
\begin{equation}\label{e-56-1}
    V_{\Theta_1}(z)=-\frac{1}{V_{\Theta_2}(z)},\quad z\in\dC_+\cap\rho(T).
\end{equation}


\section{Realizations of the class $\sN$}\label{s4}

In this section we are going to study classes of scalar  Herglotz-Nevanlinna functions and their realizations as impedance functions of L-systems of the form  \eqref{e-62}.
It follows from  \cite{ABT}, \cite{BT4} that all scalar realizable by one-dimensional input-output L-systems of the form \eqref{e-62} Herglotz-Nevanlinna functions $V(z)$ admit the integral representation
 \begin{equation}\label{hernev-real}
V(z)= Q+\int_\bbR\left (\frac{1}{\lambda-z}-\frac{\lambda}{1+\lambda^2}\right )d\mu,
\end{equation}
where $\mu$ is an infinite Borel measure with
$$
\int_\bbR\frac{d\mu(\lambda)}{1+\lambda^2}<\infty,
$$
 and  $Q=\bar Q. $

Now let us focus on the following class  of scalar  Herglotz-Nevanlinna functions. Let $\sN$ be a class of all Herglotz-Nevanlinna functions $M(z)$ that admit the representation
 \begin{equation}\label{hernev}
M(z)=\int_\bbR \left
(\frac{1}{\lambda-z}-\frac{\lambda}{1+\lambda^2}\right )
d\mu,
\end{equation}
where $\mu$ is an  infinite Borel measure.
Following our definition  in Section \ref{s3}  we note that    the {Donoghue class}  $\mM$ consists  of functions $M\in\sN$  such that
\begin{equation}\label{e-42-int-don}
\int_\bbR\frac{d\mu(\lambda)}{1+\lambda^2}=1,
\end{equation}
in integral representation \eqref{hernev} or, {equivalently,} $M(i)=i$.
Furthermore, we say (see \cite{BMkT}) that a function $M\in\sN$  belongs to the \textbf{generalized Donoghue class} $\sM_\kappa$, ($0\le\kappa<1$) if  in the representation \eqref{hernev}
\begin{equation}\label{e-38-kap}
\int_\bbR\frac{d\mu(\lambda)}{1+\lambda^2}=\frac{1-\kappa}{1+\kappa}\,,\quad\text{equivalently,}\quad M(i)=i\,\frac{1-\kappa}{1+\kappa}.
\end{equation}
Similarly  (see \cite{BMkT-2}), a function $M\in\sN$  belongs to the  \textbf{generalized Donoghue class} $\sM_\kappa^{-1}$ if in the representation \eqref{hernev}
\begin{equation}\label{e-39-kap}
\int_\bbR\frac{d\mu(\lambda)}{1+\lambda^2}=\frac{1+\kappa}{1-\kappa}\,,\quad\text{equivalently,}\quad M(i)=i\,\frac{1+\kappa}{1-\kappa}.
\end{equation}
Clearly, $\sM_0=\sM_0^{-1}=\sM$.

Now let $M$ be an arbitrary function from $\sN$ with a normalization condition
\begin{equation}\label{e-66-L}
\int_\bbR\frac{d\mu(\lambda)}{1+\lambda^2}=a,
\end{equation}
for some $a>0$. It is easy to see that $M\in\sM$ if and only if $a=1$. Also, if $a<1$, then $M\in \sM_\kappa$ with
\begin{equation}\label{e-45-kappa-1}
\kappa=\frac{1-a}{1+a},
\end{equation}
and if $a>1$, then $M\in \sM_\kappa^{-1}$ with
\begin{equation}\label{e-45-kappa-2}
\kappa=\frac{a-1}{1+a}.
 \end{equation}
This observation allows us to partition our class $\sN$ in three distinct subclasses
$$
\sN=\sM^-\cup\sM\cup\sM^+,
$$
where
\begin{equation}\label{e-66-sub}
   \begin{aligned}
    \sM^+&=\{M\in\sN\mid a<1\},\\
    \sM^-&=\{M\in\sN\mid a>1\},
    \end{aligned}
\end{equation}
and $a$ is defined by \eqref{e-66-L} for a particular representation \eqref{hernev} of a function $M$. Clearly,
\begin{equation}\label{r-unions}
    \sM^+=\bigcup_{0<\kappa<1}\sM_\kappa \quad\textrm{ and }\quad \sM^-=\bigcup_{0<\kappa<1}\sM_\kappa^{-1}.
\end{equation}
Below we will study L-system realizations of the classes $\sM^+$, $\sM$, and $\sM^-$.

The following theorem is a necessary and sufficient condition for the impedance function of an L-system under consideration (with $\kappa\ne0$) not to have a constant term $Q$ in integral representation \eqref{hernev-real}.
\begin{theorem}\label{t-10-new}
Let $\Theta$ of the form \eqref{e-62} be a minimal L-system that  satisfies the conditions of either Hypothesis \ref{setup} or Hypothesis \ref{setup-1} and such that its main operator $T$ has the von Neumann parameter $\kappa$ where $0<\kappa<1$.   Then its impedance function $V_\Theta(z)$ admits integral representation \eqref{hernev}.

Conversely, let a function $V(z)$ have integral representation \eqref{hernev}. Then it can be realized by an L-system $\Theta$ satisfying either Hypothesis \ref{setup} or Hypothesis \ref{setup-1}.
\end{theorem}
\begin{proof}
In one direction the statement of our Theorem follows directly from \cite[Lemma 5.1]{BMkT-2}.

If $V_\Theta(z)$ admits integral representation \eqref{hernev}, then it belongs to the class $\sN$ and hence falls into one of the disjoint subclasses $\sM$, $\sM_\kappa$, or $\sM_\kappa^{-1}$ for some real $\kappa$ such that $0\le\kappa<1$. If $V(z)\in\sM$ or $V(z)\in\sM_\kappa$, then it can be realized by an L-system satisfying  Hypothesis \ref{setup}. Similarly, if $V(z)\in\sM_\kappa^{-1}$, then it can be realized by an L-system satisfying  Hypothesis \ref{setup-1}. In the first case we can use model realizations $\Theta_1$ of the form \eqref{e-59'} and in the second case model realizations $\Theta_2$ of the form \eqref{e-59''} provided by Appendix \ref{A1}.
\end{proof}
In the case when $\kappa=0$ in the first part of of Theorem \ref{t-10-new}, the impedance function $V_\Theta(z)$ belongs to the Donoghue class $\sM$ regardless of the fact whether or not it meets conditions of either of Hypotheses (see \cite[Theorem 11]{BMkT}). Consequently, it admits integral representation \eqref{hernev}.
\begin{theorem}\label{t-12-new}
Let $\Theta$  be a minimal L-system of the form \eqref{e-62} that  satisfies the conditions of  Hypothesis \ref{setup} whose impedance function $V_\Theta(z)$ has integral representation \eqref{hernev-real} with normalization parameter $a$ defined by \eqref{e-66-L}.
Then $V_\Theta(z)\in\sM^+$.

Conversely, let $V(z)\in\sM^+$. Then it can be realized by an L-system $\Theta$ satisfying  Hypothesis \ref{setup}.
\end{theorem}
\begin{proof}
Let our L-system $\Theta$ satisfy the conditions of Hypothesis \ref{setup}. Then  according to the same \cite[Theorem 12]{BMkT}, $V_\Theta(z)$ belongs to the  class $\sM_\kappa$. Taking into account that in this case the value of $a$ in the integral representation \eqref{hernev} of $V_\Theta(z)$ is found by inverting \eqref{e-45-kappa-1}, that is
$$a=\frac{1-\kappa}{1+\kappa}.$$
Thus, $a<1$  and $V_\Theta(z)\in\sM^+$.

Let $V(z)\in\sM^+$. Then there is a value of $0<a<1$  defined by \eqref{e-66-L} for the integral representation \eqref{hernev} of $V(z)$. Consequently, there is a value of $\kappa$ found via \eqref{e-45-kappa-1} such that $V(z)$ belongs to the generalized Donoghue class $\sM_\kappa$. Using Appendix \ref{A2} we construct a model L-system $\Theta_1$ of the form \eqref{e-59'} that realizes $V(z)$ and satisfies Hypothesis \ref{setup}.
\end{proof}
A similar result holds true for the class $\sM^-$.
\begin{theorem}\label{t-13-new}
Let $\Theta$  be a minimal L-system of the form \eqref{e-62} that  satisfies the conditions of Hypothesis \ref{setup-1} whose impedance function $V_\Theta(z)$ has integral representation \eqref{hernev-real} with normalization parameter $a$ defined by \eqref{e-66-L}.
Then $V_\Theta(z)\in\sM^-$.

Conversely, let $V(z)\in\sM^-$. Then it can be realized by an L-system $\Theta$ satisfying  Hypothesis \ref{setup-1}.
\end{theorem}
\begin{proof}
The proof is completely similar to the one of Theorem \ref{t-12-new}. The only difference is that $a$ and $\kappa$ are related via \eqref{e-45-kappa-2} and we are referring to a model L-system $\Theta_2$ of the form \eqref{e-59''} in Appendix \ref{A1}.
\end{proof}

Below we will state and prove some inverse realization results.


\begin{theorem}\label{t-14}
Let  $V(z)$ belong to the generalized Donoghue class $\sM^+$. Then $V(z)$ can be realized as the impedance function $V_{\Theta}(z)$ of a minimal L-system $\Theta$  of the form \eqref{e-62} with  the triple $(\dot A, T, \hat A)$ that satisfies Hypothesis \ref{setup} with $A=\hat A$, the quasi-kernel of $\RE\bA$, and the von Neumann parameter $\kappa$ of $T$ given by \eqref{e-45-kappa-1}. Moreover,
 \begin{equation}\label{e-73-new}
    V(z)=V_{\Theta}(z)= \frac{1-\kappa}{1+\kappa}\,M(\dA, \hat A)(z),\quad z\in\dC_+,
 \end{equation}
 where $M(\dA, \hat A)(z)$ is the Weyl-Titchmarsh function associated with the pair $(\dA, \hat A)$.
\end{theorem}
\begin{proof}
Since $V(z)\in\sM^+$, then there is a value of $0<a<1$  defined by \eqref{e-66-L} for the integral representation \eqref{hernev} of $V(z)$. Consequently, there is a value of $\kappa$ found via \eqref{e-45-kappa-1} such that $V(z)$ belongs to the generalized Donoghue class $\sM_\kappa$. Then we can apply to \cite[Theorem 14]{BMkT} that will guarantee the existence of a minimal L-system $\Theta_\kappa$ that satisfies the conditions of Hypothesis \ref{setup} and such that $V(z)=V_{\Theta_\kappa}(z)$ satisfies \eqref{e-73-new}. Moreover, as it was shown in the proof of \cite[Theorem 14]{BMkT}, the  realizing L-system $\Theta_\kappa$ can be chosen as a minimal model L-system $\Theta_1$ of the form \eqref{e-59'} described in details in Appendix \ref{A1}. Note, that in this case the entire construction of $\Theta_1$ is based upon the measure $\mu$ in the integral representation \eqref{hernev} of the Weyl-Titchmarsh function $M(\dA, \hat A)(z)$ from \eqref{e-73-new} associated with the pair $(\dA, \hat A)=(\dot\cB,\cB)$.
\end{proof}
The following corollary immediately follows from Theorem \ref{t-14} when we set $\kappa=0$.
\begin{corollary}[\cite{BMkT}]\label{c-14}
Let  $V(z)$ belong to the Donoghue class $\sM$. Then $V(z)$ can be realized as the impedance function $V_{\Theta_0}(z)$ of a minimal L-system $\Theta_0$
 of the form \eqref{e-62} with  the triple $(\dot A, T, \hat A)$ that satisfies Hypothesis \ref{setup} with $A=\hat A$, the quasi-kernel of $\RE\bA$ and  $\kappa_0=0$. Moreover,
 \begin{equation}\label{e-73-new-0}
    V(z)=V_{\Theta_0}(z)= \,M(\dA, \hat A)(z),\quad z\in\dC_+,
 \end{equation}
 where $M(\dA, \hat A)(z)$ is the Weyl-Titchmarsh function associated with the pair $(\dA, \hat A)$.
\end{corollary}

A similar to Theorem \ref{t-14} result takes place for the class $\sM^-$.
\begin{theorem}\label{t-14-new}
Let  $V(z)$ belong to the generalized Donoghue class $\sM^-$. Then $V(z)$ can be realized as the impedance function $V_{\Theta}(z)$ of a minimal L-system $\Theta$
 of the form \eqref{e-62} with  the triple $(\dot A, T, \hat A)$ that satisfies Hypothesis \ref{setup-1} with $A=\hat A$, the quasi-kernel of $\RE\bA$, and the von Neumann parameter $\kappa$ of $T$  given by \eqref{e-45-kappa-2}. Moreover,
 \begin{equation}\label{e-73-new-2}
    V(z)=V_{\Theta}(z)= \frac{1+\kappa}{1-\kappa}\,M(\dA, \hat A)(z),\quad z\in\dC_+,
 \end{equation}
 where $M(\dA, \hat A)(z)$ is the Weyl-Titchmarsh function associated with the pair $(\dA, \hat A)$.
 \end{theorem}
\begin{proof}
The proof is completely similar to the proof of Theorem \ref{t-14} except for the fact that it relies on \cite[Theorem 5.6]{BMkT-2}. Again, we note that as it was shown in the proof of \cite[Theorem 5.6]{BMkT-2}, the  realizing L-system $\Theta_\kappa$ can be chosen as a minimal model L-system $\Theta_2$ of the form \eqref{e-59''} described in details in Appendix \ref{A1}.  In this case the entire construction of $\Theta_2$ is based upon the measure $\mu$ in the integral representation \eqref{hernev} of the Weyl-Titchmarsh function $M(\dA, \hat A)(z)$ from \eqref{e-73-new-2} associated with the pair $(\dA, \hat A)=(\dot\cB,\cB_1)$.
\end{proof}


\section{Perturbations of the class $\sM$}\label{s5}

In this section we consider a  ``perturbed" version of a subclass $\sM$ of the class $\sN$ discussed in Section \ref{s5}. First, we are going to introduce some new notations. Let $Q\ne0$ and $\sN^Q$ consist of all  Herglotz-Nevanlinna functions $V(z)$ admitting the integral representation \eqref{hernev-real}. Clearly, comparing \eqref{hernev-real} and \eqref{hernev} justifies the name and symbolism for the class $\sN^Q$. Similarly, we introduce perturbed classes $\sM^Q$, $\sM^Q_\kappa$, and $\sM^{-1,Q}_\kappa$ if normalization conditions \eqref{e-42-int-don}, \eqref{e-38-kap}, and \eqref{e-39-kap}, respectively, hold  on measure $\mu$ in \eqref{hernev-real}. The ``perturbed" versions of the classes $\sM^-$ and $\sM^+$ are $\sM^{-Q}$ and $\sM^{+Q}$, respectively. We begin with a lemma.
\begin{lemma}\label{l-12}
Let $V_0(z)$ belong to the class $\sN$ and let $V(z)=Q+V_0(z)$ be its perturbation with $Q\ne0$. Then 
\begin{equation}\label{e-35-lemma}
    V_{\alpha}(z)=\frac{\cos\alpha+(\sin\alpha)V(z)}{\sin\alpha-(\cos\alpha)V(z)},\quad z\in\dC_+,
 \end{equation}
is a Herglotz-Nevanlinna function with an infinite Borel measure in  representation \eqref{hernev-real}  for any $\alpha\in[0,\pi)$.
\end{lemma}
\begin{proof}
The function $V_0(z)\in\sN$ and thus has infinite Borel measure by the definition of the class $\sN$. Function $V(z)$ is only different from $V_0(z)$ by a constant and hence has the Borel measure as $V_0(z)$. We need to show that $V_{\alpha}(z)$ also possesses the desired property.
Let $\bar V_{\alpha}(z)$ be the complex conjugate of the function $V_{\alpha}(z)$ defined by \eqref{e-35-lemma}. Then
$$
\begin{aligned}
V_{\alpha}(z)-\bar V_{\alpha}(z)&=\frac{\cos\alpha+(\sin\alpha)V(z)}{\sin\alpha-(\cos\alpha)V(z)}-\frac{\cos\alpha+(\sin\alpha)\bar V(z)}{\sin\alpha-(\cos\alpha)\bar V(z)}\\
&=\frac{(\cos\alpha+\sin\alpha V(z))(\sin\alpha-\cos\alpha\bar V(z))}{|\sin\alpha-(\cos\alpha)V(z) |^2}\\
&\qquad-\frac{(\sin\alpha-\cos\alpha V(z))(\cos\alpha+\sin\alpha\bar V(z))}{|\sin\alpha-(\cos\alpha)V(z) |^2}\\
&=\frac{\sin^2\alpha(V(z)-\bar V(z))+\cos^2\alpha(V(z)-\bar V(z))}{|\sin\alpha-(\cos\alpha)V(z) |^2}\\&=\frac{V(z)-\bar V(z)}{|\sin\alpha-(\cos\alpha)V(z) |^2}.
\end{aligned}
$$
Consequently,
$$
\IM  V_{\alpha}(z)=\frac{\IM V(z)}{|\sin\alpha-(\cos\alpha)V(z) |^2}.
$$
On the other hand, since $Q=\bar Q$, then $\IM V(z)=\IM V_0(z)$ and hence
\begin{equation}\label{e-36-lemma}
    \IM  V_{\alpha}(z)=\frac{\IM V_0(z)}{|\sin\alpha-(\cos\alpha)(Q+V_0(z)) |^2}.
\end{equation}
Formula \eqref{e-36-lemma} and the fact that $V_0(z)\in\sN$ imply that $V_{\alpha}(z)$ is a Herglotz-Nevanlinna function. It is shown in \cite{KK74} that if $\mu_0$ is a Borel measure in integral representation \eqref{hernev-real} of the function $V_0(z)$, then
$$
\lim_{\eta\to+\infty}\eta\IM V_0(i\eta)=\int_{\mathbb R}d\mu_0(t).
$$
We are going to show that
$$
\lim_{\eta\to+\infty}\eta\IM V_\alpha(i\eta)=\int_{\mathbb R}d\mu_\alpha(t)=\infty,
$$
where $\mu_\alpha$ is the Borel measure in integral representation \eqref{hernev-real} of the function $V_\alpha(z)$.
Let
\begin{equation}\label{e-38-new}
V_{0,\alpha}(z)=\frac{\cos\alpha+(\sin\alpha)V_0(z)}{\sin\alpha-(\cos\alpha)V_0(z)},\quad z\in\dC_+.
\end{equation}
Since $V_0(z)\in\sN$, then $V_0(z)$ falls into one of the disjoint classes $\sM^+$, $\sM$, or $\sM^-$ and hence is realizable by a minimal L-system  of the form \eqref{e-62} (see Theorems \ref{t-14}, \ref{t-14-new}, and Corollary \ref{c-14}).
Relation \eqref{e-38-new} above allows us to apply the Theorem on constant $J$-unitary factor (see \cite[Theorem 8.2.3]{ABT}, \cite{ArTs03}) to the realizing L-system that guarantees the existence of another L-system  with the same main and (densely-defined) symmetric operators and the impedance $V_{0,\alpha}(z)$. Therefore,  (see \cite[Theorem 7.1.4]{ABT})  in this case
$$
\int_{\mathbb R}d\mu_{0,\alpha}(t)=\lim_{\eta\to+\infty}\eta\IM V_{0,\alpha}(i\eta)=\infty,
$$
where $\mu_{0,\alpha}$ is the Borel measure in integral representation \eqref{hernev-real} of the function $V_{0,\alpha}(z)$.

We need to show that $\eta\IM  V_{\alpha}(i\eta)$ tends to infinity as $\eta\to+\infty$. Since $V_0(z)\in\sN$ and therefore has integral representation \eqref{hernev-0}, then  (see \cite{KK74})  we have
\begin{equation}\label{e-38-beta}
 \lim_{\eta\to+\infty}\frac{V_0(i\eta)}{i\eta}=0.
\end{equation}
Equation \eqref{e-38-beta} implies one of two basic cases:
\begin{enumerate}
  \item $|V_0(i\eta)|\le C$ in the neighborhood of infinity;
  \item $V_0(i\eta)$ is unbounded when $\eta\to\infty$.
\end{enumerate}
We treat each case individually.

(1) Suppose that $|V_0(i\eta)|\le C$ in the neighborhood of infinity. Then
$$
\frac{1}{|\sin\alpha-(\cos\alpha)(Q+V_0(i\eta)) |}\ge \frac{1}{|\sin\alpha|+|\cos\alpha|\cdot|Q|+|\cos\alpha|\cdot C}.
$$
Then \eqref{e-36-lemma} yields
$$
 \begin{aligned}
\eta\IM  V_{\alpha}(i\eta)&=\frac{\eta\IM V_0(i\eta)}{|\sin\alpha-(\cos\alpha)(Q+V_0(i\eta)) |^2}\\
&\ge \eta\IM V_0(i\eta)\cdot\frac{1}{(|\sin\alpha|+|\cos\alpha|\cdot|Q|+|\cos\alpha|\cdot C)^2}
\end{aligned}
$$
The first term $\eta\IM V_0(i\eta)$ tends to infinity as $\eta\to+\infty$ while the second term is a constant. Thus, the inequality implies that
$$
\lim_{\eta\to+\infty}\eta\IM  V_{\alpha}(i\eta)=\infty,
$$
as we needed.

(2) Suppose $V_0(i\eta)$ is unbounded in the neighborhood of infinity. Then there is a sequence $\{\eta_n\}$ such that $\eta_n\to\infty$ and $V_0(i\eta_n)\rightarrow\infty$ when $n\to\infty$. Applying \eqref{e-36-lemma} yields
\begin{equation}\label{e-37-lemma}
    \begin{aligned}
\eta\IM  &V_{\alpha}(i\eta)=\frac{\eta\IM V_0(i\eta)}{|\sin\alpha-(\cos\alpha)(Q+V_0(i\eta)) |^2}\\
&=\frac{\eta\IM V_0(i\eta)}{|\sin\alpha-(\cos\alpha)V_0(i\eta) |^2}\cdot \frac{|\sin\alpha-(\cos\alpha)V_0(i\eta) |^2}{|\sin\alpha-(\cos\alpha)(Q+V_0(i\eta)) |^2}\\
&=\eta\IM V_{0,\alpha}(i\eta)\cdot \frac{|\sin\alpha-(\cos\alpha)V_0(i\eta) |^2}{|\sin\alpha-(\cos\alpha)(Q+V_0(i\eta)) |^2}.
    \end{aligned}
\end{equation}
Passing to the limit along the sequence  $\eta_n$ in the right hand side of \eqref{e-37-lemma}  we see that the first factor tends to infinity (as explained above) and the second approaches $1$ under our case assumption. Consequently, the left hand side tends to infinity. Summarizing both cases we have
$$
\lim_{\eta\to+\infty}\eta\IM  V_{\alpha}(i\eta)=\infty.
$$
A similar approach can be used to show that
$$
 \lim_{\eta\to+\infty}\frac{V_\alpha(i\eta)}{i\eta}= \lim_{\eta\to+\infty}\left(\frac{\IM V_0(i\eta)}{i\eta}\cdot\frac{1}{|\sin\alpha-(\cos\alpha)(Q+V_0(i\eta)) |^2}\right)=0.
$$
Since $V_0(z)\in\sN$, the limit of the first factor is always zero. The limit of the second factor is zero along a sequence $\eta_n$ if $V_0(i\eta_n)\rightarrow\infty$ and a constant if $\lim_{n\to\infty}V_0(i\eta_n)=B$ provided that $Q\ne\tan\alpha -B$. In the exceptional case when $Q=\tan\alpha -B$ formula \eqref{e-35-lemma} implies that  $\lim_{\eta\to+\infty}V_\alpha(i\eta)=\infty$. To work around this case
we switch  to the function $\ti V_\alpha(z)=-1/V_\alpha(z)$ for which $\lim_{\eta\to+\infty}\ti V_\alpha(i\eta)=0$ and hence $\lim_{\eta\to+\infty}\frac{\ti V_\alpha(i\eta)}{i\eta}=0$.
Repeating the argument above for the function $\ti V_\alpha(z)$ we obtain that $\ti V_\alpha(z)$  has the form \eqref{hernev-real} with infinite Borel measure $\ti\mu_\alpha$ and is realizable (see Section \ref{s10}). 
Therefore  $V_\alpha(z)$ in this exceptional case has representation \eqref{hernev-real} with infinite Borel measure $\mu_\alpha$.

Summarizing all the cases in the above reasoning and using \cite{KK74} we get
$$
\int_{\mathbb R}d\mu_\alpha(t)=\lim_{\eta\to+\infty}\eta\IM  V_{\alpha}(i\eta)=\infty.
$$
Hence,  $\mu_\alpha$ is an infinite Borel measure corresponding to the function $V_\alpha(z)$ in \eqref{e-35-lemma}.
\end{proof}


\begin{remark}\label{r-15}

Suppose $\Theta_0$ is a model L-system of the form \eqref{e-59'} that realizes the function $V_0(z)\in\sM$ 
 and constructed according to the procedure in Appendix \ref{A1} under Hypothesis \ref{setup} with $\kappa_0=0$. Recall that this model uses the Borel measure $\mu$ from the integral representation \eqref{e-52-M-q} of $V_0(z)$.
Consider a transformation $V_{0,\alpha}(z)$ of $V_0(z)$ similar to \eqref{e-35-lemma} that is
$$
V_{0,\alpha}(z)=\frac{\cos\alpha+(\sin\alpha)V_0(z)}{\sin\alpha-(\cos\alpha)V_0(z)}.
$$
It is known (see \cite{MT-S}, \cite{BMkT}) that $V_{0,\alpha}(z)$ belongs to class $\sM$ and hence is realizable. Let
 $$W_{\Theta_{0,\alpha}}(z)=(-e^{2i\alpha})W_{\Theta_0}(z),$$
where $W_{\Theta_0}(z)$ is the transfer function of $\Theta_0$. It is known \cite{ABT} that $W_{\Theta_{0,\alpha}}(z)$ and $V_{0,\alpha}(z)$ are related via \eqref{e6-3-6}.
We are going to construct an L-system $\Theta_{0,\alpha}$ (out of the elements of $\Theta_0$) that realizes $V_{0,\alpha}(z)$ and has $W_{\Theta_{0,\alpha}}(z)$ as its transfer function. Let $\dot\cB$ of the form \eqref{nacha2-ap} in Appendix \ref{A1} be the symmetric operator in the model realizing L-system $\Theta_{0}$ with deficiency vectors $g_z$ and $g_\pm$ of the form \eqref{e-52-def} in the model space $L^2(\bbR;d\mu)$.
 Consider another pair of normalized deficiency vectors
\begin{equation}\label{e-76-def}
g_+^\alpha=g_+\quad and \quad g_-^\alpha=(-e^{2i\alpha})g_-.
\end{equation}
Using the symmetric operator $\dot\cB$ as above, the model space $L^2(\bbR;d\mu)$, and this new deficiency basis \eqref{e-76-def} we build an L-system $\Theta_{0,\alpha}$ of the form \eqref{e-59'} as described in Appendix \ref{A2} with $\kappa=0$. Note, that $\Theta_{0,\alpha}$ satisfies the conditions of Hypothesis \ref{setup}. Let $s_0(z)$ and $S_0(z)$ be  the Livsic and the characteristic functions   associated with L-system $\Theta_0$, respectively,  while $s_\alpha(z)$ and $S_\alpha(z)$ be the ones related to L-system $\Theta_{0,\alpha}$ (see Appendix \ref{A1} and references therein for details).
Recall that by the definition \cite{L}, \cite{MT-S} the general formulas for $s(z)$ and $S(z)$ are
\begin{equation}\label{e-42-Liv}
s(z)=\frac{z-i}{z+i}\cdot \frac{(g_z, g_-)}{(g_z, g_+)},\quad S(z)=\frac{s(z)-\kappa} {\overline{ \kappa }\,s(z)-1}, \quad z\in \bbC_+,
\end{equation}
where $g_z$ and $g_\pm=g_{\pm i}$ are the deficiency vectors of symmetric operator $\dot\cB$ and $\kappa$ is the von Neumann parameter of the main operator  of the L-system under consideration. Then since $\kappa_0=\kappa_\alpha=0$
$$
s_0(z)=\frac{z-i}{z+i}\cdot \frac{(g_z, g_-)}{(g_z, g_+)},\; S_0(z)=-s_0(z),\; s_\alpha(z)=\frac{z-i}{z+i}\cdot \frac{(g_z, g_-^\alpha)}{(g_z, g_+^\alpha)},\; S_\alpha(z)=-s_\alpha(z).
$$
We observe that
$$
\begin{aligned}
S_\alpha(z)&=-s_\alpha(z)=-\frac{z-i}{z+i}\cdot \frac{(g_z, g_-^\alpha)}{(g_z, g_+^\alpha)}=-\frac{z-i}{z+i}\cdot \frac{(g_z, (-e^{2i\alpha})g_-)}{(g_z, g_+)}\\
&=(e^{-2i\alpha})\frac{z-i}{z+i}\cdot \frac{(g_z, g_-)}{(g_z, g_+)}=e^{-2i\alpha}s_0(z).\\
\end{aligned}
$$
Therefore,
$$
s_\alpha(z)=(-e^{-2i\alpha})s_0(z), \quad z\in \bbC_+.
$$
It was shown in \cite{BMkT}  that
$$
 S_0(z)=\frac{1}{W_{\Theta_0}(z)},\quad
$$
and 
$$
S_\alpha(z)=\frac{1}{W_{\Theta_{0,\alpha}}(z)},
$$
where $z$ is from the domain of definition of $S_0(z)$.
Consequently,
$$
-s_\alpha(z)=S_\alpha(z)=\frac{1}{W_{\Theta_{0,\alpha}}(z)}=e^{-2i\alpha}s_0(z)=(-e^{-2i\alpha})S_0(z)=\frac{1}{(-e^{2i\alpha})W_{\Theta_0}(z)}.
$$
Thus,
$$
W_{\Theta_{0,\alpha}}(z)=(-e^{2i\alpha})W_{\Theta_0}(z),
$$
and hence $\Theta_{0,\alpha}$ realizes  $V_{0,\alpha}(z)$ and has $W_{\Theta_{0,\alpha}}(z)$ as its transfer function. On the other hand, the same L-system $\Theta_{0,\alpha}$ can be considered with the old deficiency basis $g_\pm$. Choosing the old basis does not affect the main operator $\whB$ since $g_+^\alpha=g_+$ and $\kappa=\kappa_0=0$ but the quasi-kernel $\cB_\alpha$ of the real-part of the state-space operator of $\Theta_{0,\alpha}$ has its domain written as
\begin{equation}\label{e-77-qk}
\dom(\cB_\alpha)=\dom (\dot \cB)\dot +\linspan\left\{\,\frac{1}{\cdot -i}+ (e^{2i\alpha})\frac{1}{\cdot +i}\right \}.
\end{equation}
As a result, our L-system $\Theta_{0,\alpha}$ written in terms of the original deficiency basis $g_\pm$ no longer satisfies Hypothesis \ref{setup}.

\end{remark}

Let us focus on the  perturbed class $\sM^Q$. Naturally here and below we assume that $Q\ne0$ or otherwise $\sM^0=\sM$ and the class is not perturbed. Recall that every function $V(z)\in\sM^Q$ admits integral representation \eqref{hernev-real} and has condition \eqref{e-42-int-don} on the measure $\mu$. That is,
 \begin{equation}\label{e-52-M-q}
V(z)= Q+\int_\bbR\left (\frac{1}{\lambda-z}-\frac{\lambda}{1+\lambda^2}\right )d\mu,\quad \int_\bbR\frac{d\mu(\lambda)}{1+\lambda^2}=1,\quad Q=\bar Q.
\end{equation}

\begin{theorem}\label{t-18-M-q}
Let  $V(z)$ belong to the  class $\sM^Q$ and have integral representation \eqref{e-52-M-q}.  Then $V(z)$ can be realized \footnote{This means, in particular, that we choose the appropriate rigged Hilbert space and the deficiency vectors of the symmetric operator of the realizing L-system.}
as the impedance function $V_{\Theta}(z)$ of a minimal L-system $\Theta$  of the form \eqref{e-62} with  the main operator $T$ whose von Neumann's parameter  $\kappa$ is determined by the formula
  \begin{equation}\label{e-53-kappa'}
    \kappa=\frac{|Q|}{\sqrt{Q^2+4}},\quad Q\ne0.
 \end{equation}
Moreover, the quasi-kernel $\hat A$ of $\RE\bA$ of the realizing L-system $\Theta$ is defined by \eqref{DOMHAT} with
\begin{equation}\label{e-54-U-M-q}
 U=\frac{Q}{|Q|}\cdot\frac{-Q+2i}{\sqrt{Q^2+4}},\quad Q\ne0.
\end{equation}
\end{theorem}
\begin{proof}
Since $V(z)\in\sM^Q$, then $$V(z)=Q+V_0(z),$$ where $V_0(z)\in\sM$. Let $\Theta_0$ be a model L-system of the form \eqref{e-59'} that realizes the function $V_0(z)\in\sM$  and is constructed according to the procedure in Appendix \ref{A1} under Hypothesis \ref{setup} with $\kappa_0=0$.
Consider another function $V_{\Theta_\alpha}(z)$ related to our function $V(z)\in \sM^Q$ by
\begin{equation}\label{e-54-frac}
    V_{\Theta_\alpha}(z)=\frac{\cos\alpha+(\sin\alpha)V(z)}{\sin\alpha-(\cos\alpha)V(z)},\quad z\in\dC_+,
 \end{equation}
for some $\alpha\in[0,\pi)$. It is known (see \cite[Theorem 8.3.2]{ABT}) that $V_{\Theta_\alpha}(z)$  falls into the class $\sN^Q$ and hence admits integral representation \eqref{hernev-real} with an infinite Borel measure $\mu_\alpha$ (see Lemma \ref{l-12}).
By direct substitution one gets $V(i)=Q+i$ and hence
$$
\begin{aligned}
V_{\Theta_\alpha}(i)&=\frac{\cos\alpha+(\sin\alpha)V(i)}{\sin\alpha-(\cos\alpha)V(i)}=\frac{\cos\alpha+(\sin\alpha)(Q+i)}{\sin\alpha-(\cos\alpha)(Q+i)}\\
&=\frac{(\cos\alpha+Q\sin\alpha)+i\sin\alpha}{(\sin\alpha-Q\cos\alpha)-i\cos\alpha}=\frac{-Q\cos2\alpha-(1/2)Q^2\sin2\alpha}{(\sin\alpha-Q\sin\alpha)^2+\cos^2\alpha}\\
&+i\,\frac{1}{(\sin\alpha-Q\cos\alpha)^2+\cos^2\alpha}=Q_\alpha+i\int_{\dR}\frac{d\mu_\alpha(\lambda)}{1+\lambda^2}=Q_\alpha+ia_\alpha,
\end{aligned}
$$
where $Q_\alpha$  and  $\mu_\alpha$  are the elements of integral representation \eqref{hernev-real} of the function $V_{\Theta_\alpha}(z)$ and $a_\alpha=\int_{\dR}\frac{d\mu_\alpha(\lambda)}{1+\lambda^2}$. Thus,
\begin{equation}\label{e-55-q}
    Q_\alpha=\frac{-Q\cos2\alpha-(1/2)Q^2\sin2\alpha}{(\sin\alpha-Q\cos\alpha)^2+\cos^2\alpha},
\end{equation}
and
\begin{equation}\label{e-56-q-int}
   \int_{\dR}\frac{d\mu_\alpha(\lambda)}{1+\lambda^2}=\frac{1}{(\sin\alpha-Q\cos\alpha)^2+\cos^2\alpha}.
\end{equation}
Both formulas \eqref{e-55-q} and \eqref{e-56-q-int} hold true for any value of $\alpha\in[0,\pi)$ and if we set $Q_\alpha=0$ in \eqref{e-55-q}, we obtain that either $Q=0$ (we discard this case since it takes us back to the class $\sM$)
or
\begin{equation}\label{e-59-alpha-q}
  \tan 2\alpha=-\frac{2}{Q}.
\end{equation}
Changing left hand side of \eqref{e-59-alpha-q} to $\tan\alpha$ gives
$$
\tan2\alpha=\frac{2\tan\alpha}{1-\tan^2\alpha}=-\frac{2}{Q},
$$
that leads us to the quadratic equation
$$
\tan^2\alpha- Q\tan\alpha-1=0,
$$
whose solutions are
\begin{equation}\label{e-61-tan}
\tan\alpha=\frac{Q\pm\sqrt{Q^2+4}}{2},\quad \alpha\in[0,\pi].
\end{equation}
Using this formula and the identity $\sec^2\alpha=\tan^2\alpha+1$, we obtain
\begin{equation}\label{e-62-cos}
\cos^2\alpha=\frac{2}{{Q^2+4\pm Q\sqrt{Q^2+4}}}.
\end{equation}
We are going to apply the above formulas for $\tan\alpha$ and $\cos^2\alpha$ to modify the denominator in \eqref{e-56-q-int} and express it in terms of $Q$.
\begin{equation}\label{e-62-denom}
\begin{aligned}
(\sin\alpha&-Q\cos\alpha)^2+\cos^2\alpha=[\cos\alpha(\tan\alpha-Q)]^2+\cos^2\alpha\\
&=\cos^2\alpha[(\tan\alpha-Q)^2+1]\\
&=\frac{2}{{Q^2+4\pm Q\sqrt{Q^2+4}}}\left[\left(\frac{Q\pm\sqrt{Q^2+4}}{2}-Q\right)^2+1\right]\\
&=\frac{2}{{Q^2+4\pm Q\sqrt{Q^2+4}}}\left[\left(\frac{-Q\pm\sqrt{Q^2+4}}{2}\right)^2+1\right]\\
&=\frac{2}{{Q^2+4\pm Q\sqrt{Q^2+4}}}\cdot\frac{Q^2\mp 2Q\sqrt{Q^2+4}+(Q^2+4)+4}{4}\\
&=\frac{2Q^2+8\mp 2Q\sqrt{Q^2+4}}{2Q^2+8\pm 2Q\sqrt{Q^2+4}}=\frac{Q^2+4\mp Q\sqrt{Q^2+4}}{Q^2+4\pm Q\sqrt{Q^2+4}}=\frac{\sqrt{Q^2+4}\mp Q}{\sqrt{Q^2+4}\pm Q}.
\end{aligned}
\end{equation}
First let us assume that $\alpha\in[0,\pi/2)$. Then   $\tan\alpha\ge0$, and formulas  \eqref{e-61-tan} and \eqref{e-62-cos} become
\begin{equation}\label{e-61-cos-sin}
   \tan\alpha=\frac{Q+\sqrt{Q^2+4}}{2},\quad \cos^2\alpha=\frac{2}{{Q^2+4+ Q\sqrt{Q^2+4}}}.
\end{equation}
The choice of signs in \eqref{e-61-cos-sin} is unique and follows from the positivity of $\cos\alpha$ and $\tan\alpha$ in \eqref{e-61-tan}. 
Substituting these values in \eqref{e-56-q-int} we use \eqref{e-62-denom} to get
\begin{equation}\label{e-60-a-alpha}
 a_{\alpha}=   \int_{\dR}\frac{d\mu_{\alpha}(\lambda)}{1+\lambda^2}=\frac{1}{(\sin\alpha-Q\cos\alpha)^2+\cos^2\alpha}=\frac{\sqrt{Q^2+4}+ Q}{\sqrt{Q^2+4}- Q},\quad \alpha\in\left[0,\frac{\pi}{2}\right).
\end{equation}
Note that $a_{\alpha}$ given by  \eqref{e-60-a-alpha} is less than 1 if $Q<0$ and greater than 1 if $Q>0$. Moreover, \eqref{e-59-alpha-q} and \eqref{e-61-cos-sin} imply that if $Q<0$, then
$ \alpha\in\left[0,{\pi}/{4}\right)$ and $a_{\alpha}<1$. Likewise, if $Q>0$, then $ \alpha\in\left({\pi}/{4},\pi/2\right)$ and $a_{\alpha}>1$.

Now let $\alpha\in(\pi/2,\pi]$. Then   $\tan\alpha\le0$ and similar to the above analysis yields
\begin{equation}\label{e-63-cos-sin}
    \tan\alpha=\frac{Q-\sqrt{Q^2+4}}{2},\quad \cos^2\alpha=\frac{2}{{Q^2+4- Q\sqrt{Q^2+4}}}.
\end{equation}
Consequently,
\begin{equation}\label{e-64-a-alpha}
 a_{\alpha}=   \int_{\dR}\frac{d\mu_{\alpha}(\lambda)}{1+\lambda^2}=\frac{1}{(\sin\alpha-Q\cos\alpha)^2+\cos^2\alpha}=\frac{\sqrt{Q^2+4}- Q}{\sqrt{Q^2+4}+ Q},\quad \alpha\in\left(\frac{\pi}{2},\pi\right].
\end{equation}
Again we note that $a_{\alpha}$ given by  \eqref{e-64-a-alpha} is less than 1 if $Q>0$ and greater than 1 if $Q<0$. Moreover, \eqref{e-59-alpha-q} and \eqref{e-61-cos-sin} imply that if $Q<0$, then
$ \alpha\in\left({\pi}/{2},{3\pi}/{4}\right)$ and $a_{\alpha}>1$. Likewise, if $Q>0$, then $ \alpha\in\left({3\pi}/{4},\pi\right]$ and $a_{\alpha}<1$.

Considering the above, we conclude that \eqref{e-60-a-alpha} and \eqref{e-64-a-alpha} contain two possible answers of the value of $a_\alpha$ for any given non-zero $Q\in\dR$
\begin{equation}\label{e-62-1-2}
a_{\alpha_1}=\frac{\sqrt{Q^2+4}- |Q|}{\sqrt{Q^2+4}+ |Q|}<1\quad\textrm{ and }\quad a_{\alpha_2}=\frac{\sqrt{Q^2+4}+ |Q|}{\sqrt{Q^2+4}- |Q|}>1,
\end{equation}
and clearly $a_{\alpha_1}=1/a_{\alpha_2}$. Moreover, we have that
$$
V_{\Theta_{\alpha_1}}(i)=-\frac{1}{V_{\Theta_{\alpha_2}}(i)}.
$$
This (see  \cite[Theorem 5.3]{BMkT-2}) indicates that $V_{\Theta_{\alpha_1}}(z)\in\sM_\kappa$ and $V_{\Theta_{\alpha_2}}(z)\in\sM_\kappa^{-1}$ for the same value of $\kappa$ that we can find via
\eqref{e-45-kappa-1} and \eqref{e-45-kappa-2}. Indeed,
\begin{equation}\label{e-38-kap'}
a_{\alpha_1}=\int_\bbR\frac{d\mu_{\alpha_1}(\lambda)}{1+\lambda^2}=\frac{1-\kappa}{1+\kappa},
\end{equation}
and
\begin{equation}\label{e-39-kap'}
a_{\alpha_2}=\int_\bbR\frac{d\mu_{\alpha_2}(\lambda)}{1+\lambda^2}=\frac{1+\kappa}{1-\kappa},
\end{equation}
for some value $\kappa\in(0,1)$. Combining \eqref{e-45-kappa-1} and \eqref{e-45-kappa-2} with \eqref{e-38-kap'},  \eqref{e-39-kap'},  and \eqref{e-62-1-2} to solve for $\kappa$ yields
$$
\begin{aligned}
\kappa&=\frac{1-a_{\alpha_1}}{1+a_{\alpha_1}}=\frac{1-\frac{\sqrt{Q^2+4}- |Q|}{\sqrt{Q^2+4}+ |Q|}}{1+\frac{\sqrt{Q^2+4}- |Q|}{\sqrt{Q^2+4}+ |Q|}}=\frac{\sqrt{Q^2+4}+|Q|-\sqrt{Q^2+4}+|Q|}{\sqrt{Q^2+4}+|Q|+\sqrt{Q^2+4}-|Q|}\\
&=\frac{2|Q|}{2\sqrt{Q^2+4}}=\frac{|Q|}{\sqrt{Q^2+4}}=\frac{a_{\alpha_2}-1}{1+a_{\alpha_2}}.
\end{aligned}
$$
Thus,
\begin{equation}\label{e-65-kappa}
\kappa(Q)=\frac{|Q|}{\sqrt{Q^2+4}}.
\end{equation}
The graph of $\kappa$ as a function of $Q$ is shown on Figure \ref{fig-1}. We note that $\kappa(Q)$ is an even function whose derivative for $Q>0$ is
$$
\kappa'(Q)=\frac{4}{(Q^2+4)^{3/2}},\quad Q>0,
$$
giving the slope of the graph at $Q=0$ as $\kappa'(0+)=1/2$. The graph of the function is symmetric with respect to the $\kappa$-axis.
\begin{figure}
  \begin{center}
  \includegraphics[width=110mm]{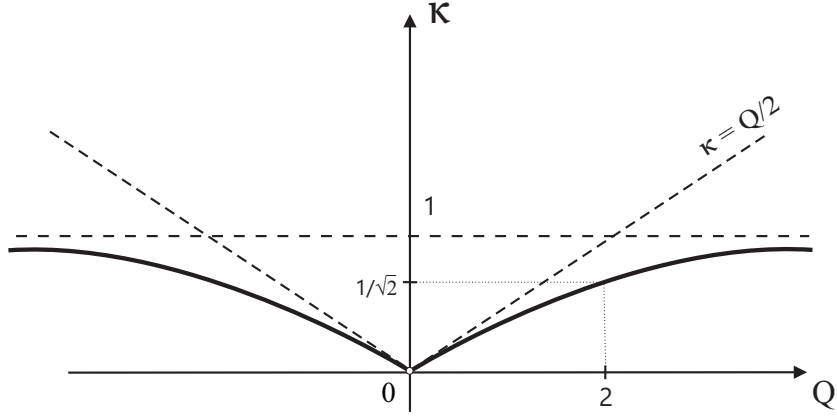}
  \caption{Class $\sM^Q$: $\kappa$ as a function of $Q$}\label{fig-1}
  \end{center}
\end{figure}

The formulas for $V_{\Theta_{\alpha_1}}(z)\in\sM_\kappa$ and $V_{\Theta_{\alpha_2}}(z)\in\sM_\kappa^{-1}$ can be derived in a more explicit format. In order to do this we begin by solving the formulas \eqref{e-62-1-2} for $Q$ which leads to (in both formulas)
$$
Q^2=\frac{(a-1)^2}{a}=a-2+\frac{1}{a},\quad\textrm{ where }\quad a=a_{\alpha_{1}} \quad\textrm{ or }\quad a=a_{\alpha_{2}}=\frac{1}{a_{\alpha_1}},
$$
and hence
\begin{equation}\label{e-60-Q}
    |Q|=\frac{1-a_{\alpha_1}}{\sqrt{a_{\alpha_1}}}=\frac{a_{\alpha_2}-1}{\sqrt{a_{\alpha_2}}}.
\end{equation}
Furthermore, using \eqref{e-61-tan} with \eqref{e-62-1-2} yields
\begin{equation}\label{e-61-new}
\begin{aligned}
\tan^2\alpha_{1,2}&=\frac{(Q\pm\sqrt{Q^2+4})^2}{4}=\frac{(Q\pm\sqrt{Q^2+4})^2}{Q^2+4-Q^2}\\
&=\frac{(\sqrt{Q^2+4}\mp |Q|)^2}{(\sqrt{Q^2+4}\pm |Q|)(\sqrt{Q^2+4}\mp |Q|)}=\frac{\sqrt{Q^2+4}\mp |Q|}{\sqrt{Q^2+4}\pm |Q|}=a_{\alpha_{1,2}}.
\end{aligned}
\end{equation}
Consequently, applying \eqref{e-35-lemma} with $\alpha=\alpha_1$ and $V(z)=Q+V_0(z)$, $V_0(z)\in\sM$, and using \eqref{e-60-Q} and \eqref{e-61-new} gives us
\begin{equation}\label{e-62-v1}
    \begin{aligned}
 V_{{\Theta_{\alpha_1}}}(z)&=\frac{\cos\alpha_1+(\sin\alpha_1)V(z)}{\sin\alpha_1-(\cos\alpha_1)V(z)}=\frac{1+\tan\alpha_1(Q+V_0(z))}{\tan\alpha_1-(Q+V_0(z))}\\
&=\frac{1\pm\sqrt{a_{\alpha_1}}(Q+V_0(z))}{\pm\sqrt{a_{\alpha_1}}-Q-V_0(z)}
=\frac{1\pm\sqrt{a_{\alpha_1}}\left(\mp\frac{1-a_{\alpha_1}}{\sqrt{a_{\alpha_1}}}+V_0(z)\right)}{\pm\sqrt{a_{\alpha_1}}\pm\frac{1-a_{\alpha_1}}{\sqrt{a_{\alpha_1}}}-V_0(z)}\\
&={a_{\alpha_1}}\frac{\sqrt{a_{\alpha_1}}+V_0(z)}{1-\sqrt{a_{\alpha_1}}\,V_0(z)}
=a_{\alpha_1}\frac{\tan\alpha_1+V_0(z)}{1-(\tan\alpha_1) V_0(z)}\\
&={a_{\alpha_1}}\frac{\sin\alpha_1+(\cos\alpha_1)V_0(z)}{\cos\alpha_1-(\sin\alpha_1)V_0(z)}.
    \end{aligned}
\end{equation}
One can notice that \eqref{e-60-Q} implies that the value of $|Q|$ does not change after replacing $a=a_{\alpha_1}$ with $a=a_{\alpha_2}=1/a_{\alpha_1}$ and hence working with $\alpha=\alpha_2$ is similar to the above and yields
\begin{equation}\label{e-62-v2}
    \begin{aligned}
  V_{{\Theta_{\alpha_2}}}(z)&=\frac{\cos\alpha_2+(\sin\alpha_2)V(z)}{\sin\alpha_2-(\cos\alpha_2)V(z)}=a_{\alpha_2}\frac{\sin\alpha_2+(\cos\alpha_2)V_0(z)}{\cos\alpha_2-(\sin\alpha_2)V_0(z)}.\\
    \end{aligned}
\end{equation}
The expressions for $V_{\Theta_{\alpha_1}}(z)\in\sM_\kappa$ in \eqref{e-62-v1} and $V_{\Theta_{\alpha_2}}(z)\in\sM_\kappa^{-1}$  in \eqref{e-62-v2}  clearly confirm that
$$
V_{\Theta_{\alpha_1}}(z)=-\frac{1}{V_{\Theta_{\alpha_2}}(z)},\quad z\in\dC_+,
$$
since $\alpha_2=\alpha_1+\pi/2$ and $a_{\alpha_2}=1/a_{\alpha_1}$. Moreover, it is easy to see that the functions
\begin{equation}\label{e-62-WT}
V_{0,\alpha_1}(z)=\frac{\sin\alpha_1+(\cos\alpha_1)V_0(z)}{\cos\alpha_1-(\sin\alpha_1)V_0(z)} \quad\textrm{ and }\quad V_{0,\alpha_2}(z)=\frac{\sin\alpha_2+(\cos\alpha_2)V_0(z)}{\cos\alpha_2-(\sin\alpha_2)V_0(z)}
\end{equation}
are both members of the class $\sM$ since they are obtained from the function $V_0(z)\in\sM$ via transformation \eqref{e-35-lemma} with the angle values $\alpha=\alpha_1+\pi/2$ and $\alpha=\alpha_2+\pi/2$, respectively.

Applying Theorems  \ref{t-14} or \ref{t-14-new} we can realize the  function $V_{\Theta_{{\alpha_1}}}(z)$ in \eqref{e-62-v1} (or $V_{\Theta_{{\alpha_2}}}(z)$ in \eqref{e-62-v2}) by a minimal L-system   that satisfies either Hypothesis \ref{setup} (or Hypothesis \ref{setup-1}). We will, however, find the realizing L-systems in a way that they will share the state-space and symmetric operator $\dot\cB$ with the model L-system $\Theta_0$ mentioned in the beginning of the proof. The procedure is described in Remark \ref{r-15} and just needs to be adapted to the current construction.


Let $\Theta_{0,\alpha_1}$ and $\Theta_{0,\alpha_2}$ be the L-systems realizing  functions $V_{0,\alpha_1}(z)$ and $V_{0,\alpha_2}(z)$  of the class $\sM$ given by \eqref{e-62-WT} and sharing the model state space and symmetric operator $\dot\cB$ with $\Theta_0$  described above. That is, both L-systems $\Theta_{0,\alpha_1}$ and $\Theta_{0,\alpha_2}$ share the state space $\calH_+ \subset L^2(\bbR;d\mu)\subset\calH_-$, (here $\mu$ is the measure in the integral representation \eqref{hernev-real}  of $V_0(z)$), the  symmetric $\dot\cB$  and main $T^0_{\calB}$  operators  with the L-system $\Theta_{0}$ (see Remark \ref{r-15} and Appendix \ref{A1} for details). Recall that we are using the fixed normalized in $L^2(\bbR;d\mu)$ deficiency pair $g_\pm$ given by \eqref{e-52-def}.
 What makes $\Theta_{0,\alpha_1}$ and $\Theta_{0,\alpha_2}$ different  from each other are the quasi-kernels of the real parts of the state-space operators of the form  \eqref{e-77-qk} that are
$$
\begin{aligned}
\dom(\cB_{\alpha_1})&=\dom (\dot \cB)\dot +\linspan\left\{\,\frac{1}{\cdot -i}+ (e^{2i\alpha_1})\frac{1}{\cdot +i}\right \},\\
\dom(\cB_{\alpha_2})&=\dom (\dot \cB)\dot +\linspan\left\{\,\frac{1}{\cdot -i}+ (e^{2i\alpha_2})\frac{1}{\cdot +i}\right \}.
\end{aligned}
$$
Recall that $\alpha_2=\alpha_1+\pi/2$ and hence $e^{2i\alpha_1}=-e^{2i\alpha_2}$.
As it was laid out in Remark \ref{r-15} we introduce new normalized deficiency vectors for $\dot\cB$ of the form \eqref{e-76-def}
\begin{equation}\label{e-67-def}
g_+^\alpha=g_+(t)=\frac{1}{t-i},\quad g_-^\alpha=(-e^{2i\alpha})g_-(t)=-\frac{e^{2i\alpha}}{t +i},\quad\alpha=\alpha_1.
\end{equation}
Note that for this choice of deficiency vectors the L-systems $\Theta_{0,\alpha_1}$ and $\Theta_{0,\alpha_2}$ satisfy Hypothesis \ref{setup} and \ref{setup-1} (with $\kappa=0$).
We are going to find the values of unimodular factors $e^{2i\alpha_1}$ and $e^{2i\alpha_2}$ in terms of our parameter $Q$.
According to  \eqref{e-59-alpha-q},  $\tan2\alpha=-2/Q$, and hence we have
$$
\cos 2\alpha_1=-\frac{Q}{\sqrt{Q^2+4}}\quad\textrm{ and }\quad \sin 2\alpha_1=\frac{2}{\sqrt{Q^2+4}},
$$
or
$$
\cos 2\alpha_2=\frac{Q}{\sqrt{Q^2+4}}\quad\textrm{ and }\quad \sin 2\alpha_2=-\frac{2}{\sqrt{Q^2+4}}.
$$
Thus,
\begin{equation}\label{e-65-B}
\begin{aligned}
-e^{2i\cdot{\alpha_1}}&=-\cos2\alpha_1-i\sin2\alpha_1=\frac{Q-2i}{\sqrt{Q^2+4}},\\
-e^{2i\cdot{\alpha_2}}&=-\cos2\alpha_2-i\sin2\alpha_2=-\frac{Q-2i}{\sqrt{Q^2+4}}.
\end{aligned}
\end{equation}
Now we can construct two new L-systems $\Theta_{{\alpha_1}}$  and  $\Theta_{{\alpha_2}}$  by modifying $\Theta_{0,\alpha_1}$ and $\Theta_{0,\alpha_2}$ as follows: the main operator $T^0_{\calB}$ (with parameter $\kappa=0$) is replaced with $\whB $ of the form \eqref{nacha3-ap} and having the von Neumann parameter $\kappa$ of the form \eqref{e-53-kappa'} with respect to the modified deficiency basis \eqref{e-67-def}.
Following Appendices \ref{A2} and \ref{A1} we have
$$
\Theta_{\alpha_1}= \begin{pmatrix} \dB_{\alpha_1}&K_{\alpha_1}&\ 1\cr \calH_+ \subset L^2(\bbR;d\mu) \subset\calH_-& &\dC\cr \end{pmatrix},
$$
where $g_+^\alpha- g_-^\alpha\in \dom(\cB_{\alpha_1})$ and $g_+^\alpha-\kappa g_-^\alpha\in \dom (\whB )$ for $\kappa$ defined in \eqref{e-53-kappa'}.
Also, applying \eqref{e-57'} and \eqref{e-17-real} (with $\varphi=\frac{1}{\sqrt2}\calR^{-1}g_+^\alpha$ and $\psi=\frac{1}{\sqrt2}\calR^{-1}g_-^\alpha$) yields
$$
       \IM\dB_{\alpha_1}=(\cdot,\chi_{\alpha_1})\chi_{\alpha_1},\quad \chi_{\alpha_1}=\sqrt{\frac{1-\kappa}{1+\kappa}}\left(\frac{1}{2}\,\calR^{-1}g_+^\alpha- \frac{1}{2}\,\calR^{-1}g_-^\alpha\right),
 $$
 $$
  \RE\dB_1=\dot\cB^*+\frac{i}{4}\left(\cdot,\calR^{-1}g_+^\alpha+\calR^{-1}g_-^\alpha\right)\left(\calR^{-1}g_+^\alpha-\calR^{-1}g_-^\alpha\right),
 $$
and  $K_{\alpha_1} c=c\cdot \chi_{\alpha_1}$, $K_1^*f=(f,\chi_{\alpha_1})$, $(f\in\calH_+)$.

The L-system $\Theta_{\alpha_2}$ has exact same format once parameter $\alpha_1$ is replaced with $\alpha_2$. Alternatively it can also be written in terms of $\alpha_1$ using the fact that $e^{2i\alpha_1}=-e^{2i\alpha_2}$.
 We remind the reader that in this case the entire construction of $\Theta_{{\alpha_1}}$  (or  $\Theta_{{\alpha_2}}$) is based upon the measure $\mu$ in the integral representation \eqref{hernev} of our unperturbed function $V_{0}(z)\in\sM$ and that $\Theta_{{\alpha_1}}$  and  $\Theta_{{\alpha_2}}$ realize $V_{\Theta_{\alpha_1}}(z)$ and $V_{\Theta_{\alpha_2}}(z)$, respectively.
By construction $\Theta_{{\alpha_1}}$ and $\Theta_{{\alpha_2}}$ share the same main operator and hence the same value of $\kappa$ of the form \eqref{e-53-kappa'}.

Let $\Theta_{{\alpha_1}}$ and  $\Theta_{{\alpha_2}}$ be constructed above L-systems whose impedance functions are $V_{\Theta_{\alpha_1}}(z)$ and $V_{\Theta_{\alpha_2}}(z)$.
For $j=1,2$ consider $W_{\Theta_{\alpha_{j}}}(z)$, the transfer function of $\Theta_{{\alpha_{j}}}$ related to $V_{\Theta_{{\alpha_{j}}}}(z)$  via \eqref{e6-3-6} with $J=1$, and the function
$$
W(z)=\frac{1-i V(z)}{1+i V(z)}.
$$
It was shown in  \cite[Theorem 8.3.2]{ABT} that in this case
\begin{equation}\label{e-61-Junitary}
W(z)=(-e^{2i\cdot{\alpha_j}})W_{\Theta_{{\alpha_j}}}(z),\quad j=1,2.
\end{equation}
Relation \eqref{e-61-Junitary} above allows us to apply the Theorem on constant $J$-unitary factor (see \cite[Theorem 8.2.3]{ABT}, \cite{ArTs03}) that guarantees the existence of an L-system $\Theta$ with the same main operator $\whB$ as in $\Theta_{\alpha_j}$, $(j=1,2)$ and such that  $V(z)=V_\Theta(z)$, ($z\in\dC_\pm$). In order to find the quasi-kernel $\cB_U$ of the real part of the state-space operator in $\Theta$ and its unimodular parameter $U$  we apply formula \eqref{e-216-ls} of Appendix \ref{A2}. Adapted to our case and written for $z=-i$ it becomes
$$
 \left\{
      \begin{array}{ll}
        \frac{\overline{\kappa^2+1+2\kappa U}}{\sqrt2|1+\kappa U|\sqrt{1-\kappa^2}}\, a(-i)+\frac{\overline{\kappa^2 U+2\kappa+U}}{\sqrt2|1+\kappa U|\sqrt{1-\kappa^2}}\,b(-i)=V(-i) & \hbox{} \\
&\\
        -\frac{i\sqrt{1-\kappa^2}}{\sqrt2|1+\kappa U|}a(-i)+\frac{i\sqrt{1-\kappa^2}\bar U}{\sqrt2|1+\kappa U|}b(-i)=1. & \hbox{}
      \end{array}
    \right.
$$
Taking into account that $a(-i)=0$ and $V(-i)=Q+V_0(-i)=Q-i$ we get
\begin{equation}\label{e-70-sys}
     \left\{
      \begin{array}{ll}
       \frac{\overline{\kappa^2 U+2\kappa+U}}{\sqrt2|1+\kappa U|\sqrt{1-\kappa^2}}\,b(-i)=Q-i & \hbox{} \\
&\\
        \frac{i\sqrt{1-\kappa^2}\bar U}{\sqrt2|1+\kappa U|}b(-i)=1. & \hbox{}
      \end{array}
    \right.
\end{equation}
Eliminating $b(-i)$ and solving \eqref{e-70-sys} for $U$ yields
$$
U=\frac{-2\kappa^2+Qi-\kappa^2Qi}{2\kappa}.
$$
Substituting the value of $\kappa$ from \eqref{e-53-kappa'} and simplifying produces
$$
U=\frac{Q}{|Q|}\cdot\frac{-Q+2i}{\sqrt{Q^2+4}},
$$
that matches \eqref{e-54-U-M-q}. As it was explained in the end of Appendix \ref{A2}, linear system \eqref{e-216-ls} always has a unique solution for any value of $U$ and any $z\in\dC_\pm$.  In particular, using $U$ of the form \eqref{e-54-U-M-q} above and $V_\Theta(z)=Q+V_0(z)$ in \eqref{e-216-ls} produces the following solution pair
$$
a(z)=\frac{1}{\sqrt2}(V_0(z)+i),\quad b(z)=\frac{1}{\sqrt2}(V_0(z)-i)U.
$$


Now we are ready to describe the L-system $\Theta$ that realizes $V(z)$. Let $\Theta_{{\alpha_1}}$ be the  L-system that  was described earlier in the proof. Then $\Theta_{{\alpha_1}}$ is based upon the  triple $(\dot\cB, \whB ,\cB_{\alpha_1})$ that relies on the measure $\mu$ in integral representation \eqref{hernev-real} of the function $V_0(z)$ and so does the state space $\calH_+ \subset L^2(\bbR;d\mu)\subset\calH_-$ of $\Theta_{{\alpha_1}}$. We construct the L-system $\Theta$ of the form \eqref{e-59} in Appendix \ref{A2} corresponding to the model triple $(\dot\cB,\whB,\cB_U)$. This  L-system $\Theta$ shares the state space $\calH_+ \subset L^2(\bbR;d\mu)\subset\calH_-$, symmetric $\dot \cB$, and main $\whB$ operators with  $\Theta_{{\alpha_1}}$ but only differs by the quasi-kernel of the real part of the state-space operator $\cB_U$ which is defined via \eqref{e-187-new}. Note that in this construction we are using modified  deficiency vectors $g_\pm^\alpha$ of the form \eqref{e-67-def}. The values of $\kappa$ and $U$ in $\Theta$ are given by \eqref{e-53-kappa'} and \eqref{e-54-U-M-q}, respectively. The components of  $\Theta$ are described in Appendix \ref{A1}. This model L-system $\Theta$ realizes our function $V(z)$.

The proof is complete.
\end{proof}
\begin{remark}\label{r-16}
As we mentioned  in the end of the proof of Theorem \ref{t-18-M-q}, the modified deficiency basis $g_\pm^\alpha$ of the form \eqref{e-67-def} is used in the construction and description of the L-system $\Theta$ that realizes our function $V(z)$. Both von Neumann's parameters $\kappa$ and $U$ in $\Theta$ are given by \eqref{e-53-kappa'} and \eqref{e-54-U-M-q} with respect to these deficiency vectors. Alternatively, we could use the original normalized deficiency pair $g_\pm$ of the form \eqref{e-52-def} in the space $L^2(\bbR;d\mu)$. This approach, however, would produce a (generally speaking) complex value of the parameter $\kappa$ and a different value of parameter $U$. The corresponding formulas are
\begin{equation}\label{e-71-k}
    \kappa=\frac{|Q|(-e^{2i\alpha})}{\sqrt{Q^2+4}}=\frac{|Q|(Q-2i)}{Q^2+4}
\end{equation}
and
\begin{equation}\label{e-72-U}
 U=\frac{Q(-e^{2i\alpha})}{|Q|}\cdot\frac{-Q+2i}{\sqrt{Q^2+4}}=\frac{Q}{|Q|}\cdot\frac{4-Q^2+4Qi}{Q^2+4}.
\end{equation}
Here we used the value of $-e^{2i\alpha}$ found via \eqref{e-65-B}.

\end{remark}

Note that formulas  \eqref{e-53-kappa'} and \eqref{e-54-U-M-q} in the statement of Theorem \ref{t-18-M-q} indicate that the change of $Q$ to  $(-Q)$ in integral representation \eqref{e-52-M-q} of a function $V(z)$ to be realized, maintains the same value of $\kappa$ in the realizing L-system but  changes $U$ to $\bar U$, i.e, $U(-Q)=\bar U(Q)$.
\begin{corollary}\label{c-19} Let $V_1(z)=Q+V_{10}(z)$ and $V_2(z)=Q+V_{20}(z)$ be two functions such that $V_{10}(z)$ and $V_{20}(z)$ $(V_{10}(z)\ne V_{20}(z))$ belong to the Donoghue class $\sM$. Then both $V_1(z)$ and $V_2(z)$ can be realized as the impedance functions  of two not bi-unitarily equivalent L-systems of the form \eqref{e-62} whose  main operators have the same  von Neumann's parameter  $\kappa$ determined by \eqref{e-53-kappa'} and the quasi-kernels of the real parts of state-space operators defined via the same $U$  of the form \eqref{e-54-U-M-q}.
\end{corollary}
\begin{proof}
The proof immediately follows from Theorem \ref{t-18-M-q} and the structure of formula \eqref{e-53-kappa'}.  The realizing L-systems corresponding to $V_1(z)$ and $V_2(z)$ have different state-spaces and symmetric operators since they are constructed based on different measures in integral representation \eqref{hernev-0} of $V_{10}(z)$ and $V_{20}(z)$. The construction of these model realizing L-systems was described in details in the proof of Theorem \ref{t-18-M-q} and in Appendix \ref{A1}. The von Neumann parameters $\kappa$ and $U$, however, are the same for both L-systems and correspond to the respective basis of deficiency vectors \eqref{e-52-def} in a different space for each case.

Also, the realizing L-systems corresponding to different functions $V(z)\in\sM^Q$ are not bi-unitarily equivalent (otherwise their impedance functions would match).
\end{proof}
We bring attention of the reader to the fact that the von Neumann parameter $\kappa$ of an L-system  realizing a function $V(z)\in\sM^Q$ depends \textbf{only} on the value of $Q$ (see formula \eqref{e-53-kappa'}). The same is true about the parameter $U$ in \eqref{e-54-U-M-q}. Consequently, \textbf{all} the functions $V(z)\in\sM^Q$ with a fixed constant $Q$ in their integral representation \eqref{hernev-real} can be realized by  model L-systems $\Theta$ of the form \eqref{e-59} as described in the end of the proof of Theorem \ref{t-18-M-q}.
It is also worth mentioning that \eqref{e-53-kappa'} implies that $\kappa$ is an even function of $Q$ and hence $\kappa(-Q)=\kappa(Q)$.


\section{Perturbations of the classes $\sM_\kappa$ and $\sM_\kappa^{-1}$}\label{s6}
In this section we  treat the remaining perturbed subclasses of the class $\sN^Q$. We begin by proving an auxiliary result of algebraic nature that will become useful later on.
\begin{lemma}\label{l-20}
Let $b=Q^2+a^2-1$. The following statements are true for all non-zero $Q\in \dR$ and $a>0:$
\begin{enumerate}
  \item if $0<a<1$, then
\begin{equation}\label{e-78-ineq-1}
\frac{a(b-\sqrt{b^2+4Q^2})^2+4aQ^2}{\left(b-\sqrt{b^2+4Q^2}-2Q^2\right)^2+4a^2 Q^2}<1;
\end{equation}
  \item if $a>1$, then
\begin{equation}\label{e-79-ineq-2}
\frac{a(b+\sqrt{b^2+4Q^2})^2+4aQ^2}{\left(b+\sqrt{b^2+4Q^2}-2Q^2\right)^2+4a^2 Q^2}>1.
\end{equation}
\end{enumerate}
Moreover, for every $a>0$ and $Q\ne0$
\begin{equation}\label{e-85-recipr}
    \frac{a(b-\sqrt{b^2+4Q^2})^2+4aQ^2}{\left(b-\sqrt{b^2+4Q^2}-2Q^2\right)^2+4a^2 Q^2}\cdot\frac{a(b+\sqrt{b^2+4Q^2})^2+4aQ^2}{\left(b+\sqrt{b^2+4Q^2}-2Q^2\right)^2+4a^2 Q^2}=1.
\end{equation}
\end{lemma}
\begin{proof}
(1) Let $0<a<1$.
Since the denominator in the left hand side of \eqref{e-78-ineq-1} is always positive, inequality \eqref{e-78-ineq-1} is equivalent to
$$
4aQ^2-4a^2Q^2<\left(b-\sqrt{b^2+4Q^2}-2Q^2\right)^2-a(b-\sqrt{b^2+4Q^2})^2,
$$
or
\begin{equation}\label{e-85-ineq}
    4a(1-a)<\frac{\left(b-\sqrt{b^2+4Q^2}-2Q^2\right)^2-a(b-\sqrt{b^2+4Q^2})^2}{Q^2}.
\end{equation}
In order to prove \eqref{e-78-ineq-1} we are going to show that the absolute minimum of the right hand side of \eqref{e-85-ineq} is greater than $4a(1-a)$. For further convenience we set
$$
z=Q^2\textrm{ and } d(z)=b-\sqrt{b^2+4Q^2}=z+a^2-1-\sqrt{(z+a^2-1)^2+4z}.
$$
Then the right hand side of \eqref{e-85-ineq} can be written as a function of $z$, $(z>0)$ as
$$
f(z)=\frac{(d(z)-2z)^2-ad^2(z)}{z}=\frac{(1-a)d^2(z)-4zd(z)+4z^2}{z}.
$$
Taking the derivative with respect to $z$ yields
$$
f'(z)=\frac{(2-2a)zd(z)d'(z)-(1-a)d^2(z)-4z^2d'(z)+4z^2}{z^2}.
$$
Taking into account that $b=z+a^2-1$ and
$$
d'(z)=1-\frac{(z+a^2-1)+2}{\sqrt{(z+a^2-1)^2+4z}}=1-\frac{b+2}{b-d}=\frac{d(z)+2}{d(z)-b},
$$
we set $f'(z)=0$ and solve for $z>0$. This yields a point of the absolute minimum at $z_0=a^2+a+1$ such that
$$
f(z_0)=(16-16a)(a^2+a+1).
$$
Clearly, $f(z_0)>4a(1-a)$ for any $0<a<1$. This confirms \eqref{e-85-ineq} for any real $Q\ne0$ and hence \eqref{e-78-ineq-1} is proved.

(2) Let $a>1$.   Following  the first part  we write an equivalent to \eqref{e-79-ineq-2} inequalities:
inequality \eqref{e-79-ineq-2} is equivalent to
$$
4a^2 Q^2-4a Q^2<a(b+\sqrt{b^2+4Q^2})^2-\left(b+\sqrt{b^2+4Q^2}-2Q^2\right)^2,
$$
or
\begin{equation}\label{e-86-ineq}
    4a(a-1)<\frac{a(b+\sqrt{b^2+4Q^2})^2-\left(b+\sqrt{b^2+4Q^2}-2Q^2\right)^2}{Q^2}.
\end{equation}
Similarly to part (1) we set
$$
z=Q^2\textrm{ and } d(z)=b+\sqrt{b^2+4Q^2}=z+a^2-1+\sqrt{(z+a^2-1)^2+4z}.
$$
Then the right hand side of \eqref{e-86-ineq} can be written as a function of $z$, $(z>0)$ as
$$
f(z)=\frac{ad^2(z)-(d(z)-2z)^2}{z}=\frac{(a-1)d^2(z)+4zd(z)-4z^2}{z}.
$$
Following part (1) steps we take the derivative with respect to $z$ yielding
$$
f'(z)=\frac{(2a-2)zd(z)d'(z)-(a-1)d^2(z)+4z^2d'(z)-4z^2}{z^2}.
$$
Taking into account that $b=z+a^2-1$ and
$$
d'(z)=1+\frac{(z+a^2-1)+2}{\sqrt{(z+a^2-1)^2+4z}}=1+\frac{b+2}{d-b}=\frac{d(z)+2}{d(z)-b},
$$
we set $f'(z)=0$ and solve for $z>0$. As in part (1) we get a point of the absolute minimum at $z_0=a^2+a+1$ such that
$$
f(z_0)=(8a-8)(a^2+a+1).
$$
Clearly, $f(z_0)>4a(a-1)$ for any $a>1$. This confirms \eqref{e-86-ineq} for any real $Q\ne0$ and hence \eqref{e-79-ineq-2} is proved.

Formula \eqref{e-85-recipr} is checked by direct computations.
\end{proof}

Below we state and prove a realization theorem for the class $\sM_{\kappa_0}^Q$.
\begin{theorem}\label{t-18}
Let  $V(z)$ belong to the  class $\sM^Q_{\kappa_0}$ and have integral representation \eqref{hernev-real}. Let $0<a<1$ be a normalization parameter related to $V(z)$ via \eqref{e-66-L}. Then $V(z)$ can be realized as the impedance function $V_{\Theta}(z)$ of a minimal L-system $\Theta$  of the form \eqref{e-62} with  the main operator $T$ whose von Neumann's parameter  $\kappa$ ($0<\kappa<1$) is determined  by the formula
  \begin{equation}\label{e-53-kappa-prime}
    \kappa=\frac{\left(b-2Q^2-\sqrt{b^2+4Q^2}\right)^2-a\left(b-\sqrt{b^2+4Q^2}\right)^2+4Q^2a(a-1)}{\left(b-2Q^2-\sqrt{b^2+4Q^2}\right)^2+a\left(b-\sqrt{b^2+4Q^2}\right)^2+4Q^2a(a+1)},
 \end{equation}
 where $Q\ne0$ and
 \begin{equation}\label{e-78-b}
    b=Q^2+a^2-1.
\end{equation}
Moreover, the quasi-kernel $\hat A$ of $\RE\bA$ of the realizing L-system $\Theta$ is defined by \eqref{DOMHAT} with 
\begin{equation}\label{e-75-U}
    U=\frac{(a+Qi)(1-\kappa^2)-1-\kappa^2}{2\kappa},\quad Q\ne0.
\end{equation}
 \end{theorem}
\begin{proof}
As before we consider function $V_{\Theta_\alpha}(z)$ related to our function $V(z)\in \sM^Q_{\kappa_0}$ by \eqref{e-54-frac}
for some $\alpha\in[0,\pi)$. It is known (see \cite[Theorem 8.3.2]{ABT}) that $V(z)$ also falls into the class $\sN^Q$ and hence admits integral representation \eqref{hernev-real}. Applying  Lemma \ref{l-12} we conclude that the measure $\mu_\alpha$ in the integral representation \eqref{hernev-real} of $V(z)$ is unbounded.
By direct substitution one gets $V(i)=Q+ia$ and hence
$$
\begin{aligned}
V_{\Theta_\alpha}(i)&=\frac{\cos\alpha+(\sin\alpha)V(i)}{\sin\alpha-(\cos\alpha)V(i)}=\frac{\cos\alpha+(\sin\alpha)(Q+ia)}{\sin\alpha-(\cos\alpha)(Q+ia)}\\
&=\frac{(\cos\alpha+Q\sin\alpha)+ia\sin\alpha}{(\sin\alpha-Q\cos\alpha)-ia\cos\alpha}=\frac{(1/2)(1-Q^2-a^2)\sin2\alpha-Q\cos2\alpha}{(\sin\alpha-Q\sin\alpha)^2+a^2\cos^2\alpha}\\
&+i\,\frac{a}{(\sin\alpha-Q\cos\alpha)^2+a^2\cos^2\alpha}=Q_\alpha+i\int_{\dR}\frac{d\mu_\alpha(\lambda)}{1+\lambda^2}=Q_\alpha+ia_\alpha,
\end{aligned}
$$
where $Q_\alpha$  and  $\mu_\alpha$  are the elements of integral representation \eqref{hernev-real} of the function $V_{\Theta_\alpha}(z)$ and $a_\alpha=\int_{\dR}\frac{d\mu_\alpha(\lambda)}{1+\lambda^2}$. Thus,
\begin{equation}\label{e-55-q'}
    Q_\alpha=\frac{(1/2)(1-Q^2-a^2)\sin2\alpha-Q\cos2\alpha}{(\sin\alpha-Q\cos\alpha)^2+a^2\cos^2\alpha},
\end{equation}
and
\begin{equation}\label{e-56-q-int'}
   \int_{\dR}\frac{d\mu_\alpha(\lambda)}{1+\lambda^2}=\frac{a}{(\sin\alpha-Q\cos\alpha)^2+a^2\cos^2\alpha}.
\end{equation}
Both formulas \eqref{e-55-q'} and \eqref{e-56-q-int'} hold true for any value of $\alpha\in[0,\pi)$ and if we set $Q_\alpha=0$ in \eqref{e-55-q}, then we obtain that either $Q=0$ (we discard this case since it takes us back to the class $\sM_{\kappa_0}$) or
\begin{equation}\label{e-74-q-alpha}
    \tan 2\alpha=\frac{2Q}{1-Q^2-a^2}.
\end{equation}
Using the formula for the double angle
$$
\tan 2\alpha=\frac{2\tan\alpha}{1-\tan^2\alpha},
$$
we obtain the quadratic equation
$$
Q\tan^2\alpha-(Q^2+a^2-1)\tan\alpha-Q=0,
$$
whose solutions are
\begin{equation}\label{e-76-tan}
\tan\alpha=\frac{Q^2+a^2-1\pm\sqrt{(Q^2+a^2-1)^2+4Q^2}}{2Q}.
\end{equation}
For further convenience in calculations we will use $b=Q^2+a^2-1$ as defined in \eqref{e-78-b}.
Utilizing \eqref{e-76-tan} with \eqref{e-78-b} and basic trig relations we get
\begin{equation}\label{e-87-cos-tan}
\begin{aligned}
\tan\alpha&=\frac{b\pm\sqrt{b^2+4Q^2}}{2Q},\quad
\cos^2\alpha=\frac{4Q^2}{\left(b\pm\sqrt{b^2+4Q^2}\right)^2+4Q^2}.\\
\end{aligned}
\end{equation}
We use the above formulas for $\tan\alpha$ and $\cos^2\alpha$ to modify the denominator in \eqref{e-56-q-int'} and express it in terms of $Q$.
\begin{equation}\label{e-87-denom}
\begin{aligned}
(\sin\alpha&-Q\cos\alpha)^2+a^2\cos^2\alpha=[\cos\alpha(\tan\alpha-Q)]^2+a^2\cos^2\alpha\\
&=\cos^2\alpha[(\tan\alpha-Q)^2+a^2]\\
&=\frac{4Q^2}{\left(b\pm\sqrt{b^2+4Q^2}\right)^2+4Q^2}\left[\left(\frac{b\pm\sqrt{b^2+4Q^2}}{2Q}-Q\right)^2+a^2\right]\\
&=\frac{4Q^2}{\left(b\pm\sqrt{b^2+4Q^2}\right)^2+4Q^2}\left[\frac{\left(b\pm\sqrt{b^2+4Q^2}-2Q^2\right)^2}{4Q^2}+a^2\right]\\
&=\frac{4Q^2}{\left(b\pm\sqrt{b^2+4Q^2}\right)^2+4Q^2}\cdot\frac{\left(b\pm\sqrt{b^2+4Q^2}-2Q^2\right)^2+4Q^2a^2}{4Q^2}\\
&=\frac{\left(b\pm\sqrt{b^2+4Q^2}-2Q^2\right)^2+4Q^2a^2}{\left(b\pm\sqrt{b^2+4Q^2}\right)^2+4Q^2}.
\end{aligned}
\end{equation}
Case 1. Let us assume that $\alpha\in[0,\pi/2)$ and $Q>0$. Then  $\tan\alpha\ge0$ and formulas \eqref{e-87-cos-tan} become
\begin{equation}\label{e-83-cos-sin}
    \tan\alpha=\frac{b+\sqrt{b^2+4Q^2}}{2Q},\quad \cos^2\alpha=\frac{4Q^2}{\left(b+\sqrt{b^2+4Q^2}\right)^2+4Q^2}.
\end{equation}
Indeed, since $|b|\le \sqrt{b^2+4Q^2}$ for all $Q\in\dR$, then $\tan\alpha\ge0$ only when $(+)$ sign is chosen in \eqref{e-76-tan} and \eqref{e-87-cos-tan}. Then the  choice of signs in \eqref{e-83-cos-sin} is unique and follows from the positivity of  $\tan\alpha$ in \eqref{e-87-cos-tan}.
Substituting these values in \eqref{e-56-q-int'} we use \eqref{e-87-denom} to get
\begin{equation}\label{e-77-a-alpha}
 a_{\alpha}=   \int_{\dR}\frac{d\mu_{\alpha}(\lambda)}{1+\lambda^2}=\frac{a(b+\sqrt{b^2+4Q^2})^2+4aQ^2}{\left(b+\sqrt{b^2+4Q^2}-2Q^2\right)^2+4a^2Q^2},\quad Q>0,\;\alpha\in\left[0,\frac{\pi}{2}\right).
\end{equation}
Case 2. Let us assume that $\alpha\in[0,\pi/2)$ and $Q<0$. Then  $\tan\alpha\ge0$ and formulas \eqref{e-87-cos-tan} become
\begin{equation}\label{e-85-cos-sin}
    \tan\alpha=\frac{b-\sqrt{b^2+4Q^2}}{2Q},\quad \cos^2\alpha=\frac{4Q^2}{\left(b-\sqrt{b^2+4Q^2}\right)^2+4Q^2}.
\end{equation}
As before, since $|b|\le \sqrt{b^2+4Q^2}$ for all $Q\in\dR$, then $\tan\alpha\ge0$ only when $(-)$ sign is chosen in \eqref{e-76-tan} and \eqref{e-87-cos-tan}. Again the  choice of signs in \eqref{e-83-cos-sin} is unique and follows from the  positivity of $\tan\alpha$.
Substituting these values in \eqref{e-56-q-int'} with the help of \eqref{e-87-denom} yields
\begin{equation}\label{e-86-a-alpha}
 a_{\alpha}=  \frac{a(b-\sqrt{b^2+4Q^2})^2+4aQ^2}{\left(b-\sqrt{b^2+4Q^2}-2Q^2\right)^2+4a^2Q^2},\quad  Q<0,\;\alpha\in\left[0,\frac{\pi}{2}\right).
\end{equation}
Case 3. Now let $\alpha\in(\pi/2,\pi]$ and $Q>0$. Then  $\tan\alpha\le0$ and similar to the above analysis leads to
\begin{equation*}\label{e-87-cos-sin}
    \tan\alpha=\frac{b-\sqrt{b^2+4Q^2}}{2Q},\quad \cos^2\alpha=\frac{4Q^2}{\left(b-\sqrt{b^2+4Q^2}\right)^2+4Q^2}.
\end{equation*}
Consequently,
\begin{equation}\label{e-88-a-alpha}
 a_{\alpha}=  \frac{a(b-\sqrt{b^2+4Q^2})^2+4aQ^2}{\left(b-\sqrt{b^2+4Q^2}-2Q^2\right)^2+4a^2Q^2},\quad  Q>0,\;\alpha\in\left(\frac{\pi}{2},\pi\right].
\end{equation}
Case 4. Now let $\alpha\in(\pi/2,\pi]$ and $Q<0$. Then  $\tan\alpha\le0$ and hence
\begin{equation*}\label{e-89-cos-sin}
    \tan\alpha=\frac{b+\sqrt{b^2+4Q^2}}{2Q},\quad \cos^2\alpha=\frac{4Q^2}{\left(b+\sqrt{b^2+4Q^2}\right)^2+4Q^2}.
\end{equation*}
Consequently,
\begin{equation}\label{e-90-a-alpha}
 a_{\alpha}=   \frac{a(b+\sqrt{b^2+4Q^2})^2+4aQ^2}{\left(b+\sqrt{b^2+4Q^2}-2Q^2\right)^2+4a^2Q^2},\quad  Q<0,\;\alpha\in\left(\frac{\pi}{2},\pi\right].
\end{equation}
Therefore, for every given value of $Q\in\dR$, the value of $a_\alpha$ is found either with \eqref{e-88-a-alpha} or with \eqref{e-90-a-alpha}. But according to Lemma \ref{l-20}  and formula \eqref{e-85-recipr} we know that expressions in \eqref{e-88-a-alpha} and \eqref{e-90-a-alpha} are reciprocals of each other for all real $Q\ne0$. It is also worth noting that for $0<a<1$
$$
\lim_{Q\to 0}\frac{a(b-\sqrt{b^2+4Q^2})^2+4aQ^2}{\left(b-\sqrt{b^2+4Q^2}-2Q^2\right)^2+4a^2Q^2}=a,
$$
while
$$
\lim_{Q\to 0}\frac{a(b+\sqrt{b^2+4Q^2})^2+4aQ^2}{\left(b+\sqrt{b^2+4Q^2}-2Q^2\right)^2+4a^2Q^2}=\frac{1}{a}.
$$
As it was shown in \cite[Lemma 5.1 and Theorem 5.2]{BMkT-2} there are only two possible values of $a_\alpha$ such that $Q_\alpha=0$ and both cases lead to the same value of $\kappa$. Consequently, either of the formulas \eqref{e-88-a-alpha} and \eqref{e-90-a-alpha} can be used to find $\kappa$. Thus, without loss of generality, we set
\begin{equation}\label{e-80-a-alpha}
 a_{\alpha}=   \frac{a(b-\sqrt{b^2+4Q^2})^2+4aQ^2}{\left(b-\sqrt{b^2+4Q^2}-2Q^2\right)^2+4a^2Q^2}.
\end{equation}
Furthermore, examining \eqref{e-80-a-alpha} and applying Lemma \ref{t-18} show that for any $Q\in(-\infty,0)\cup(0,+\infty)$ and $0<a<1$ we have
$$
a_\alpha<1,
$$
implying that $V_{\Theta_\alpha}(z)\in\sM_\kappa$ for  some value $\kappa\in(0,1)$ such that
(see \eqref{e-45-kappa-1})
$$
\kappa=\frac{1-a_\alpha}{1+a_\alpha}.
$$
Combining this with \eqref{e-80-a-alpha} yields
$$
\kappa(Q,a)=\frac{\left(b-2Q^2-\sqrt{b^2+4Q^2}\right)^2-a\left(b-\sqrt{b^2+4Q^2}\right)^2+4Q^2a(a-1)}{\left(b-2Q^2-\sqrt{b^2+4Q^2}\right)^2+a\left(b-\sqrt{b^2+4Q^2}\right)^2+4Q^2a(a+1)},
$$
which proves \eqref{e-53-kappa-prime}. The graph of $\kappa$ as a function of $Q$ is shown on Figure \ref{fig-2}. Note that the vertex of the graph is located at the value of $\kappa=\kappa_0=\frac{1-a}{1+a}$. Moreover, if $a\rightarrow 1$, then $\kappa_0\rightarrow 0$ as indicated on the picture.
\begin{figure}
  \begin{center}
  \includegraphics[width=110mm]{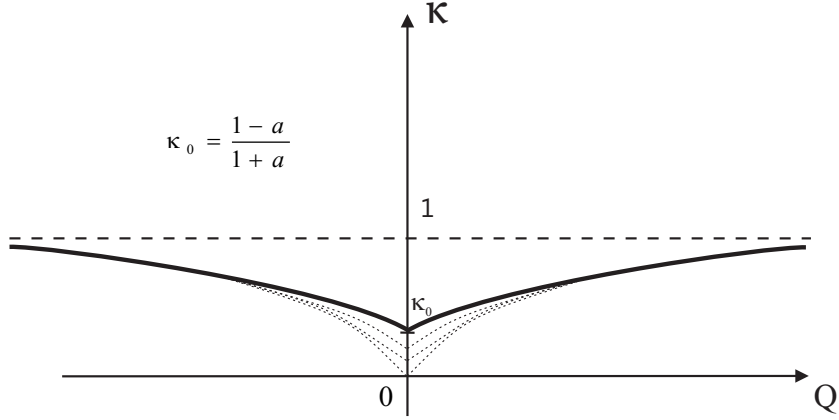}
  \caption{Class $\sM^Q_\kappa$: $\kappa$ as a function of $Q$ for $0<a<1$}\label{fig-2}
  \end{center}
\end{figure}

 Applying Theorems  \ref{t-14}  we can realize $V_{\Theta_{\alpha}}(z)$ by a minimal model L-system $\Theta_{{\alpha}}$ of the form  \eqref{e-59'}  described in details in Appendix \ref{A1}.  We remind the reader that in this case the entire construction of $\Theta_{{\alpha}}$  depends on the measure $\mu_\alpha$ in the integral representation \eqref{hernev} of our function $V_{\Theta_{{\alpha}}}(z)$. Also $\Theta_{{\alpha}}$ is based upon the model triple $(\dot\cB, \whB ,\cB)$ of the form \eqref{nacha1-ap}--\eqref{nacha3-ap} with $\kappa$ given by \eqref{e-53-kappa-prime} and satisfies the conditions of Hypothesis \ref{setup}. Moreover,  the value of $\kappa$ in \eqref{e-53-kappa-prime} used in this constructions is the von Neumann parameter of the main operator $\whB$ with respect to deficiency vectors of the form \eqref{e-52-def} normalized in the space $L^2(\bbR;d\mu_\alpha)$.
Consider $W_{\Theta_{\alpha}}(z)$, the transfer function of $\Theta_{{\alpha}}$, that is related to $V_{\Theta_{{\alpha}}}(z)$ via \eqref{e6-3-6} with $J=1$, and the function
$$
W(z)=\frac{1-i V(z)}{1+i V(z)}.
$$
Since $V_{\Theta_{{\alpha}}}(z)$ and $V(z)$ are related by \eqref{e-54-frac}, we have
\begin{equation}\label{e-61-Junitary'}
W(z)=(-e^{-2i\cdot{\alpha}})W_{\Theta_{{\alpha}}}(z).
\end{equation}
Relation \eqref{e-61-Junitary'} above allows us to apply the Theorem on constant $J$-unitary factor (see \cite[Theorem 8.2.3]{ABT}, \cite{ArTs03}) that guarantees the existence of another L-system $\Theta$ with the same main operator $\whB$ as $\Theta_{\alpha}$ and such that  $V(z)=V_\Theta(z)$, ($z\in\dC_\pm$).
In order to find the quasi-kernel $\cB_U$ of the real part of the state-space operator in  $\Theta$ and its unimodular parameter $U$ we simply repeat the argument in the end of the proof of Theorem \ref{t-18-M-q}. To find $U$ and show that it matches \eqref{e-75-U} we apply formula \eqref{e-216-ls} of Appendix \ref{A2}. Adapting \eqref{e-70-sys} to our case and taking into account that $V(-i)=Q-ai$  we get
\begin{equation}\label{e-93-sys}
     \left\{
      \begin{array}{ll}
       \frac{\overline{\kappa^2 U+2\kappa+U}}{\sqrt2|1+\kappa U|\sqrt{1-\kappa^2}}\,b(-i)=Q-ai & \hbox{} \\
&\\
        \frac{i\sqrt{1-\kappa^2}\bar U}{\sqrt2|1+\kappa U|}b(-i)=1. & \hbox{}
      \end{array}
    \right.
\end{equation}
Eliminating $b(-i)$ and solving \eqref{e-93-sys} for $U$ yields
$$
U=\frac{(a+Qi)(1-\kappa^2)-1-\kappa^2}{2\kappa}.
$$
that matches \eqref{e-75-U}.


Now we are ready to give a description of the model L-system $\Theta$ of the form \eqref{e-59} that realizes $V(z)$. This L-system is based upon the model triple $(\dot \cB,   \whB ,\cB_U)$  that relies on the measure $\mu_{\alpha}$ in  integral representation \eqref{hernev-real} of $V(z)$ and has the state space $\calH_+ \subset L^2(\bbR;d\mu_{\alpha})\subset\calH_-$. The operators $\whB$ and $\cB_U$ are parameterized by the von Neumann parameters $\kappa$ and $U$ given by \eqref{e-53-kappa-prime} and \eqref{e-75-U}, respectively. We also note that the realizing   L-system $\Theta$ shares the state space, symmetric, and main operators with the L-system $\Theta_{{\alpha}}$ but only differs by the quasi-kernel of the real part of the state-space operator.   The detailed construction of  $\Theta$ is described in Appendix \ref{A1}. This model L-system $\Theta$ realizes our function $V(z)$.
\end{proof}

\begin{remark}\label{r-19}
Both von Neumann's parameterizations of operators $\whB$ and $\cB_U$ in the proof of Theorem \ref{t-18}  are written with respect to the deficiency vectors  of the symmetric operator $\dot \cB$ of the form \eqref{e-52-def} in the space $L^2(\bbR;d\mu_{\alpha})$. Moreover, the result of Theorem \ref{t-18} also holds true for $Q=0$ as it directly follows from Theorem \ref{t-12-new}. In this case formulas \eqref{e-53-kappa-prime}  and \eqref{e-75-U} get reduced to the unperturbed case and yield $\kappa(0)=\kappa_0$ and $U(0)=-1$.
\end{remark}

\begin{corollary}\label{c-21} Let $V_1(z)=Q+V_{10}(z)$ and $V_2(z)=Q+V_{20}(z)$ be two functions such that $V_{10}(z)$ and $V_{20}(z)$ $(V_{10}(z)\ne V_{20}(z))$  belong to the generalized Donoghue class $\sM_{\kappa_0}$. Then both $V_1(z)$ and $V_2(z)$ can be realized as the impedance functions  of two not bi-unitarily equivalent L-systems of the form \eqref{e-62} whose  main operators have the same  von Neumann parameter  $\kappa$ determined by \eqref{e-53-kappa-prime} and the quasi-kernels of the real parts of state-space operators defined via the same $U$  of the form \eqref{e-75-U}.
\end{corollary}
\begin{proof}
Both functions  $V_1(z)$ and $V_2(z)$ share the same value of $a=\frac{1-\kappa_0}{1+\kappa_0}$. The rest follows from
 from Theorem \ref{t-18} and the structure of formulas \eqref{e-53-kappa-prime} and \eqref{e-75-U}. The realizing L-systems corresponding to $V_1(z)$ and $V_2(z)$ have different state-spaces and symmetric operators since they are constructed based on different measures in integral representation \eqref{hernev-0} of $V_{10}(z)$ and $V_{20}(z)$. The construction of these model realizing L-systems was described in details in the proof of Theorem \ref{t-18} and in Appendix \ref{A1}. The von Neumann parameters $\kappa$ and $U$, however, are the same for both L-systems and correspond to the respective basis of deficiency vectors \eqref{e-52-def} in a different space for each case.

Also, the realizing L-systems corresponding to different functions $V(z)\in\sM^Q_{\kappa_0}$ are not bi-unitarily equivalent (otherwise their impedance functions would match).
\end{proof}

Now we consider realizations for class $\sM_{\kappa_0}^{-1,Q}$.
\begin{theorem}\label{t-20}
Let  $V(z)$ belong to the  class $\sM^{-1,Q}_{\kappa_0}$ and have integral representation \eqref{hernev-real}. Let $a>1$ be a normalization parameter related to $V(z)$ via \eqref{e-66-L}. Then $V(z)$ can be realized as the impedance function $V_{\Theta}(z)$ of a minimal L-system $\Theta$  of the form \eqref{e-62} with  the main operator $T$ whose von Neumann's parameter  $\kappa$ is determined by the formula
  \begin{equation}\label{e-85-kappa-prime}
    \kappa=\frac{a\left(b+\sqrt{b^2+4Q^2}\right)^2-\left(b-2Q^2+\sqrt{b^2+4Q^2}\right)^2-4Q^2a(a-1)}{\left(b-2Q^2+\sqrt{b^2+4Q^2}\right)^2+a\left(b+\sqrt{b^2+4Q^2}\right)^2+4Q^2a(a+1)}, \quad Q\ne0,
 \end{equation}
 where
 \begin{equation}\label{e-86-b}
    b=Q^2+a^2-1.
\end{equation}
Moreover, the quasi-kernel $\hat A$ of $\RE\bA$ of the realizing L-system $\Theta$ is defined by \eqref{DOMHAT} with 
\begin{equation}\label{e-87-U}
    U=\frac{(a+Qi)(1-\kappa^2)-1-\kappa^2}{2\kappa},\quad Q\ne0,
\end{equation}
where $\kappa$ and $b$ are  given by \eqref{e-85-kappa-prime} and \eqref{e-86-b}, respectively\footnote{Even though the formulas \eqref{e-75-U} and \eqref{e-87-U} for $U$ formally look the same they produce different values due to the fact that they involve different values for $\kappa$ given by formulas \eqref{e-53-kappa-prime} and \eqref{e-85-kappa-prime}, respectively.}.
 \end{theorem}
\begin{proof}
As before we consider function $V_{\Theta_\alpha}(z)$ related to our function $V(z)\in \sM^{-1,Q}_{\kappa_0}$ by \eqref{e-54-frac} for some $\alpha\in[0,\pi)$.
Using the same reasoning as in the proof of Theorem \ref{t-18} we have that $V_{\Theta_\alpha}(z)$
 falls into the class $\sN^Q$ and hence admits integral representation \eqref{hernev-real} with an unbounded measure $\mu_\alpha$.
As earlier we use  direct substitution to  get $V(i)=Q+ia$ and obtain
$$
V_{\Theta_\alpha}(i)=Q_\alpha+ia_\alpha,
$$
where $Q_\alpha$  is given by \eqref{e-55-q'}  and $a_\alpha$ by \eqref{e-56-q-int'}. Formally repeating the argument of the proof of Theorem \ref{t-18} we again obtain
\eqref{e-77-a-alpha} or  \eqref{e-86-a-alpha}, that is
\begin{equation*}
 a_{\alpha}=   \int_{\dR}\frac{d\mu_{\alpha}(\lambda)}{1+\lambda^2}=\frac{a(b\pm\sqrt{b^2+4Q^2})^2+4aQ^2}{\left(b\pm\sqrt{b^2+4Q^2}-2Q^2\right)^2+4a^2 Q^2}.
\end{equation*}
This time, however, we know that since $V(z)\in \sM^{-1,Q}_{\kappa_0}$ we have that $a>1$.
It is worth mentioning that for $a>1$
$$
\lim_{Q\to 0}\frac{a(b-\sqrt{b^2+4Q^2})^2+4aQ^2}{\left(b-\sqrt{b^2+4Q^2}-2Q^2\right)^2+4a^2 Q^2}=\frac{1}{a},
$$
while
$$
\lim_{Q\to 0}\frac{a(b+\sqrt{b^2+4Q^2})^2+4aQ^2}{\left(b+\sqrt{b^2+4Q^2}-2Q^2\right)^2+4aQ^2}=a.
$$
Applying similar to the proof of Theorem \ref{t-18} reasoning we choose the second formula and set
\begin{equation}\label{e-98-a-alpha}
 a_{\alpha}=   \frac{a(b+\sqrt{b^2+4Q^2})^2+4aQ^2}{\left(b+\sqrt{b^2+4Q^2}-2Q^2\right)^2+4a^2 Q^2}.
\end{equation}
Furthermore, examining \eqref{e-98-a-alpha} and applying Lemma \ref{l-20} show that for any $Q\in(-\infty,0)\cup(0,+\infty)$ and $a>1$ we have
$$
a_\alpha>1,
$$
implying that $V_{\Theta_\alpha}(z)\in\sM_\kappa^{-1}$ for  some value $\kappa\in(0,1)$ such that
(see \eqref{e-45-kappa-1})
$$
\kappa=\frac{a_\alpha-1}{1+a_\alpha}.
$$
Combining this with \eqref{e-80-a-alpha} yields
$$
\kappa(Q,a)=\frac{a\left(b+\sqrt{b^2+4Q^2}\right)^2-\left(b-2Q^2+\sqrt{b^2+4Q^2}\right)^2-4Q^2a(a-1)}{\left(b-2Q^2+\sqrt{b^2+4Q^2}\right)^2+a\left(b+\sqrt{b^2+4Q^2}\right)^2+4Q^2a(a+1)},
$$
which proves \eqref{e-85-kappa-prime}. The graph of $\kappa$ as a function of $Q$ is shown on Figure \ref{fig-3}. Note that the vertex of the graph is located at the value of $\kappa=\kappa_0=\frac{a-1}{1+a}$. Moreover, if $a\rightarrow\infty$, then $\kappa_0\rightarrow 1$ as indicated on the picture.
\begin{figure}
  \begin{center}
  \includegraphics[width=110mm]{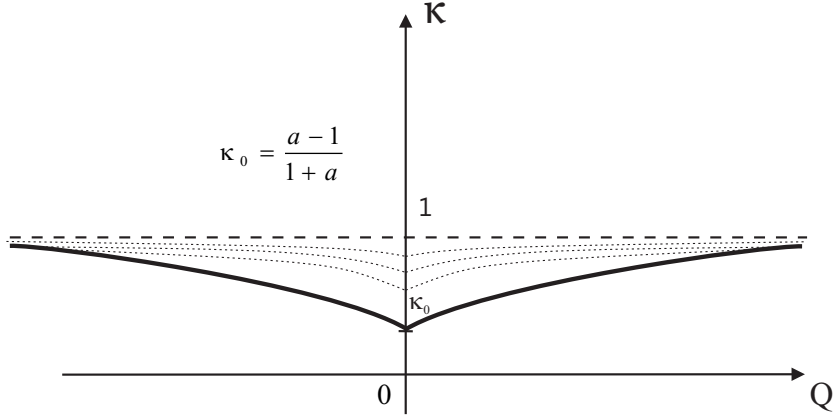}
  \caption{Class $\sM^{-1,Q}_\kappa$: $\kappa$ as a function of $Q$ for $a>1$}\label{fig-3}
  \end{center}
\end{figure}

 Applying Theorem  \ref{t-14-new}  we can realize $V_{\Theta_{\alpha}}(z)$ by a minimal model L-system $\Theta_{{\alpha}}$ of the form  \eqref{e-59'}  described in details in Appendix \ref{A1}.  As we did before, we note that in this case the entire construction of $\Theta_{{\alpha}}$  depends on the measure $\mu_\alpha$ in the integral representation \eqref{hernev} of our function $V_{\Theta_{{\alpha}}}(z)$. Also $\Theta_{{\alpha}}$ is based upon the model triple $(\dot\cB, \whB ,\cB)$ of the form \eqref{nacha1-ap}--\eqref{nacha3-ap} with $\kappa$ given by \eqref{e-85-kappa-prime} and satisfies the conditions of Hypothesis \ref{setup-1}. Moreover,  the value of $\kappa$ in \eqref{e-85-kappa-prime}  used in this constructions is the von Neumann parameter of the main operator $\whB$ with respect to deficiency vectors of the form \eqref{e-52-def} normalized  in the space $L^2(\bbR;d\mu_\alpha)$.
Consider $W_{\Theta_{\alpha}}(z)$, the transfer function of $\Theta_{{\alpha}}$, that is related to $V_{\Theta_{{\alpha}}}(z)$ via \eqref{e6-3-6} with $J=1$, and the function
$$
W(z)=\frac{1-i V(z)}{1+i V(z)}.
$$
Since $V_{\Theta_{{\alpha}}}(z)$ and $V(z)$ are related by \eqref{e-54-frac}, we have
\begin{equation}\label{e-97-Junitary}
W(z)=(-e^{-2i\cdot{\alpha}})W_{\Theta_{{\alpha}}}(z).
\end{equation}
Relation \eqref{e-97-Junitary} above allows us to apply the Theorem on constant $J$-unitary factor (see \cite[Theorem 8.2.3]{ABT}, \cite{ArTs03}) that guarantees the existence of another L-system $\Theta$ with the same main operator $\whB$ as $\Theta_{\alpha}$ and such that $V(z)=V_\Theta(z)$, ($z\in\dC_\pm$).
In order to find the quasi-kernel $\cB_U$ of the real part of the state-space operator in  $\Theta$ and its unimodular parameter $U$ we simply repeat the argument in the end of the proof of Theorem \ref{t-18}.

At this point we can describe the model L-system $\Theta$ of the form \eqref{e-59} that realizes $V(z)$. This L-system is based upon the model triple $(\dot \cB,   \whB ,\cB_U)$  that relies on the measure $\mu_{\alpha}$ in integral representation \eqref{hernev-real} of $V(z)$ and has the state space $\calH_+ \subset L^2(\bbR;d\mu_{\alpha})\subset\calH_-$. The operators $\whB$ and $\cB_U$ are parameterized by the von Neumann parameters $\kappa$ and $U$ given by \eqref{e-85-kappa-prime}  and \eqref{e-87-U}, respectively. Our realizing   L-system $\Theta$ shares the state space, symmetric, and main operators with the L-system $\Theta_{{\alpha}}$ but only differs by the quasi-kernel of the real part of the state-space operator.  The detailed construction of  $\Theta$ is described in Appendix \ref{A1}. This model L-system $\Theta$ realizes our function $V(z)$.
\end{proof}

\begin{remark}\label{r-21}
Both von Neumann's parameterizations of operators $\whB$ and $\cB_U$ in the proof of Theorem \ref{t-20}  are written with respect to the deficiency vectors  of the symmetric operator $\dot \cB$ of the form \eqref{e-52-def} in the space $L^2(\bbR;d\mu_{\alpha})$. Moreover, the result of Theorem \ref{t-20} also holds true for $Q=0$ as it directly follows from Theorem \ref{t-13-new}. In this case formulas \eqref{e-85-kappa-prime}  and \eqref{e-87-U} get reduced to the unperturbed case and yield $\kappa(0)=\kappa_0$ and $U(0)=1$.
\end{remark}

\begin{corollary}\label{c-23} Let $V_1(z)=Q+V_{10}(z)$ and $V_2(z)=Q+V_{20}(z)$ be two functions such that $V_{10}(z)$ and $V_{20}(z)$ $(V_{10}(z)\ne V_{20}(z))$ belong to the generalized Donoghue class $\sM_{\kappa_0}^{-1}$. Then both $V_1(z)$ and $V_2(z)$ can be realized as the impedance functions  of two not bi-unitarily equivalent L-systems of the form \eqref{e-62} whose main operators have the same  von Neumann parameter  $\kappa$ determined by \eqref{e-85-kappa-prime}  and the quasi-kernels of the real parts of state-space operators defined via the same $U$  of the form \eqref{e-87-U}.
\end{corollary}
\begin{proof}
Both functions  $V_1(z)$ and $V_2(z)$ share the same value of $a=\frac{1+\kappa_0}{1-\kappa_0}$. The rest follows from   Theorem \ref{t-20} and the structure of formulas \eqref{e-85-kappa-prime} and \eqref{e-87-U}. The realizing L-systems corresponding to $V_1(z)$ and $V_2(z)$ have different state-spaces and symmetric operators since they are constructed based on different measures in integral representation \eqref{hernev-0} of $V_{10}(z)$ and $V_{20}(z)$. The construction of these model realizing L-systems was described in details in the proof of Theorem \ref{t-20} and in Appendix \ref{A1}. The von Neumann parameters $\kappa$ and $U$, however, are the same for both L-systems and correspond to the respective basis of deficiency vectors \eqref{e-52-def} in a different space for each case.

Also, the realizing L-systems corresponding to different functions $V(z)\in\sM^Q_{\kappa_0}$ are not bi-unitarily equivalent (otherwise their impedance functions would match).
\end{proof}

As we did  in the end of  Section \ref{s5}, we point out
 that the von Neumann parameter $\kappa$ of an L-system  realizing a function $V(z)\in\sM^Q_{\kappa_0}$ or $V(z)\in\sM^{-1,Q}_{\kappa_0}$ depends \textbf{only} on the value of $Q$ (see formulas \eqref{e-53-kappa-prime} or \eqref{e-85-kappa-prime}) for a fixed value of $\kappa_0$. The same is true about the parameter $U$ in \eqref{e-75-U} or \eqref{e-87-U}. Consequently, \textbf{all} the functions $V(z)\in\sM^Q_{\kappa_0}$ or $V(z)\in\sM^{-1,Q}_{\kappa_0}$  with  fixed values of constant $Q$ and $\kappa_0$  can be realized with model L-systems $\Theta$ of the form \eqref{e-59} described in the  proofs of Theorems \ref{t-18} and \ref{t-20}, respectively. These realizing L-systems, however, differ from each other by state-spaces and symmetric operators that depend on different measures in integral representation \eqref{hernev-0} of $V(z)$. An alternative approach when realization is performed with L-systems that have the same $Q$-independent state space and  symmetric operator  will be presented in the next section. Also, the realizing L-systems corresponding to different functions $V(z)\in\sM^Q_{\kappa_0}$ (or $V(z)\in\sM^{-1,Q}_{\kappa_0}$) are not bi-unitarily equivalent (otherwise their impedance functions would match).

We also mention that Theorems \ref{t-18} and \ref{t-20} imply that  $\kappa(-Q)=\kappa(Q)$ and $U(-Q)=\bar{U}(Q)$.

\section{A universal model and resolvent formula for    L-systems}\label{s10}

In this section we discuss a universal model of an L-system realizing a function from  generalized Donoghue classes perturbed by a parameter $Q$ and  derive a formula for the resolvent of its main operator.
  First we are going to show that a perturbed  function $Q+V(z)\in\sN^Q$ can be realized  by an L-system that has the same state space and symmetric operator that do not depend on the value of $Q$.
  In order to do that we develop an alternative to Theorems \ref{t-18} and \ref{t-20} realization method.

 For any  function $Q+V(z)\in\sN^Q$ we have that
 $$
Q+V(z)=Q+aV_0(z)=Q+a\int_\bbR \left(\frac{1}{\lambda-z}-\frac{\lambda}{1+\lambda^2}\right)d\mu(\lambda)\;\textrm{ with }\;  \int_\bbR\frac{d\mu(\lambda)}{1+\lambda^2}=1,
 $$
where $V(z)=a V_0(z)$ (see \eqref{e-imp-m}) and $V_0(z)\in\sM$.
Regardless of whether the value of normalizing  parameter $a$  is greater, smaller, or equal to 1 we proceed as follows. 
$$
\begin{aligned}
Q+V(z)&=Q+aV_0(z)=a\left(\frac{Q}{a}+V_0(z)\right)\\
&=a\left(\frac{Q}{a}+\int_\bbR \left(\frac{1}{\lambda-z}-\frac{\lambda}{1+\lambda^2}\right)d\mu(\lambda)\right),
\end{aligned}
 $$
where
$$
V_0(z)=\int_\bbR \left(\frac{1}{\lambda-z}-\frac{\lambda}{1+\lambda^2}\right)d\mu(\lambda),\quad \int_\bbR\frac{d\mu(\lambda)}{1+\lambda^2}=1.
$$
As we have shown in Theorem \ref{t-18-M-q} the perturbed function $\frac{Q}{a}+V_0(z)\in\sM^Q$ above can be  realized  by an L-system $\ti\Theta$ that has the same state space and symmetric operator as an L-system $\Theta_0$ realizing $V_0(z)\in\sM$. Moreover, the corresponding von Neumann parameters of L-system $\ti\Theta$ are found via \eqref{e-53-kappa'} and \eqref{e-54-U-M-q}, that is
\begin{equation*}
    \ti\kappa=\frac{|Q/a|}{\sqrt{(Q/a)^2+4}},\quad  \ti U=\frac{Q/a}{|Q/a|}\cdot\frac{-Q/a+2i}{\sqrt{(Q/a)^2+4}},
\end{equation*}
or
\begin{equation}\label{e-112-k-u}
    \ti\kappa=\frac{|Q|}{\sqrt{Q^2+4a^2}},\quad  \ti U=\frac{Q}{|Q|}\cdot\frac{-Q+2ai}{\sqrt{Q^2+4a^2}}.
\end{equation}
Both von Neumann's parameters $\ti\kappa$ and $\ti U$ in \eqref{e-112-k-u} correspond to the deficiency basis
\begin{equation}\label{e-100-def}
g_+^\alpha(t)=\frac{1}{t-i},\quad g_-^\alpha(t)=-\frac{e^{2i\alpha}}{t +i},
\end{equation}
normalized in the space $L^2(\bbR;d\mu)$, where
\begin{equation}\label{e-101-phase}
-e^{2i{\alpha}}=\frac{Q/a-2i}{\sqrt{(Q/a)^2+4}}=\frac{Q-2ai}{\sqrt{Q^2+4a^2}},
\end{equation}
(see \eqref{e-67-def} and the proof of Theorem \ref{t-18-M-q}).

We are going to use Theorem \ref{t-18-M-q}, Appendix \ref{A2}, and the model triple described there to assemble $\ti\Theta$ based on these values of $\ti\kappa$ and $\ti U$.  In order to do that we take the Hilbert space $L^2(\bbR;d\mu)$ and introduce a symmetric  multiplication  operator $\dot\cB$ of the form \eqref{nacha2-ap} (see Appendix \ref{A1}) by
\begin{equation}\label{e-136-nacha2-ap}
\dom(\dot \cB)=\left \{f\in \dom(\cB_{\ti U})\, \bigg | \, \int_\bbR f(\lambda)d \mu(\lambda) =0\right \},
\end{equation}
whose normalized deficiency elements $g_+^\alpha$ and $g_-^\alpha$ of the form \eqref{e-100-def} in the space $L^2(\bbR;d\mu)$. We let $\cB_{\ti U}$ be a self-adjoint extension of $\dot\cB$ of the form \eqref{e-136-nacha2-ap} with
$$
\dom(\cB_{\ti U})=\dom (\dot \cB)\dot +\linspan\left\{\,\frac{1}{\cdot -i}+\ti U\frac{(-e^{2i{\alpha}})}{\cdot +i}\right \},
$$
where $\ti U$ is given by \eqref{e-112-k-u} and let  $T_{\calB}^{\ti\kappa} $ be   the dissipative restriction of the operator  $\dot \cB^*$  on
\begin{equation*}\label{e-137-nacha3-ap}
\dom(T_{\calB}^{\ti\kappa})=\dom (\dot \cB)\dot +\linspan\left\{\,\frac{1}{\cdot -i}- \ti\kappa\frac{(-e^{2i{\alpha}})}{\cdot +i}\right \},
\end{equation*}
where $\ti\kappa$ is given by \eqref{e-112-k-u}. Using the model triple $(\dot \cB,   T_{\calB}^{\ti\kappa} ,\cB_{\ti U})$ above in the Hilbert space $L^2(\bbR;d\mu)$ we built the $(*)$-extension $\ti\dB$ of $T_{\calB}^{\ti\kappa}$ and construct an L-system
\begin{equation}\label{e-114-un}
\ti\Theta= \begin{pmatrix} \ti\dB&\ti K&\ 1\cr \calH_+ \subset L^2(\bbR;d\mu) \subset
\calH_-& &\dC\cr \end{pmatrix}
\end{equation}
of the form \eqref{e-59} as explained in the proof of  Theorem \ref{t-18-M-q}. Here  $\ti K c=c\cdot \ti\chi$, $\ti K^*f=(f,\ti\chi)$, $(f\in\calH_+)$ where $\ti\chi$ is given by \eqref{e-212} or
$$
    \ti\chi=\frac{\ti\kappa^2+1+2\ti\kappa \ti U}{\sqrt2|1+\ti\kappa \ti U|\sqrt{1-\ti\kappa^2}}\varphi+ \frac{\ti\kappa^2\ti U+2\ti\kappa+\ti U}{\sqrt2|1+\ti\kappa\ti U|\sqrt{1-\ti\kappa^2}}\psi.
$$
Taking into account that in the above formula $\|\varphi\|_-=\|\psi\|_-=1$ we use
\begin{equation}\label{e-106-phi-psi}
\varphi=\calR^{-1}\left(\frac{1}{\sqrt2}g^\alpha_+\right)\quad\textrm{ and }\quad\psi=\calR^{-1}\left(\frac{1}{\sqrt2}g_-^\alpha\right),
\end{equation}
where $\calR$ is the  Riesz-Berezansky   operator (see Appendix \ref{A2}), we get
$$
    \ti\chi=\frac{\ti\kappa^2+1+2\ti\kappa \ti U}{2|1+\ti\kappa \ti U|\sqrt{1-\ti\kappa^2}}(\calR^{-1}g^\alpha_+)+ \frac{\ti\kappa^2\ti U+2\ti\kappa+\ti U}{2|1+\ti\kappa\ti U|\sqrt{1-\ti\kappa^2}}(\calR^{-1}g^\alpha_-).
$$
 As one can see  $\dot\cB$ is chosen to be the same as a symmetric operator in L-system $\Theta_0$ realizing $V_0(z)\in\sM$  and  hence does not depend on the value of $Q$.
 Moreover, since $\ti\Theta$ realizes $\frac{Q}{a}+V_0(z)$, we have
$$
V_{\ti\Theta}(z)=\frac{Q}{a}+V_0(z)=\left((\RE \ti\dB-z I)^{-1}\ti\chi,\ti\chi) \right).
$$
Furthermore, (see \eqref{e-214})
$$
    \begin{aligned}
\RE\ti\dB&=\dot\cB^*-\frac{i}{2|1+\ti\kappa\ti U|^2}(\cdot,\varphi-\ti U\psi)\left((\ti\kappa^2+1+2\ti\kappa\ti U)\varphi+ (\ti\kappa^2\ti U+2\ti\kappa+\ti U)\psi\right)\\
&=\dot\cB^*-\left(\cdot,\frac{-i\sqrt{1-\ti\kappa^2}}{\sqrt2|1+\ti\kappa\ti U|}\varphi+ \frac{i\ti U\sqrt{1-\ti\kappa^2}}{\sqrt2|1+\ti\kappa\ti U|}\psi\right)\ti\chi\\
&=\dot\cB^*-\frac{i\sqrt{1-\ti\kappa^2}}{\sqrt2|1+\ti\kappa\ti U|}\left(\cdot,\varphi-\ti U\psi\right)\ti\chi\\
&=\dot\cB^*-\frac{i\sqrt{1-\ti\kappa^2}}{2|1+\ti\kappa\ti U|}\left(\cdot,(\calR^{-1}g^\alpha_+)-\ti U(\calR^{-1}g^\alpha_-)\right)\ti\chi.
    \end{aligned}
$$
Let
\begin{equation}\label{e-113-chi}
        \begin{aligned}
\chi&=\sqrt{a}\,\ti\chi=\sqrt{\frac{a}{2}}\left(\frac{\ti\kappa^2+1+2\ti\kappa \ti U}{|1+\ti\kappa \ti U|\sqrt{1-\ti\kappa^2}}\varphi+ \frac{\ti\kappa^2\ti U+2\ti\kappa+\ti U}{|1+\ti\kappa\ti U|\sqrt{1-\ti\kappa^2}}\psi\right)\\
&={\frac{\sqrt a}{2}}\left(\frac{\ti\kappa^2+1+2\ti\kappa \ti U}{|1+\ti\kappa \ti U|\sqrt{1-\ti\kappa^2}}(\calR^{-1}g^\alpha_+)+ \frac{\ti\kappa^2\ti U+2\ti\kappa+\ti U}{|1+\ti\kappa\ti U|\sqrt{1-\ti\kappa^2}}(\calR^{-1}g^\alpha_-)\right).
    \end{aligned}
\end{equation}
Then applying \eqref{e-106-phi-psi} and \eqref{e-113-chi}
$$
        \begin{aligned}
\RE\ti\dB&=\dot\cB^*-\frac{i\sqrt{1-\ti\kappa^2}}{\sqrt{2a}|1+\ti\kappa\ti U|}\left(\cdot,\varphi-\ti U\psi\right)\chi\\
&=\dot\cB^*-\frac{i\sqrt{1-\ti\kappa^2}}{2\sqrt{a}|1+\ti\kappa\ti U|}\left(\cdot,\calR^{-1}g^\alpha_+-\ti U\calR^{-1}g^\alpha_-\right)\chi.
    \end{aligned}
$$
Observe,
$$
Q+V(z)=a\left(\frac{Q}{a}+V_0(z)\right)=a\left((\RE \ti\dB-z I)^{-1}\ti\chi,\ti\chi \right)=\left((\RE \ti\dB-z I)^{-1}\chi,\chi\right).
$$
We introduce a new operator
\begin{equation}\label{e-114-bA}
    \dB=\RE\ti\dB+i(\cdot,\chi)\chi.
\end{equation}
According to \cite[Theorem 6.4.6]{ABT} this operator $\dB$ is a $(*)$-extension of an operator $T$ ($T\supset\dot\cB$, $T^*\supset\dot\cB$) with some von Neumann's parameter $\kappa$. We are going to find $\kappa$. We have
$$
 \begin{aligned}
\dB&=\dot\cB^*-\left(\cdot,\frac{-i\sqrt{1-\ti\kappa^2}}{\sqrt{2a}|1+\ti\kappa\ti U|}\varphi+ \frac{i\ti U\sqrt{1-\ti\kappa^2}}{\sqrt{2a}|1+\ti\kappa\ti U|}\psi\right)\chi+i(\cdot,\chi)\chi\\
&=\dot\cB^*-\left(\cdot,\frac{-i\sqrt{1-\ti\kappa^2}}{\sqrt{2a}|1+\ti\kappa\ti U|}\varphi+ \frac{i\ti U\sqrt{1-\ti\kappa^2}}{\sqrt{2a}|1+\ti\kappa\ti U|}\psi\right)\chi\\
&\quad+\sqrt{\frac{a}{2}}\left(\cdot,-i\frac{\ti\kappa^2+1+2\ti\kappa \ti U}{|1+\ti\kappa \ti U|\sqrt{1-\ti\kappa^2}}\varphi-i\frac{\ti\kappa^2\ti U+2\ti\kappa+\ti U}{|1+\ti\kappa\ti U|\sqrt{1-\ti\kappa^2}}\psi\right)\chi\\
&=\dot\cB^*+\left(\cdot,\left[\frac{i\sqrt{1-\ti\kappa^2}}{\sqrt{2a}|1+\ti\kappa\ti U|}-\frac{i\sqrt{a}(\ti\kappa^2+1+2\ti\kappa \ti U)}{\sqrt2|1+\ti\kappa \ti U|\sqrt{1-\ti\kappa^2}}\right]\varphi\right.\\
&\quad+\left.\left[\frac{-i\ti U\sqrt{1-\ti\kappa^2}}{\sqrt{2a}|1+\ti\kappa\ti U|}-\frac{i\sqrt{a}(\ti\kappa^2\ti U+2\ti\kappa +\ti U)}{\sqrt2|1+\ti\kappa \ti U|\sqrt{1-\ti\kappa^2}}\right]\psi\right)\chi.\\
    \end{aligned}
$$
Applying \eqref{e-106-phi-psi} yields
$$
 \begin{aligned}
\dB&=\dot\cB^*+\left(\cdot,\left[\frac{i\sqrt{1-\ti\kappa^2}}{2\sqrt{a}|1+\ti\kappa\ti U|}-\frac{i\sqrt{a}(\ti\kappa^2+1+2\ti\kappa \ti U)}{2|1+\ti\kappa \ti U|\sqrt{1-\ti\kappa^2}}\right](\calR^{-1}g^\alpha_+)\right.\\
&\quad+\left.\left[\frac{-i\ti U\sqrt{1-\ti\kappa^2}}{2\sqrt{a}|1+\ti\kappa\ti U|}-\frac{i\sqrt{a}(\ti\kappa^2\ti U+2\ti\kappa +\ti U)}{2|1+\ti\kappa \ti U|\sqrt{1-\ti\kappa^2}}\right](\calR^{-1}g^\alpha_-)\right)\chi\\
&=\dot\cB^*+\frac{i}{2\sqrt{a}|1+\ti\kappa\ti U|\sqrt{1-\ti\kappa^2}}\left(\cdot,\left[a(\ti\kappa^2+2\ti\kappa\ti U+1)-1+\ti\kappa^2\right](\calR^{-1}g^\alpha_+)\right.\\
&\quad+\left.\left[a(\ti\kappa^2\ti U+2\ti\kappa+\ti U)+\ti U-\ti U\ti\kappa^2\right](\calR^{-1}g^\alpha_-)\right)\chi.
   \end{aligned}
$$
In order to find $\kappa$ we rely on the fact that if $f=g_+^\alpha-\kappa g_-^\alpha\in\dom(T)$, then $\dB f=T f$. This yields
$$
\left[a(\ti\kappa^2+2\ti\kappa\ti U+1)-1+\ti\kappa^2\right]-\bar\kappa\left[a(\ti\kappa^2\ti U+2\ti\kappa+\ti U)+\ti U-\ti U\ti\kappa^2\right]=0.
$$
Solving for $\kappa$ results in
\begin{equation*}
    \begin{aligned}
    \kappa&=\overline{\left(\frac{a(\ti \kappa^2+2\ti \kappa \ti U+1)-(1-\ti\kappa^2)}{a(\ti\kappa^2\ti U+2\ti\kappa +\ti U)+\ti U(1-\ti\kappa^2)}\right)}=\frac{a(\ti \kappa^2+2\ti \kappa \overline{\ti U}+1)-(1-\ti\kappa^2)}{a(\ti\kappa^2\overline{\ti U}+2\ti\kappa +\overline{\ti U})+\overline{\ti U}(1-\ti\kappa^2)}\cdot\frac{\ti U}{\ti U}\\
    &=\frac{a(\ti\kappa^2{\ti U}+2\ti\kappa +{\ti U})-{\ti U}(1-\ti\kappa^2)}{a(\ti \kappa^2+2\ti \kappa {\ti U}+1)+(1-\ti\kappa^2)},
    \end{aligned}
\end{equation*}
where $\ti\kappa$ and $\ti U$ are defined by \eqref{e-112-k-u}. Substituting the values of $\ti\kappa$ and $\ti U$ from \eqref{e-112-k-u} into the above formula yields the following
\begin{equation}\label{e-115-kappa}
    \kappa=\frac{(a-1-Qi)\sqrt{Q^2+4a^2}}{(a-1)Q-(Q^2+2a^2+2a)i}.
\end{equation}
Therefore, $\dB$ is a $(*)$-extension of an operator $T=T(Q)$ defined by the von Neumann parameter $\kappa$ in \eqref{e-115-kappa} and $\dot\cB\subset T\subset\dB$, $\dot\cB\subset T^*\subset\dB^*$. Furthermore, we can include $\dB$ into a new L-system
\begin{equation}\label{e-136-un-model}
\Theta^Q= \begin{pmatrix} \dB&K&\ 1\cr \calH_+ \subset L^2(\bbR;d\mu) \subset
\calH_-& &\dC\cr \end{pmatrix},
\end{equation}
that will share the symmetric operator $\dot\cB$ and state-space with $\ti\Theta$ (and ultimately with $\Theta_0$). Here  $K c=c\cdot \chi$, $K^*f=(f,\chi)$, $(f\in\calH_+)$.
According to our explicit construction the impedance function $V_\Theta(z)=Q+V(z)$. Consequently, we have presented another realization of $Q+V(z)$ by a minimal L-system $\Theta^Q$ whose symmetric operator and state-space are independent of the perturbing parameter $Q$. We will refer to $\Theta^Q$ in \eqref{e-136-un-model} as a \textit{universal model} of  L-system of the form \eqref{e-62}.

As we mentioned in Remark \ref{r-16}, the modified deficiency basis $g_\pm^\alpha$ of the form \eqref{e-100-def} is used in the construction and description of the universal model  $\Theta^Q$. Both von Neumann's parameters $\kappa$ and $\ti U$ in $\Theta^Q$ are given by \eqref{e-115-kappa} and \eqref{e-112-k-u} with respect to these deficiency vectors. Alternatively, we could use the original normalized deficiency pair $g_\pm$ of the form \eqref{e-52-def} in the space $L^2(\bbR;d\mu)$. This approach produces different values of parameters $\kappa$ and  $\ti U$. The corresponding formulas are
\begin{equation}\label{e-111-k}
    \kappa=\frac{(a-1-Qi)\sqrt{Q^2+4a^2}}{(a-1)Q-(Q^2+2a^2+2a)i}\cdot\frac{Q-2ai}{\sqrt{Q^2+4a^2}}=\frac{(a-1-Qi)(Q-2ai)}{(a-1)Q-(Q^2+2a^2+2a)i},
\end{equation}
and
\begin{equation}\label{e-112-U}
\ti U=\frac{Q}{|Q|}\cdot\frac{-Q+2ai}{\sqrt{Q^2+4a^2}}\cdot\frac{Q-2ai}{\sqrt{Q^2+4a^2}}=-\frac{Q}{|Q|}\cdot\frac{4a^2-Q^2+8aQi}{Q^2+4a^2}.
\end{equation}
Here we used the value of $-e^{2i\alpha}$ found via \eqref{e-101-phase}.

\begin{theorem}\label{t-34-model}
Let $\Theta$ be a minimal L-system of the form \eqref{e-62} whose impedance function is
\begin{equation}\label{e-138-V}
V_\Theta(z)=Q+a\int_\bbR \left(\frac{1}{\lambda-z}-\frac{\lambda}{1+\lambda^2}\right)d\mu(\lambda)\quad\textrm{ with }\quad  \int_\bbR\frac{d\mu(\lambda)}{1+\lambda^2}=1.
\end{equation}
Then $\Theta$ is bi-unitarily equivalent to a universal model L-system $\Theta^Q$ of the form \eqref{e-136-un-model} constructed upon the values of $Q$ and $a$.

Moreover, the resolvent\footnote{A functional model in resolvent form of a dissipative quasi-selfadjoint extension of a prime symmetric operator with deficiency indices $(1,1)$ was established first by K.A.~Makarov and one of the authors (E.T.) in \cite{MT-S}.} of the main operator $T=T(Q)$ of $\Theta^Q$ acting on $L^2(\bbR;d\mu)$ is given by
\begin{equation}\label{e-98-res-form-1}
(T - zI )^{-1}=(\cB- zI )^{-1}-p(z)(\cdot\, ,g_{\overline{z}})g_z, 
\end{equation}
with
\begin{equation}\label{e-99-res-form-2}
p(z)=\left (M(\dot\cB, \cB)(z)+i\frac{\mathds{k}+1}{\mathds{k}-1}\right )^{-1},\quad z\in\rho(T)\cap \rho(\cB).
\end{equation}
Here $\dot\cB$ is symmetric operator of the form \eqref{e-136-nacha2-ap}, $g_z$, $g_\pm=g_{\pm i}$ are deficiency vectors of $\dot\cB$ of the form \eqref{e-52-def}, $\cB$ is a self-adjoint extension of $\dot\cB$ such that $g_+-g_-\in\dom(\cB)$,  $M(\dot\cB, \cB)$ is the Weyl-Titchmarsh function associated with the pair  $(\dot\cB, \cB)$, and 
$$
\mathds{k}=\frac{(a-1-Qi)(Q-2ai)}{(a-1)Q-(Q^2+2a^2+2a)i}.
$$
\end{theorem}
\begin{proof}
As we have shown above, the model L-system $\Theta^Q$ of the form \eqref{e-136-un-model} constructed upon the values of $Q$ and $a$ realizes the function
 $$
Q+a\int_\bbR \left(\frac{1}{\lambda-z}-\frac{\lambda}{1+\lambda^2}\right)d\mu(\lambda)\quad\textrm{ with }\quad  \int_\bbR\frac{d\mu(\lambda)}{1+\lambda^2}=1.
 $$
 Hence, $V_\Theta(z)=V_{\Theta^Q}(z)$ for all $z\in\dC_+$. Also, both L-systems $\Theta$ and $\Theta^Q$ are minimal. Then, according to Theorem on Bi-unitary Equivalence (see \cite[Theorem 6.6.10]{ABT}),  L-system $\Theta^Q$ is bi-unitarily equivalent to  our L-system $\Theta$.

Resolvent formula \eqref{e-98-res-form-1}-\eqref{e-99-res-form-2} was proven in \cite[Theorem 5.2]{MT-S} for the model triple 
with an appropriate (possibly complex) value of $\kappa$ mentioned in Hypothesis \ref{setup}. In our case we simply apply this result to the model triple $(\dot \cB,   T ,\cB)$ in the Hilbert space $L^2(\bbR;d\mu)$ that is present in our universal model L-system $\Theta$ and clearly satisfies Hypothesis \ref{setup} if one chooses the deficiency basis $g_\pm$  of the form \eqref{e-52-def} and von Neumann's parameter $\mathds{k}$ for the main operator $T$ given by \eqref{e-111-k}.
\end{proof}
We should mention that the Weyl-Titchmarsh function $M(\dot\cB, \cB)$ that appears in formula \eqref{e-99-res-form-2} can be replaced with a linear-fractional transformation of the Weyl-Titchmarsh function $M(\dot\cB, \cB_U)$ associated with the quasi-kernel of the $\RE\dB$ of the L-system $\Theta^Q$. In order to do that we utilize the corresponding transformation law for the Weyl-Titchmarsh functions  (see   \cite{GT}, \cite{MT-S}) for some $\alpha \in [0,\pi)$ as follows
\begin{equation*}\label{transm}
M(\dot\cB, \cB)=\frac{(\cos\alpha) \, M(\dot\cB, \cB_U)-\sin \alpha}{
\cos\alpha +(\sin\alpha) \,M(\dot\cB, \cB_U)}.
\end{equation*}

We note that any L-system $\Theta$ of the form \eqref{e-62} has its impedance function in the form \eqref{e-138-V} for some values of $Q$ and $a$. Thus, $\Theta^Q$ justifies its name of the universal model. We should  mention that $\Theta^Q$ is bi-unitarily equivalent to the realization L-systems  constructed in Theorems \ref{t-18} and \ref{t-20}. We also know (see \cite{ABT}) that if two minimal L-systems are bi-unitarily equivalent, then their main operators are unitarily equivalent and hence the absolute value of the von Neumann parameters  is preserved.   In particular, it follows directly from Theorems \ref{t-18}, \ref{t-20}, and \ref{t-34-model} that if $\kappa$ is defined by \eqref{e-115-kappa} then
\begin{equation}\label{e-120-kappas}
   |\kappa|=\left\{
               \begin{array}{ll}
                 \frac{\left(b-2Q^2-\sqrt{b^2+4Q^2}\right)^2-a\left(b-\sqrt{b^2+4Q^2}\right)^2+4Q^2a(a-1)}{\left(b-2Q^2-\sqrt{b^2+4Q^2}\right)^2+a\left(b-\sqrt{b^2+4Q^2}\right)^2+4Q^2a(a+1)}, & \hbox{if $a<1$;} \\
                 \frac{a\left(b+\sqrt{b^2+4Q^2}\right)^2-\left(b-2Q^2+\sqrt{b^2+4Q^2}\right)^2-4Q^2a(a-1)}{\left(b-2Q^2+\sqrt{b^2+4Q^2}\right)^2+a\left(b+\sqrt{b^2+4Q^2}\right)^2+4Q^2a(a+1)}, & \hbox{if $a>1$;} \\
                 \frac{|Q|}{\sqrt{Q^2+4}}, & \hbox{if $a=1$.}
              \end{array}
             \right.
\end{equation}
Here as before $b$ is defined by \eqref{e-78-b}.

We also observe that the value of $\kappa$  defined by \eqref{e-115-kappa} becomes real and non-negative if $a=1$ and matches the third case in \eqref{e-120-kappas}. If $a\ne1$, then
$\kappa$ is, generally speaking, complex and hence $\kappa=|\kappa|e^{i\phi},$ where $\phi=\phi(Q,a)$ depends on parameters $Q$ and $a$. Thus, the value of the phase factor of $\kappa$ can be found for given $Q$ and $a$ as
$$
e^{i\phi}=\frac{\kappa}{|\kappa|},
$$
where $|\kappa|$ is defined by \eqref{e-120-kappas}. In Example 2 in the end of the current paper we will present two different realizations of a function $V(z)\in\sM^Q_{\kappa_0}$ such that one of them uses a real positive $\kappa$ and the other one complex.

The following theorem will rely on the derivations in the beginning of this section to enhance Theorem \ref{t-18}. It shows that an existing L-system realization of an unperturbed function of the class $\sM_{\kappa_0}$ can be used to create a realization of a perturbed function preserving the state space and symmetric operator of the unperturbed L-system for any value of $Q$.
\begin{theorem}\label{t-22-A}
Let $V(z)\in\sM_{\kappa_0}$ be realized by an L-system $\Theta_0$ containing  a symmetric operator $\dA$ and a state-space $\calH_-\subset\calH\subset\calH_-$. Then for any real $Q\ne0$ the function $Q+V(z)$ can be realized by an L-system $\Theta$ with the same symmetric operator $\dA$ and a state space $\calH_-\subset\calH\subset\calH_-$ as in $\Theta_0$.
\end{theorem}
\begin{proof}
Most of the work was done in the beginning of the section where we have shown that for any  value of $Q\ne0$  an L-system realizing $Q+V(z)$  can be chosen to keep the same symmetric operator and state-space.  All we have to do now is to show that there exists a value of $Q\ne0$ such that a realizing L-system for  $Q+V(z)$ can be picked so that its symmetric operator and state-space will match $\dA$ and $\calH_-\subset\calH\subset\calH_-$ of $\Theta_0$.

Let $Q=\sqrt{1-a^2}$, where $a$ is the parameter corresponding to $V(z)$ via \eqref{e-66-L}. In this case parameter $b$ in \eqref{e-78-b} is $b=Q^2-(1-a^2)=0$. Consequently, the value of parameter $\alpha$ in \eqref{e-74-q-alpha} is such that
$$
    \tan 2\alpha=\frac{2Q}{1-Q^2-a^2}=+\infty,
$$
and hence $\alpha=\pi/4$. Consider function $V_{\alpha}(z)$ related to our function $Q+V(z)$ by \eqref{e-54-frac}
$$
 V_{\alpha}(z)=\frac{\cos\alpha+(\sin\alpha)(Q+V(z))}{\sin\alpha-(\cos\alpha)(Q+V(z))}.
$$
Under our assumptions we have
$$
\begin{aligned}
 V_{\frac{\pi}{4}}(z)&=\frac{\frac{1}{\sqrt2}+\frac{1}{\sqrt2}(\sqrt{1-a^2}+V(z))}{\frac{1}{\sqrt2}-\frac{1}{\sqrt2}(\sqrt{1-a^2}+V(z))}=\frac{1+\sqrt{1-a^2}+V(z)}{1-\sqrt{1-a^2}-V(z)}\\
 &=\frac{1+\sqrt{1-a^2}+aV_0(z)}{1-\sqrt{1-a^2}-aV_0(z)},
 \end{aligned}
$$
where $V_0(z)\in\sM$ is such that $V(z)=aV_0(z)$ (see \eqref{e-imp-m}). On the other hand by our construction and \eqref{e-imp-m} again $V_{\frac{\pi}{4}}(z)=a_{\frac{\pi}{4}}V_1(z)$, where $V_1(z)\in\sM$. We use \eqref{e-80-a-alpha} to find
$$
a_{\frac{\pi}{4}}=\frac{4a(1-a^2)+4a(1-a^2)}{[-2(1-a^2)+2\sqrt{1-a^2}]^2+4a^2(1-a^2)}=\frac{a}{1-\sqrt{1-a^2}}.
$$
Taking into account that $0<a<1$ we set
$$
\sin\beta=a,\quad \cos\beta=\sqrt{1-a^2},\quad \beta\in(0,\pi/2).
$$
Using trigonometric identities we get
\begin{equation}\label{e-116-cot}
\cot(\beta/2)=\frac{2\sin(\beta/2)\cos(\beta/2)}{2\sin^2(\beta/2)}=\frac{\sin\beta}{1-\cos\beta}=\frac{a}{1-\sqrt{1-a^2}}=a_{\frac{\pi}{4}}.
\end{equation}
Moreover, performing transformations below and applying \eqref{e-116-cot} yields
$$
\begin{aligned}
V_{\frac{\pi}{4}}(z)&=a_{\frac{\pi}{4}}V_1(z)=\frac{1+\sqrt{1-a^2}+aV_0(z)}{1-\sqrt{1-a^2}-aV_0(z)}=\frac{1+\cos\beta+(\sin\beta) V_0(z)}{1-\cos\beta-(\sin\beta) V_0(z)}\\
&=\frac{2\cos^2(\beta/2)+2\cos(\beta/2)\sin(\beta/2)V_0(z)}{2\sin^2(\beta/2)-2\sin(\beta/2)\cos(\beta/2)V_0(z)}\\
&=\cot(\beta/2)\frac{\cos(\beta/2)+\sin(\beta/2)V_0(z)}{\sin(\beta/2)-\cos(\beta/2)V_0(z)}=a_{\frac{\pi}{4}}\frac{\cos(\beta/2)+\sin(\beta/2)V_0(z)}{\sin(\beta/2)-\cos(\beta/2)V_0(z)}.
\end{aligned}
$$
Therefore,
\begin{equation}\label{e-117-V1}
    V_1(z)=\frac{\cos(\beta/2)+\sin(\beta/2)V_0(z)}{\sin(\beta/2)-\cos(\beta/2)V_0(z)}.
\end{equation}
Consequently, $\sqrt{1-a^2}+V(z)$ can be realized by an L-system that has the same symmetric operator $\dA$ and a state-space $\calH_-\subset\calH\subset\calH_-$ as in $\Theta_0$. This completes the proof.
\end{proof}
\begin{corollary}\label{c-26}
Let $V(z)\in\sM_{\kappa_0}$ and its realizing L-system $\Theta_0$ be the same as in Theorem \ref{t-22-A}. Then the function $Q+V(z)$ can be realized by an L-system $\Theta$ with the same symmetric operator $\dA$ and a state space $\calH_-\subset\calH\subset\calH_-$ as in $\Theta_0$ and  von Neumann's parameters $\kappa$ and $U$ given by \eqref{e-53-kappa-prime}  and \eqref{e-75-U}, respectively.
\end{corollary}
\begin{proof}
All we have to do is apply Theorem \ref{t-22-A} to the proof of Theorem \ref{t-18}. Since we already know that according to Theorem \ref{t-22-A} the perturbed function $Q+V(z)$ is realized by an L-system $\Theta$ with the symmetric operator $\dA$ and a state space $\calH_-\subset\calH\subset\calH_-$, we apply the transformation \eqref{e-54-frac} to $Q+V(z)$ with $\tan\alpha$ given by \eqref{e-76-tan}. Doing this and following the proof of Theorem \ref{t-18} yields the function $V_{\Theta_\alpha}(z)\in \sM_{\kappa}$, where $\kappa$ is given by \eqref{e-53-kappa-prime}.  Applying the Theorem on constant $J$-unitary factor (see \cite[Theorem 8.2.3]{ABT}, \cite{ArTs03}) we obtain an L-system $\Theta_\alpha$ that realizes $V_{\Theta_\alpha}(z)\in \sM_{\kappa}$ but has the same symmetric operator $\dA$ and a state space $\calH_-\subset\calH\subset\calH_-$ as in $\Theta_0$. Then the transfer functions of $\Theta$ and $\Theta_\alpha$ are related by  \eqref{e-61-Junitary'}. We use this L-system $\Theta_\alpha$ to replace the model L-system with the same name in the proof of Theorem \ref{t-18}. Then we follow the remaining steps  in the proof of Theorem \ref{t-18} that lead us to the conclusion that the  von Neumann parameters $\kappa$ and $U$ of the corresponding operators in $\Theta$ are given by \eqref{e-53-kappa-prime}  and \eqref{e-75-U}, respectively.
\end{proof}

Similar to Theorem \ref{t-22-A} and Corollary \ref{c-26}  results can be obtained for a function $V(z)\in\sM_{\kappa_0}^{-1}$.


\section{Forward and inverse theorems}\label{s7}

In this section we treat forward  and inverse realization results for all the subclasses of the class $\sN^Q$ that was introduced in Section \ref{s5}.
\begin{theorem}\label{t-24}
Let $\Theta$ be a minimal L-system   of the form \eqref{e-62} with  the main operator $T$ and its von Neumann's parameter  $\kappa$, $(0\le\kappa<1)$. Then only one of the following takes place:
\begin{enumerate}
  \item $V_\Theta(z)$ belongs to class $\sM^Q$ and $\kappa$ is determined by \eqref{e-53-kappa'} for some $Q$;
  \item $V_\Theta(z)$ belongs to class $\sM^Q_{\kappa_0}$ and $\kappa$ is determined by \eqref{e-53-kappa-prime} for some $Q$ and $a=\frac{1-\kappa_0}{1+\kappa_0}$;
  \item $V_\Theta(z)$ belongs to class $\sM^{-1,Q}_{\kappa_0}$  and $\kappa$ is determined by \eqref{e-85-kappa-prime} for some $Q$ and $a=\frac{1+\kappa_0}{1-\kappa_0}$.
\end{enumerate}
The values of $Q$ and $\kappa_0$ are determined from integral representation \eqref{hernev-real} of $V_\Theta(z)$.
 \end{theorem}
\begin{proof}
It is known (see \cite{ABT}) that the impedance function $V_\Theta(z)$ has integral representation \eqref{hernev-real}, that is
$$
V_\Theta(z)= Q+\int_\bbR\left (\frac{1}{\lambda-z}-\frac{\lambda}{1+\lambda^2}\right )d\mu,\quad
\int_\bbR\frac{d\mu(\lambda)}{1+\lambda^2}=a<\infty.
$$
If $Q=0$, then the result of our theorem follows from Theorems \ref{t-10-new}--\ref{t-13-new}. Assume that $Q\ne0$.
Obviously, either $0<a<1$ or $a=1$ or $a>1$. If $a=1$, then $V_\Theta(z)$ belongs to class $\sM^Q$ by the definition. To see that the von Neumann parameter  $\kappa$ of the main operator $T$ of our L-system $\Theta$ is determined by \eqref{e-53-kappa'} for some $Q$ we simply note that
Theorem \ref{t-18-M-q} offers a model L-system realizing $V_\Theta(z)$ and having the von Neumann parameter  $\kappa$ of the main operator $T$ given by \eqref{e-53-kappa'} that uses $Q$ from integral representation \eqref{hernev-real} above. This L-system  is bi-unitarily equivalent to our L-system $\Theta$ (see \cite[Theorem 6.6.10]{ABT}) and hence preserves the absolute value of the main operators' von Neumann's parameters. Thus, $\kappa$ in this case is related to $Q$ from integral representation \eqref{hernev-real} of $V_\Theta(z)$ by \eqref{e-53-kappa'}.

Suppose $0<a<1$. Then, by the definition, $V_\Theta(z)$ belongs to class $\sM^Q_{\kappa_0}$ with (see \eqref{e-45-kappa-1})
$$
\kappa_0=\frac{1-a}{1+a}.
$$
In order to show that the von Neumann parameter  $\kappa$ of the main operator $T$ of our L-system $\Theta$ is determined by \eqref{e-53-kappa-prime} we note that Theorem \ref{t-18} offers a model L-system realizing $V_\Theta(z)$ and having the von Neumann parameter  $\kappa$ of the main operator $T$ given by \eqref{e-53-kappa-prime} based on the parameter $Q$ from integral representation \eqref{hernev-real} above. This L-system  is bi-unitarily equivalent to our L-system $\Theta$ (see \cite[Theorem 6.6.10]{ABT}) and hence preserves the absolute value of the main operators' von Neumann's parameters. Thus, $\kappa$ in this case is related to $Q$ from integral representation \eqref{hernev-real} of $V_\Theta(z)$ by \eqref{e-53-kappa-prime}.


Assume now that $a>1$. Then, by the definition, $V_\Theta(z)$ belongs to class $\sM^{-1,Q}_{\kappa_0}$ with (see \eqref{e-45-kappa-2})
$$
\kappa_0=\frac{a-1}{1+a}.
$$
Similarly to the previous case we note that Theorem \ref{t-20} offers a model L-system realizing $V_\Theta(z)$ and having the von Neumann parameter  $\kappa$ of the main operator $T$ given by \eqref{e-85-kappa-prime} based on the parameter $Q$ from integral representation \eqref{hernev-real} above. This L-system  is bi-unitarily equivalent to our L-system $\Theta$ and hence preserves the absolute value of the main operators' von Neumann's parameters. Thus, $\kappa$ in this case is related to $Q$ from integral representation \eqref{hernev-real} of $V_\Theta(z)$ by \eqref{e-85-kappa-prime}.

The proof is complete.
\end{proof}

Now we state an inverse realization theorem for the largest class under consideration $\sN^Q$.
\begin{theorem}\label{t-25}
Let  $V(z)$ belong to the  class $\sN^Q$, ($Q\ne0$) and have integral representation \eqref{e-52-M-q}.  Then $V(z)$ can be realized as the impedance function $V_{\Theta}(z)$ of a minimal L-system $\Theta$  of the form \eqref{e-62} with  the main operator $T$ and its von Neumann's parameter  $\kappa=\kappa(Q)$ and the quasi-kernel $\hat A$ of $\RE\bA$ and its von Neumann's parameter  $U=U(Q)$. Moreover, if $Q$ is a constant term from integral representation \eqref{e-52-M-q} of $V(z)$, then  exactly only one of the following takes place:
\begin{enumerate}
  \item  $\kappa$ and $U$ are determined by \eqref{e-53-kappa'} and \eqref{e-54-U-M-q};
  \item  $\kappa$ and $U$ are determined by \eqref{e-53-kappa-prime} and \eqref{e-75-U};
  \item $\kappa$ and $U$ are determined by  \eqref{e-85-kappa-prime} and \eqref{e-87-U}.
\end{enumerate}
The value of $a$ in the above formulas is determined via \eqref{e-66-L} for the measure $\mu$ from integral representation \eqref{hernev-real} of $V_\Theta(z)$.
 \end{theorem}
\begin{proof}
The class $\sN^Q$ was introduced in Section \ref{s5} as a ``perturbed" version of the class $\sN$. As we have shown in Section \ref{s5}, the class $\sN$ allows the partition
$$
\sN=\sM^-\cup\sM\cup\sM^+,
$$
where sets on the right do not intersect. Clearly then a similar non-intersecting partition
\begin{equation}\label{e-92-partition}
\sN^Q=\sM^{-Q}\cup\sM^Q\cup\sM^{+Q},
\end{equation}
takes place.
Consequently, every function $V(z)\in\sN^Q$ falls into exactly one of the classes $\sM^{-Q}$, $\sM^Q$, or $\sM^{+Q}$. Suppose that $V(z)\in\sM^{-Q}$. Then according to \eqref{e-66-sub}  we have that $V(z)\in\sM^{-1,Q}_{\kappa_0}$ for some $0<\kappa_0<1$ that is found from the integral representation \eqref{e-52-M-q} of $V(z)$ via \eqref{e-66-L} and  \eqref{e-45-kappa-2}. Therefore, we can apply Theorem \ref{t-20} that will confirm case (3) of the current theorem. Cases (1) and (2) are proved similarly. The partition formula \eqref{e-92-partition} implies that only one case is possible for a given function $V(z)\in\sN^{Q}$.
\end{proof}
Theorems \ref{t-24} and \ref{t-25} above improve and refine general realization theorems (see \cite{ABT}) for the case of one-dimensional input-output space. These results give explicit formulas for von Neumann's parameters of the main operator  and the quasi-kernel of the real part of the state-space operator of the realizing L-system.

\section{Unimodular transformations}\label{s8}

Now let us consider an L-system $\Theta$ of the form \eqref{e-62} with a main operator $T$ and the transfer function $W_\Theta(z)$. Let $B$ be a complex number such that $|B|=1$. It was shown in \cite[Theorem 8.2.3]{ABT} that there exists another L-system $\Theta_B$ of the form \eqref{e-62} with the same main operator $T$ and such that $W_{\Theta_B}(z)=W_\Theta(z)B$. The following definition was introduced in \cite{BMkT-4}.
\begin{definition}\label{d-7}
L-systems $\Theta$ and  $\Theta_\alpha$ of the form \eqref{e-62} are called  \textbf{unimodular transformations} of each other  for some $\alpha\in[0,\pi)$ if
\begin{equation}\label{e-35-uni}
    W_{\Theta_\alpha}(z)=W_\Theta(z)\cdot (-e^{2i\alpha}),
\end{equation}
where $W_\Theta(z)$ and $W_{\Theta_\alpha}(z)$ are transfer functions of the corresponding L-systems. 
\end{definition}
Note that $\Theta_{\frac{\pi}{2}}=\Theta$.
It is known (see \cite[Theorem 8.3.1]{ABT}) that if $\Theta_\alpha$ is a unimodular transformation of $\Theta$ and  $V_{\Theta_\alpha}(z)$ is its impedance function, then \eqref{e-64-alpha} takes place.

\begin{proposition}\label{p-30}
Let $Q\ne0$ and
$$
V_1(z)=Q+\int_\bbR \left(\frac{1}{\lambda-z}-\frac{\lambda}{1+\lambda^2}\right )d\mu(\lambda),\,  V_2(z)=-Q+\int_\bbR \left(\frac{1}{\lambda-z}-\frac{\lambda}{1+\lambda^2}\right )d\mu(\lambda)
$$
be two functions belonging to classes $\sN^Q$ and $\sN^{-Q}$, respectively. Then $V_1(z)$ and $V_2(z)$ can be realized by two L-systems $\Theta_1$ and $\Theta_2$ that are unimodular transformations of each other.
\end{proposition}
\begin{proof}
Theorems \ref{t-18-M-q}, \ref{t-18}, and \ref{t-20} provide us with a way to realize both $V_1(z)$ and $V_2(z)$ by model L-systems $\Theta_1$ and $\Theta_2$ sharing the same state space and symmetric operator regardless of whether
$$a=\int_\bbR\frac{d\mu(\lambda)}{1+\lambda^2}$$
is less, greater, or equal to 1. Moreover, in either case formulas \eqref{e-53-kappa'}, \eqref{e-53-kappa-prime}, or \eqref{e-85-kappa-prime} describe $\kappa$ as an even function of $Q$, i.e., $\kappa(-Q)=\kappa(Q)$. Consequently, both L-systems $\Theta_1$ and $\Theta_2$ will also share the same main operator. Applying \cite[Theorem 8.2.1]{ABT} gives us $ W_{\Theta_1}(z)=W_{\Theta_2}(z)\cdot (-e^{2i\alpha})$ for some  $\alpha\in[0,\pi)$ and hence $\Theta_1$ and $\Theta_2$  are unimodular transformations of each other.
\end{proof}

\begin{theorem}\label{t-22-Q}
Let  $V(z)$ and $V_0(z)$ belong to the  class $\sM_{\kappa_0}$, ($\kappa_0\ne0$) and $\sM$, respectively, and are related by
 \begin{equation}\label{e-90-new}
    V(z)= \frac{1-\kappa_0}{1+\kappa_0}\,V_0(z),\quad z\in\dC_+.
 \end{equation}
Then  the following statements are true:
\begin{enumerate}
  \item  $V(z)$ and $V_0(z)$ can be realized as the impedance functions  of  minimal L-systems $\Theta$ and $\Theta_0$  of the form \eqref{e-62} with   the same symmetric operator $\dA$;
  \item  $\Theta$ is not a unimodular transformation of $\Theta_0$ for any $\alpha\in[0,\pi)$;
  \item for every number $Q\ne0$ there exist a  positive number $Q_0$ such that the perturbed functions $Q+V(z)$ and $ Q_0+V_0(z)$ can be realized as the impedance functions  of  minimal L-systems $\Theta^Q$ and $\Theta_0^{Q_0}$ so that $\Theta^Q$ is a unimodular transformation of $\Theta_0^{Q_0}$.
\end{enumerate}
 \end{theorem}
 \begin{proof}
 (1) As we have shown in \cite[Section 4]{BMkT}, both functions $V(z)$ and $V_0(z)$ allow model realization by minimal L-systems $\Theta$ and $\Theta_0$ of the form \eqref{e-62-1-1} whose constructions is described in Appendix \ref{A1} and uses the same model symmetric operator $\dot B$. Note that under this construction both L-systems $\Theta$ and $\Theta_0$ obey Hypothesis \ref{setup}.

(2) Let us assume that $\Theta$ is  a unimodular transformation of $\Theta_0$ for some $\alpha\in[0,\pi)$. Then, $V(z)$ and $V_0(z)$ must be related via \eqref{e-64-alpha}. In particular,
$$
 V(i)=\frac{\cos\alpha+(\sin\alpha) V_0(i)}{\sin\alpha-(\cos\alpha) V_0(i)}.
$$
Moreover, $V_0(z)\in\sM$ and hence $V_0(i)=i$. Plugging this value in the above formula we obtain that $V(i)=i$ as well. On the other hand, \eqref{e-90-new} yields then
$$
i=V(i)= \frac{1-\kappa_0}{1+\kappa_0}\,V_0(i)=\frac{1-\kappa_0}{1+\kappa_0}\,i,
$$
that is only possible if $\kappa_0=0$. Thus we arrive at a contradiction with the fact that $\kappa_0\ne0$.

(3) 
According to Corollary \ref{c-26} we know that the functions $Q+V(z)$ and $Q_0+V_0(z)$ can be realized by minimal L-systems $\Theta^Q$ and $\Theta_0^{Q_0}$ sharing the state space and symmetric operator for any value of $Q_0$. Moreover, the von Neumann parameter $\kappa_1(Q)$ of the main operator of $\Theta^Q$ is determined by \eqref{e-53-kappa-prime} while the von Neumann parameter $\kappa_2(Q_0)$ of the main operator of $\Theta_0^{Q_0}$ is given by \eqref{e-53-kappa'}. We are going to find the value of $Q_0$ so that $\kappa_1(Q)=\kappa_2(Q_0)$.
We set the value of $\kappa_1(Q)$ in the left hand side of formula \eqref{e-53-kappa'} to obtain
\begin{equation}\label{e-91-kappa}
    \kappa_1(Q)=\frac{Q_0}{\sqrt{Q_0^2+4}},\quad Q>0,\; Q_0>0.
\end{equation}
Solving \eqref{e-91-kappa} for $Q_0>0$ we get
\begin{equation}\label{e-92-Q-0}
    Q_0=\frac{2\kappa_1(Q)}{\sqrt{1-\kappa^2_1(Q)}}.
\end{equation}
This value of $Q_0$ is used in the construction of $\Theta_0^{Q_0}$ and guarantees that  $\kappa_1(Q)=\kappa_2(Q_0)$. Consequently,  $\Theta^Q$ and  $\Theta_0^{Q_0}$ are two model L-systems with the same state space, symmetric, and main operators.  Applying \cite[Theorem 8.2.1]{ABT} yields that $\Theta^Q$ is a unimodular transformation of $\Theta_0^{Q_0}$.
\end{proof}
A similar result takes place for the class $\sM_{\kappa_0}^{-1}$.
\begin{theorem}\label{t-23-Q}
Let  $V(z)$ and $V_0(z)$ belong to the  class $\sM_{\kappa_0}^{-1}$, ($\kappa_0\ne0$) and $\sM$, respectively, and are related by
 \begin{equation}\label{e-92-new}
    V(z)= \frac{1+\kappa_0}{1-\kappa_0}\,V_0(z),\quad z\in\dC_+.
 \end{equation}
Then  the following statements are true:
\begin{enumerate}
  \item  $V(z)$ and $V_0(z)$ can be realized as the impedance functions  of  minimal L-systems $\Theta$ and $\Theta_0$  of the form \eqref{e-62} with  the  same symmetric operator $\dA$;
  \item  $\Theta_0$ is not a unimodular transformation of $\Theta$ for any $\alpha\in[0,\pi)$;
  \item for every number $Q\ne0$ there exists a  positive number $Q_0$ such that the perturbed functions $Q+V(z)$ and $Q_0+V_0(z)$ can be realized as the impedance functions  of  minimal L-systems $\Theta^Q$ and $\Theta_0^{Q_0}$ so that $\Theta^Q$ is a unimodular transformation of $\Theta_0^{Q_0}$.
\end{enumerate}
 \end{theorem}
 \begin{proof}
The proof structure here completely resembles the one of Theorem \ref{t-22-Q} except for the fact that it also relies on the realization theorem for the class $\sM_{\kappa_0}^{-1}$ found in \cite[Theorem 5.6]{BMkT-2}.
\end{proof}

\begin{remark}\label{r-25}
It can be shown that under the conditions of Theorem \ref{t-23-Q} there are  values of $a>1$ from \eqref{e-66-L} (corresponding to the integral representation \eqref{hernev-real} of $V(z)$) such that there exists  a special value of $Q>0$. This $Q$ is special in a way that both perturbed functions $Q+V(z)$ and $Q+V_0(z)$ can be realized as the impedance functions  of   L-systems $\Theta^Q$ and $\Theta_0^{Q}$ such that $\Theta^Q$ is a unimodular transformation of $\Theta_0^{Q}$. For example, let $a=4$. The existence of such a universal value $Q$ is confirmed by the Intermediate Value Theorem while solving the equation
$$
\begin{aligned}
f(Q)&=\frac{a\left(b+\sqrt{b^2+4Q^2}\right)^2-\left(b-2Q^2+\sqrt{b^2+4Q^2}\right)^2-4Q^2a(a-1)}{\left(b-2Q^2+\sqrt{b^2+4Q^2}\right)^2+a\left(b+\sqrt{b^2+4Q^2}\right)^2+4Q^2a(a+1)}\\
&\quad-\frac{Q}{\sqrt{Q^2+4}}=0,
\end{aligned}
$$
where $a=4$ is related to $\kappa_0=3/5$ via \eqref{e-45-kappa-2} and $b$ is given by \eqref{e-78-b}. In order to show that $f(Q)=0$ has solutions one confirms that $f(1)>0$ while $f(3)<0$. Thus there is a value of $Q\in(1,3)$ such that $f(Q)=0$.
As one can immediately see, the first term in the difference above gives the value of the von Neumann parameter $\kappa$ described by  formula \eqref{e-85-kappa-prime} and the second term provides the same via formula \eqref{e-53-kappa'}  for positive $Q$. The solution of the equation guarantees (see Theorems \ref{t-18-M-q} and \ref{t-18}) the existence of  L-systems $\Theta^Q$ and $\Theta_0^{Q}$ with the same main operator. Consequently, $\Theta^Q$ is a unimodular transformation of $\Theta_0^{Q}$.
The question whether such a value of $Q$ exists for all $a>1$ remains, however, open.
\end{remark}

\section{Inverse problems for L-systems}\label{s9}

In this section we are going to show how to re-construct a so-called ``perturbed" L-system based on a given one. Suppose we are given an L-system $\Theta_0$ whose impedance function $V_{\Theta_0}(z)$ belongs to one of the Donoghue classes $\sM$, $\sM_{\kappa_0}$, or $\sM_{\kappa_0}^{-1}$. Let also $Q\ne0$ be any real number. Our goal is to build another L-system $\Theta^Q$ based on the elements of the original L-system $\Theta_0$ so that its impedance function $V_{\Theta^Q}(z)$ is such that $V_{\Theta^Q}(z)=Q+V_{\Theta_0}(z)$. Our techniques and procedures will be based on the ones developed in Sections \ref{s5} and \ref{s6}. We begin with the basic Donoghue class $\sM$.
\begin{theorem}\label{t-22-M-q}
Let $\Theta_0$ be  an L-system   of the form \eqref{e-62} satisfying the conditions of Hypothesis \ref{setup} and such that its impedance function $V_{\Theta_0}(z)$  belongs to the  class $\sM$.  Then for any real number $Q\ne0$ there exists another L-system $\Theta^Q$  with the same symmetric operator $\dA$ as in $\Theta_0$ and such that
$$
V_{\Theta^Q}(z)=Q+V_{\Theta_0}(z).
$$
Moreover, the von Neumann parameter  $\kappa=\kappa(Q)$ of   its main operator $T^Q$  is determined by the formula  \eqref{e-53-kappa'} while the quasi-kernel $\hat A^Q$ of $\RE\bA^Q$ of the L-system $\Theta^Q$ is defined by \eqref{DOMHAT} with \eqref{e-54-U-M-q}.
\end{theorem}
\begin{proof}
Let $g_+$ and $g_-$ be the deficiency vectors of the symmetric operator $\dA$ of $\Theta_0$.  Then  according to conditions of Hypothesis \ref{setup} we have $g_+- g_-\in \dom (\hat A_0)$, where $\hat A_0$ is
the quasi-kernel of $\RE\bA_0$ of the L-system $\Theta_0$. We note that since $V_{\Theta_0}(z)\in\sM$, then the von Neumann parameter  $\kappa_0$ of  the main operator $T_0$ equals zero, i.e., $\kappa_0=0$ (see \cite{BMkT}).

Consider a function $V(z)=Q+V_{\Theta_0}(z)$ for some non-zero real number $Q$. If one applies transformation \eqref{e-54-frac} to $V(z)$ he would get
$$
V_{\alpha}(z)=\frac{\cos\alpha+(\sin\alpha)V(z)}{\sin\alpha-(\cos\alpha)V(z)}=\frac{\cos\alpha+(\sin\alpha)(Q+V_0(z))}{\sin\alpha-(\cos\alpha)(Q+V_0(z))}.
$$
As we did in the proof of Theorem \ref{t-18-M-q}, we show that
$
V_{\Theta_\alpha}(i)=Q_\alpha+ia_\alpha,
$
where $Q_\alpha$  and  $\mu_\alpha$  are the elements of integral representation \eqref{hernev-real} of the function $V_{\alpha}(z)$ and $a_\alpha=\int_{\dR}\frac{d\mu_\alpha(\lambda)}{1+\lambda^2}$.
We have also shown in the proof of Theorem \ref{t-18-M-q} that the function $V_{\alpha}(z)$ above would belong to the class $\sM_\kappa$, where $\kappa$ is determined from $Q$ via \eqref{e-53-kappa'}, if the value
of the angle $\alpha$ is determined by \eqref{e-61-cos-sin} for $Q<0$ and by \eqref{e-63-cos-sin} for $Q>0$. Using the procedure described in Appendix \ref{A2} we can construct an L-system $\Theta_\alpha$
obeying the conditions of Hypothesis \ref{setup} and such that $V_{\Theta_\alpha}(z)=V_{\alpha}(z)$. This construction is unique and relies on the fixed choice of  the deficiency vectors $g_+$ and $g_-$ of the
symmetric operator $\dA$ of $\Theta_0$ and the value of $\kappa$ given by \eqref{e-53-kappa'}. In order to obtain the desired L-system $\Theta^Q$ we follow the approach of the proof of Theorem \ref{t-18-M-q} and
modify the construction of $\Theta_\alpha$ so that it will include the value of $U$ given by \eqref{e-54-U-M-q} instead of $U=1$ (see Appendix \ref{A2}). Then, as it was shown in the proof of Theorem \ref{t-18-M-q},
we have an L-system $\Theta^Q$ such that $V_{\Theta^Q}(z)=Q+V_{\Theta_0}(z).$
\end{proof}

A similar to Theorem \ref{t-22-M-q} result takes place for the class $\sM_{\kappa_0}$.

\begin{theorem}\label{t-23}
Let $\Theta_{\kappa_0}$ be  an L-system   of the form \eqref{e-62} 
such that its impedance function $V_{\Theta_0}(z)$  belongs to the  class $\sM_{\kappa_0}$.  Then for any real number $Q\ne0$ there exists another L-system $\Theta^Q_\kappa$  with the same symmetric operator $\dA$
as in $\Theta_{\kappa_0}$ and such that
$$
V_{\Theta^Q_{\kappa}}(z)=Q+V_{\Theta_{\kappa_0}}(z).
$$
Moreover, the von Neumann parameter  $\kappa=\kappa(Q)$ of   its main operator $T^Q$  is determined by the formula  \eqref{e-53-kappa-prime} while the quasi-kernel $\hat A^Q$ of $\RE\bA^Q$ of the L-system $\Theta^Q_\kappa$ is defined by \eqref{DOMHAT} with \eqref{e-75-U}.
\end{theorem}
\begin{proof}
The idea of the proof resembles the one of Theorem \ref{t-22-M-q}. We start with a fixed pair of deficiency vectors $g_+$ and $g_-$ of the symmetric operator $\dA$ of $\Theta_{\kappa_0}$ and then work our way down following the steps above to construct an L-system $\Theta^Q_\kappa$. Again we use the construction procedure described in Appendix \ref{A2} with $g_+$ and $g_-$ and values of $\kappa$ and $U$ defined by \eqref{e-53-kappa-prime} and \eqref{e-75-U}, respectively. As a result we obtain an L-system  $\Theta^Q_\kappa$ such that $V_{\Theta^Q_{\kappa}}(z)=Q+V_{\Theta_{\kappa_0}}(z)$ as it was explained  in Theorem \ref{t-22-A} and Corollary \ref{c-26}.
\end{proof}
We  state an analogues result for the class $\sM_{\kappa_0}^{-1}$. Its proof is done with exactly same method as the one in Theorems \ref{t-22-M-q} and \ref{t-23}.
\begin{theorem}\label{t-24-M}
Let $\Theta_{\kappa_0}$ be  an L-system   of the form \eqref{e-62} 
such that its impedance function $V_{\Theta_0}(z)$  belongs to the  class $\sM_{\kappa_0}^{-1}$.  Then for any real number $Q\ne0$ there exists another L-system $\Theta^Q_\kappa$  with the same symmetric operator $\dA$
as in $\Theta_{\kappa_0}$ and such that
$$
V_{\Theta^Q_{\kappa}}(z)=Q+V_{\Theta_{\kappa_0}}(z).
$$
Moreover, the von Neumann parameter  $\kappa=\kappa(Q)$ of   its main operator $T^Q$  is determined by the formula  \eqref{e-85-kappa-prime} while the quasi-kernel $\hat A^Q$ of $\RE\bA^Q$ of the L-system $\Theta^Q$ is defined by \eqref{DOMHAT} with \eqref{e-87-U}.
\end{theorem}

The ``perturbed" L-system $\Theta^Q$ whose construction is based on a given L-system $\Theta$ (subject to either of Hypotheses \ref{setup} or \ref{setup-1}) and described in details in the proofs of Theorems \ref{t-22-M-q}--\ref{t-24-M} above will be called the \textbf{perturbation} of an L-system  $\Theta$. As we mentioned above the construction of the perturbation of a given L-system  relies on the fixed choice of the deficiency vectors of the symmetric operator of $\Theta$ and a $Q$-dependent pair of von Neumann's parameters $\kappa$ and $U$. Finally, the impedance functions of the perturbed and original L-systems are related by $V_{\Theta^Q}(z)=Q+V_{\Theta}(z)$.

We continue with this section studying unimodular transformations of perturbed L-systems.
\begin{theorem}\label{t-29}
Let $\Theta_{\kappa_0}$ and $\Theta$  be   L-systems   of the form \eqref{e-62} with  the same symmetric operator $\dA$, satisfying the conditions of Hypothesis \ref{setup}, and
such that their impedance function $V_{\Theta_{\kappa_0}}(z)$ and $V_{\Theta}(z)$ belong to the  class $\sM_{\kappa_0}$ and $\sM$, respectively.
Then for any perturbation $\Theta^Q_{\kappa_0}$, ($Q>0$) of $\Theta_{\kappa_0}$ there exists a  perturbation $\Theta^{Q_0}$, ($Q_0>0$) of $\Theta$ such that  $\Theta^Q_{\kappa_0}$ is a unimodular transformation of $\Theta^{Q_0}$.
\end{theorem}
\begin{proof}
First we note that under the condition of the theorem formula \eqref{e-imp-m} takes place, that is  $V_{\Theta_{\kappa_0}}(z)$ and $V_{\Theta}(z)$ satisfy
 \begin{equation}\label{e-94-new}
    V_{\Theta_{\kappa_0}}(z)= \frac{1-\kappa_0}{1+\kappa_0}\,V_{\Theta}(z),\quad z\in\dC_+.
 \end{equation}
Then (as it was explained in the proof of Theorem \ref{t-22-Q}) $\Theta_{\kappa_0}$ can not be a unimodular transformation of $\Theta$ while  \eqref{e-94-new} takes place.
Starting with L-system $\Theta_{\kappa_0}$ and an arbitrary $Q>0$, we use  Appendix \ref{A2} together with  Theorem \ref{t-18} that provide us with a construction of a perturbed  L-system $\Theta^Q_{\kappa_0}$ (with  the the same symmetric operator $\dA$) whose impedance function is $V_{\Theta_{\kappa_0}^Q}(z)=Q+V_{\Theta_{\kappa_0}}(z)$. Moreover, the von Neumann parameter $\kappa=\kappa(Q)$ of the main operator of $\Theta^Q$ is determined by \eqref{e-53-kappa-prime}. Following the logic of step (3) of the proof of Theorem \ref{t-22-Q} we calculate another number $Q_0>0$ using formula \eqref{e-92-Q-0}. We use this number $Q_0$ to perturb the L-system $\Theta$ to obtain $\Theta^{Q_0}$ so that $V_{\Theta^{Q_0}}(z)=Q_0+V_{\Theta}(z)$. By construction both L-systems $\Theta^Q_{\kappa_0}$ and $\Theta^{Q_0}$ share the same von Neumann's parameters of their main operators and hence $\Theta^{Q_0}$ is a unimodular transformation of $\Theta^Q_{\kappa_0}$.
\end{proof}
A similar result with a similar proof is stated for the class $\sM_{\kappa_0}^{-1}$.
\begin{theorem}\label{t-30}
Let $\Theta_{\kappa_0}$ and $\Theta$  be   L-systems   of the form \eqref{e-62} with  the  same symmetric operator $\dA$, satisfying the conditions of Hypothesis \ref{setup-1}, and
such that their impedance function $V_{\Theta_{\kappa_0}}(z)$ and $V_{\Theta}(z)$ belong to the  class $\sM_{\kappa_0}^{-1}$ and $\sM$, respectively.
Then for any perturbation $\Theta^Q_{\kappa_0}$, ($Q>0$) of $\Theta_{\kappa_0}$ there exists a  perturbation $\Theta^{Q_0}$, ($Q_0>0$) of $\Theta$ such that  $\Theta^Q_{\kappa_0}$ is a unimodular transformation of $\Theta^{Q_0}$.
\end{theorem}

In the next theorem we are going to look into a unimodular transformation that changes a given L-system to the one with the impedance function having the opposite constant term.
\begin{theorem}\label{t-39}
Let $\Theta$  be  an L-system   of the form \eqref{e-62} with the main operator $T$ parameterized by von Neumann's parameter $\kappa$ and the quasi-kernel $\hat A$ of the real part of the state-space operator $\bA$ parameterized by $U$. Let also the impedance function of $\Theta$ be
$$
V_\Theta(z)=Q+\int_\bbR \left(\frac{1}{\lambda-z}-\frac{\lambda}{1+\lambda^2}\right )d\mu(\lambda) \;\textrm{ with }\;  \int_\bbR\frac{d\mu(\lambda)}{1+\lambda^2}=a\;\textrm{ and }\; Q\ne0.
$$
Then there exists a unimodular transformation $\Theta_\alpha$ of $\Theta$ having the same state space, symmetric and main operators, and such that
\begin{equation}\label{e-126-V}
V_{\Theta_\alpha}(z)=-Q+\int_\bbR \left(\frac{1}{\lambda-z}-\frac{\lambda}{1+\lambda^2}\right )d\mu(\lambda).
\end{equation}
Moreover, the value of the unimodular factor is
\begin{equation}\label{e-127-uni}
    -e^{2i\alpha}=\frac{1-a^2-Q^2-2Qi}{1-a^2-Q^2+2Qi}.
\end{equation}
\end{theorem}
\begin{proof}
We rely on the proof of Theorem \ref{t-22-M-q} and ideas of Proposition \ref{p-30}. Clearly, $V_\Theta(z)\in\sN^Q$.  We start with a fixed pair of deficiency vectors $g_+$ and $g_-$ of the symmetric operator $\dA$ of $\Theta$ and then work our way down following the steps of the proof of Theorem \ref{t-22-M-q} to construct an L-system $\Theta_\alpha$. We apply the construction procedure described in Appendix \ref{A2} with $g_+$ and $g_-$. While doing so we use the values of $\kappa$ that parameterizes $T$ in $\Theta$ and $\bar U$ as it was mentioned in Theorems \ref{t-18-M-q}, \ref{t-18}, and \ref{t-20} that changing $Q$ to $-Q$ results in changing $U$ to $\bar U$ in formulas \eqref{e-54-U-M-q}, \eqref{e-75-U}, and \eqref{e-87-U}. As a result we obtain an L-system  $\Theta_\alpha$ such that $V_{\Theta_\alpha}(z)$ is given by \eqref{e-126-V}. Since both L-systems share the state space, symmetric and main operators, we can apply \cite[Theorem 8.2.1]{ABT} that gives us $ W_{\Theta}(z)=W_{\Theta_\alpha}(z)\cdot (-e^{2i\alpha})$ for some  $\alpha\in[0,\pi)$ and hence $\Theta$ is a   unimodular transformation of $\Theta_\alpha$.

In order to find the value of $(-e^{2i\alpha})$ we observe that for $V(-i)=Q-ai$ and $V_{\Theta_\alpha}(-i)=-Q-ai$. 
Then using \eqref{e6-3-6} we get
$$
W_\Theta(-i)=\frac{1-a-iQ}{1+a+iQ} \quad\textrm{ and }\quad W_{\Theta_\alpha}(-i)=\frac{1-a+iQ}{1+a-iQ}.
$$
Consequently,
$$
-e^{2i\alpha}=\frac{W_{\Theta}(-i)}{W_{\Theta_\alpha}(-i)}=\frac{1-a-iQ}{1+a+iQ}\cdot \frac{1+a-iQ}{1-a+iQ}=\frac{1-a^2-Q^2-2Qi}{1-a^2-Q^2+2Qi},
$$
that confirms \eqref{e-127-uni}.
\end{proof}

Now, suppose $\Theta$ is a fixed L-system of the form \eqref{e-62} such that $V_{\Theta}(z)$  belongs to the  class $\sM$. Assume in addition that $\Theta$ satisfies the conditions of Hypothesis \ref{setup}. Consider an L-system valued function $F(Q)$ that takes a real non-zero value of $Q$ and maps it into a perturbed L-system $\Theta^Q$ constructed as explained above. That is,
\begin{equation}\label{e-101-F}
    F:Q\mapsto \Theta^Q.
\end{equation}
Clearly, the impedance function of $F(Q)$ is $V_{F(Q)}(z)=V_{\Theta^Q}(z)=Q+V_{\Theta}(z)$ and $F(0)=\Theta$. It follows from the construction of the perturbed L-system that all $F(Q)$ share the same symmetric operator $\dA$ as the original L-system $\Theta$. The rest of the components of L-systems $F(Q)$ are explained by Theorem \ref{t-22-M-q}. Let us introduce a unimodular transformation  $F_\alpha(Q)$ of $F(Q)$, where $\alpha=\alpha(Q)$ is defined by \eqref{e-59-alpha-q}, i.e. $\tan2\alpha=-2/Q$. It follows from Definition \ref{d-7} of the unimodular transformation and \eqref{e-64-alpha} that
  \begin{equation}\label{e-102-F-alpha}
    V_{F_\alpha(Q)}(z)=\frac{\cos\alpha+(\sin\alpha) V_{F(Q)}(z)}{\sin\alpha-(\cos\alpha) V_{F(Q)}(z)}=\frac{\cos\alpha+(\sin\alpha) (Q+V_{\Theta}(z))}{\sin\alpha-(\cos\alpha) (Q+V_{\Theta}(z)}.
  \end{equation}

We will be particularly interested in L-systems that represent the one-sided limits of $F_\alpha(Q)$ as $Q\to0$. Let
\begin{equation}\label{e-102-Q+Q}
    \Theta_-=\lim_{Q\to0-}F_\alpha(Q)\quad\textrm{ and }\quad \Theta_+=\lim_{Q\to0+}F_\alpha(Q).
\end{equation}
For every $Q\ne0$ the von Neumann parameters $\kappa=\kappa(Q)$ and $U=U(Q)$ of the perturbed L-system $F(Q)=\Theta^Q$ are described by the formulas  \eqref{e-53-kappa'} and \eqref{e-54-U-M-q}, respectively. Consequently, the limit values are
\begin{equation}\label{e-103-k-U}
   \kappa_+=\kappa_-= \lim_{Q\to0}\kappa(Q)=0,\; U_-=\lim_{Q\to0-}U(Q)=-i,\; U_+=\lim_{Q\to0+}U(Q)=i.
\end{equation}
We are going to use these values to provide a constructive  description of L-systems $\Theta_-$ and $\Theta_+$ with the help of Appendix \ref{A2}. In order to describe $\Theta_-$, we use formula \eqref{e3-39-new} with values $\kappa=0$ and $U=U_-=-i$ to get $H=H_-=-1$. Then \eqref{e4-62} yields
$$
S_{\bA_-}=\left(
            \begin{array}{cc}
              0 & -1 \\
              0 & i \\
            \end{array}
          \right),\quad S_{\bA_-^*}=\left(
            \begin{array}{cc}
              -i & 0 \\
              -1 & 0 \\
            \end{array}
          \right).
$$
Furthermore, according to \eqref{e3-40} and \eqref{e-21-star} we have that
\begin{equation}\label{e-104-bA-}
    \bA_-=\dA^*-(\cdot,\psi)[\varphi-i\psi],\quad \bA_-^*=\dA^*-i(\cdot,\varphi)[\varphi-i\psi],
\end{equation}
are the state space operator of the L-system $\Theta_-$ and its adjoint. Here $\varphi=\calR^{-1}(g_+)$ and $\psi=\calR^{-1}(g_-)$ as laid out by Appendix \ref{A2}. Consequently,
\begin{equation}\label{e-105-chi-}
       \IM\bA_-=\left(\frac{1}{2}\right)(\cdot,\varphi-i\psi)(\varphi- i\psi)=(\cdot,\chi_-)\chi_-,\quad \chi_-=\frac{1}{\sqrt2}\,(\varphi- i\psi).
\end{equation}
Setting $K_-c=c\cdot\chi_-$, $(c\in\dC)$, we obtain
\begin{equation}\label{e-106-Theta-}
\Theta_-= \begin{pmatrix} \bA_-&K_-&\ 1\cr \calH_+ \subset \calH \subset
\calH_-& &\dC\cr \end{pmatrix}.
\end{equation}
Similarly one derives $\Theta_+$. Using formula \eqref{e3-39-new} with values $\kappa=0$ and $U=U_+=i$ we get $H=H_+=1$. Then \eqref{e4-62} yields
$$
S_{\bA_+}=\left(
            \begin{array}{cc}
              0 & 1 \\
              0 & i \\
            \end{array}
          \right),\quad S_{\bA_+^*}=\left(
            \begin{array}{cc}
              -i & 0 \\
              1 & 0 \\
            \end{array}
          \right).
$$
Applying \eqref{e3-40} and \eqref{e-21-star} again gives
\begin{equation}\label{e-104-bA+}
    \bA_+=\dA^*+(\cdot,\psi)[\varphi+i\psi],\quad \bA_+^*=\dA^*-i(\cdot,\varphi)[\varphi+i\psi],
\end{equation}
and
$$
\IM\bA_+=\left(\frac{1}{2}\right)(\cdot,\varphi+i\psi)(\varphi+ i\psi)=(\cdot,\chi_+)\chi_+,\quad \chi_+=\frac{1}{\sqrt2}\,(\varphi+ i\psi).
$$
Setting $K_+c=c\cdot\chi_+$, $(c\in\dC)$, we obtain
\begin{equation}\label{e-108-Theta+}
\Theta_+= \begin{pmatrix} \bA_+&K_+&\ 1\cr \calH_+ \subset \calH \subset
\calH_-& &\dC\cr \end{pmatrix}.
\end{equation}
Let us see that the L-systems $\Theta_-$ and $\Theta_+$ given by \eqref{e-106-Theta-} and \eqref{e-108-Theta+} are the ``limit" L-systems we are looking for.
In order to build the connection between $V_{\Theta_-}(z)$ and $V_{\Theta}(z)$ we let $Q\to0-$ in \eqref{e-59-alpha-q}. This yields the value
$$\alpha_-=\lim_{Q\to0-}\alpha(Q)=\pi/4$$
and hence \eqref{e-102-F-alpha} gives for $z\in\dC_+$
\begin{equation}\label{e-110-V-}
 V_{\Theta_-}(z)=\frac{\cos\alpha_-+(\sin\alpha_-)V_\Theta(z)}{\sin\alpha_--(\cos\alpha_-)V_\Theta(z)}=\frac{\frac{1}{\sqrt2}+\frac{1}{\sqrt2}V_\Theta(z)}{\frac{1}{\sqrt2}-\frac{1}{\sqrt2}V_\Theta(z)}=\frac{1+V_\Theta(z)}{1-V_\Theta(z)}.
\end{equation}
Similarly, when we let $Q\to0+$ in \eqref{e-59-alpha-q}, we have $$\alpha_+=\lim_{Q\to0+}\alpha(Q)=3\pi/4,$$ and thus for $z\in\dC_+$
\begin{equation}\label{e-110-V+}
 V_{\Theta_+}(z)=\frac{\cos\alpha_++(\sin\alpha_+)V_\Theta(z)}{\sin\alpha_+-(\cos\alpha_+)V_\Theta(z)}=\frac{-\frac{1}{\sqrt2}+\frac{1}{\sqrt2}V_\Theta(z)}{\frac{1}{\sqrt2}+\frac{1}{\sqrt2}V_\Theta(z)}=-\frac{1-V_\Theta(z)}{1+V_\Theta(z)}.
\end{equation}
Functions $V_{\Theta_-}(z)$ and $V_{\Theta_+}(z)$ in \eqref{e-110-V-} and \eqref{e-110-V+} are indeed the impedance functions of L-systems $\Theta_-$ and $\Theta_+$ given by \eqref{e-106-Theta-} and \eqref{e-108-Theta+}, respectively. This follows from the fact that $\Theta_-$ and $\Theta_+$ are unimodular transformations of the original L-system $\Theta$ corresponding to the values of $\alpha_-=\pi/4$ and $\alpha_+=3\pi/4$ since the transfer functions $W_{\Theta_-}(z)$ and $W_{\Theta_+}(z)$ are related to $W_\Theta(z)$ as
$$
\begin{aligned}
W_{\Theta_-}(z)&=W_\Theta(z)\cdot (-e^{2i\alpha_-})=W_\Theta(z)\cdot (-e^{2i\cdot\frac{\pi}{4}})=-iW_\Theta(z),\\
W_{\Theta_+}(z)&=W_\Theta(z)\cdot (-e^{2i\alpha_+})=W_\Theta(z)\cdot (-e^{2i\cdot\frac{3\pi}{4}})=iW_\Theta(z).\\
\end{aligned}
$$
The above relations can be obtained from \eqref{e-103-k-U} and \cite[Theorem 7]{BMkT}.

Considering the relations \eqref{e-110-V-} and \eqref{e-110-V+} we conclude that the only situation when $V_\Theta(z)= V_{\Theta_+}(z)= V_{\Theta_-}(z)$ in $\dC_+$ occurs when $V_\Theta(z)\equiv i$.

\section{Realization Guide and Uniqueness}\label{s11}

In this section we are going to summarize our realization results for different classes of $\sN^Q$, $Q\ne0$. Table \ref{Table-1} schematically describes realization structure for each subclass  of $\sN^Q$. We note that the construction of a realizing L-system can be made in a way that the state-space is $ \calH_+ \subset L^2(\bbR;d\mu) \subset\calH_-$, where $\mu$ is a Borel measure from the integral representation of the function $V_0(z)\in\sM$ that is
$$
V_0(z)=\int_\bbR \left(\frac{1}{\lambda-z}-\frac{\lambda}{1+\lambda^2}\right )d\mu(\lambda) \quad\textrm{ with }\quad  \int_\bbR\frac{d\mu(\lambda)}{1+\lambda^2}=1,\quad z\in \bbC_+.
$$
The realizing ($Q$-dependent) L-system takes the form
\begin{equation}\label{e-129-model}
\Theta= \begin{pmatrix} \dB&K&\ 1\cr \calH_+ \subset L^2(\bbR;d\mu) \subset
\calH_-& &\dC\cr \end{pmatrix},
\end{equation}
where all the operators are constructed according to the procedure described in Appendix \ref{A1} with moduli of
von Neumann's parameters $\kappa=\kappa(Q)$ and corresponding $U=U(Q)$ given by  Table \ref{Table-1} for each specific class under consideration. We note that the value of the von Neumann parameter $\kappa$ is  such that $0<\kappa<1$. Moreover, Theorem \ref{t-22-A} and Corollary \ref{c-26} yield that the deficiency vectors of the symmetric operator of $\Theta$ can be chosen of the form \eqref{e-52-def} in $L^2(\bbR;d\mu)$. Also, it follows from \cite[Section 6.3]{ABT}, \cite{BMkT} that the transfer function $W_\Theta(z)$ of $\Theta$ in \eqref{e-129-model} has the normalization condition
$$
|W_\Theta(i)|=\frac{1}{\kappa(Q)},\quad  |W_\Theta(-i)|={\kappa(Q)},\quad Q\ne0,
$$
where $\kappa=\kappa(Q)$ is given in  Table \ref{Table-1} for each specific class.
\begin{table}[]
\centering
\begin{tabular}{|c|c|c|}
\hline
 &  &  \\
 Class& Function & L-system   \\
  &  &  \\ \hline
 &  &  \\
&  $V(z)=Q+V_0(z)$,&  $\kappa=\frac{|Q|}{\sqrt{Q^2+4}}$\\
  $\sM^Q$ & $V_0(z)\in\sM$ &  \\
 &  &  $U=\frac{Q}{|Q|}\cdot\frac{-Q+2i}{\sqrt{Q^2+4}}$\\
  &  &  \\  \hline
  &  &  \\
 & $V(z)=Q+aV_0(z),$ & $\kappa=\frac{\left(b-2Q^2-\sqrt{b^2+4Q^2}\right)^2-a\left(b-\sqrt{b^2+4Q^2}\right)^2+4Q^2a(a-1)}{\left(b-2Q^2-\sqrt{b^2+4Q^2}\right)^2+a\left(b-\sqrt{b^2+4Q^2}\right)^2+4Q^2a(a+1)}$ \\
 &  &  \\
 $\sM^Q_{\kappa_0}$&  $V_0(z)\in\sM$& $b=Q^2+a^2-1$ \\
 &  &  \\
   & $a=\frac{1-\kappa_0}{1+\kappa_0}$ & $U=\frac{(a+Qi)(1-\kappa^2)-1-\kappa^2}{2\kappa}$ \\
  &  &  \\  \hline
 &  &  \\
&  $V(z)=Q+aV_0(z),$& $\kappa=\frac{a\left(b+\sqrt{b^2+4Q^2}\right)^2-\left(b-2Q^2+\sqrt{b^2+4Q^2}\right)^2-4Q^2a(a-1)}{\left(b-2Q^2+\sqrt{b^2+4Q^2}\right)^2+a\left(b+\sqrt{b^2+4Q^2}\right)^2+4Q^2a(a+1)}$ \\
 &  &  \\
 $\sM^{-1,Q}_{\kappa_0}$  &  $V_0(z)\in\sM$& $b=Q^2+a^2-1$ \\
  &  &  \\
   & $a=\frac{1+\kappa_0}{1-\kappa_0}$ & $U=\frac{(a+Qi)(1-\kappa^2)-1-\kappa^2}{2\kappa}$ \\
     &  &  \\  \hline
     \multicolumn{1}{l}{} & \multicolumn{1}{l}{} & \multicolumn{1}{l}{}
\end{tabular}
\caption{Realization Guide}
\label{Table-1}
\end{table}
Alternatively, a realizing L-system can be constructed with a universal model method described in Section \ref{s10}. In this case the values of  the von Neumann parameters $\kappa$ and $U$ are (generally speaking) complex and given by
$$
  \kappa=\frac{(a-1-Qi)\sqrt{Q^2+4a^2}}{(a-1)Q-(Q^2+2a^2+2a)i}\quad \textrm{and }\quad  U=\frac{Q}{|Q|}\cdot\frac{-Q+2ai}{\sqrt{Q^2+4a^2}}
$$
regardless of the value of $a$. A model realizing L-system $\Theta$ in this case still has a form \eqref{e-129-model} but its operators are constructed differently using the procedure presented in Section \ref{s10}. 

We also emphasize that both parameters $\kappa$ and $U$ (in either form in Table \ref{Table-1}) do not depend on a particular realization of an unperturbed function $aV_0(z)$  but only on perturbing parameter $Q$ and  normalization $a$ of the representing measure. Consequently, an L-system realizing the perturbed function can  be constructed based on some realization of $aV_0(z)$ with the help of  parameters $\kappa$ and $U$.

The following theorem gives certain uniqueness condition for the realization of a function of the class $\sN^Q$.
\begin{theorem}\label{t-33}
Suppose $V(z)=Q+aV_0(z)$, $V_0(z)\in\sM$ is a function of the class $\sN^Q$, ($Q\ne0$). Let also $\Theta_0$ be  a minimal L-system   of the form \eqref{e-62} that realizes $V_0(z)$ and contains a symmetric operator $\dA$ with deficiency vectors $g_z$ and  fixed normalized deficiency vectors $g_\pm=g_{\pm i}$. Then the perturbed function $V(z)$ admits a unique realization by L-system $\Theta$ with the same symmetric operator $\dA$  in the corresponding state space if the Liv\u sic function
$$
s(z)=\frac{z-i}{z+i}\cdot \frac{(g_z, g_-)}{(g_z, g_+)},
$$
of $\dA$ is not identical zero in the upper half-plane.
\end{theorem}
\begin{proof}
Assume the contrary, that is there are two different L-systems $\Theta_1$ and $\Theta_2$ of the form \eqref{e-62} that realize $V(z)$ and share the state-space and symmetric operator $\dA$ with $\Theta_0$. Then $\Theta_1$ and $\Theta_2$ contain two different main operators $T_1$ and $T_2$, ($T_1\ne T_2$) that are both dissipative quasi-selfadjoint extensions of the same symmetric operator $\dA$. At the same time since both  $\Theta_1$ and $\Theta_2$ have the same impedance function $V(z)$, they are bi-unitarily equivalent to each other (see \cite[Theorem 6.6.10]{ABT}). This implies that the operators $T_1$ and $T_2$ are unitarily equivalent. As it was shown in \cite{AtDerTsek81} in this case $s(z)\equiv0$ in the upper half-plane which contradicts the condition of the theorem.
\end{proof}
\begin{remark}\label{r-36}
Let the conditions of Theorem \ref{t-33} be satisfied. Then in the case when $s(z)\equiv0$ it is possible to construct two different realizing L-systems sharing the same state-space and symmetric operator as it will be shown in Example 2 in the next section. We note that all different realizations of $V(z)$ by minimal L-systems sharing the same state-space and symmetric operator will be bi-unitarily equivalent to each other (see \cite[Theorem 6.6.10]{ABT}).
\end{remark}

\section{Examples}

\noindent
\textbf{Example 1.}
This example is designed to illustrate the  construction of a perturbed L-system starting with an L-system whose impedance function belongs to the class $\sM$ using the method developed in the proof of Theorem  \ref{t-18-M-q}.

In the space $\calH=L^2_{\dR}=L^2_{(-\infty,0]}\oplus L^2_{[0,\infty)}$ we consider a prime symmetric operator
\begin{equation}\label{e-87-sym}
\dA x=i\frac{dx}{dt}
\end{equation}
on
$$
\begin{aligned}
\dom(\dA)&=\left\{x(t)=\left(
                       \begin{array}{c}
                         x_1(t) \\
                         x_2(t) \\
                       \end{array}
                     \right)\,\Big|\,x(t) -\text{abs. cont.},\right.\\
                     &\qquad\left. x'(t)\in L^2_{\dR},\, x_1(0-)=x_2(0+)=0\right\}.\\
\end{aligned}
$$
This operator $\dA$ is a model operator (according to the Liv\u sic Theorem  \cite{AG93}, \cite{L})
for any prime symmetric operator with deficiency indices $(1, 1)$ that admits  dissipative extension with the spectrum filling the entire open upper half-plane. At the same time this operator is a model for any prime symmetric operator that admits different dissipative unitarily equivalent extensions \cite{AtDerTsek81}.
Its  deficiency vectors  are easy to find
\begin{equation}\label{e-87-def}
g_z=\left(
    \begin{array}{c}
      e^{-izt}  \\
                0 \\
      \end{array}
     \right),\; \IM z>0,\qquad
g_z=\left(
\begin{array}{c}
               0 \\
      e^{-izt} \\
       \end{array}
     \right),\; \IM z<0.
\end{equation}
In particular, for $z=\pm i$ the   (normalized in $(+)$-norm) deficiency vectors  are
\begin{equation}\label{e-88-def}
g_+=\left(
    \begin{array}{c}
      e^t  \\
                0 \\
      \end{array}
     \right)
\in \sN_i,\, (t<0),\qquad
g_-=\left(
\begin{array}{c}
               0 \\
      e^{-t} \\
       \end{array}
     \right)\in \sN_{-i},\,(t>0).
\end{equation}
Consider also,
\begin{equation}\label{e-89-ext}
    \begin{aligned}
A x&=i\frac{dx}{dt},\\
\dom(A)&=\left\{x(t)=\left(
                       \begin{array}{c}
                         x_1(t) \\
                         x_2(t) \\
                       \end{array}
                     \right)
\,\Big|\,x_1(t),\,x_2(t) -\text{abs. cont.},\right.\\
&\left. x'_1(t)\in L^2_{(-\infty,0]},\, x'_2(t)\in L^2_{[0,\infty)},\,x_1(0-)=-x_2(0+)\right\}.\\
    \end{aligned}
\end{equation}
Clearly, $g_+-g_-\in\dom(A)$ and hence $A$ is  a self-adjoint extension of $\dA$ satisfying the conditions of Hypothesis \ref{setup}. Furthermore,
\begin{equation}\label{e-90-T}
    \begin{aligned}
T x&=i\frac{dx}{dt},\\
\dom(T)&=\left\{x(t)=\left(
                       \begin{array}{c}
                         x_1(t) \\
                         x_2(t) \\
                       \end{array}
                     \right)
\,\Big|\,x_1(t),\,x_2(t) -\text{abs. cont.},\right.\\
&\left. x'_1(t)\in L^2_{(-\infty,0]},\, x'_2(t)\in L^2_{[0,\infty)},\,x_2(0+)=0\right\}\\
    \end{aligned}
\end{equation}
is a quasi-self-adjoint extension of $\dA$ parameterized by a von Neumann parameter $\kappa=0$ that satisfies the conditions of Hypothesis \ref{setup}. Using direct check we obtain
\begin{equation}\label{e-91-T-star}
    \begin{aligned}
T^* x&=i\frac{dx}{dt},\\
\dom(T^*)&=\left\{x(t)=\left(
                       \begin{array}{c}
                         x_1(t) \\
                         x_2(t) \\
                       \end{array}
                     \right)
\,\Big|\,x_1(t),\,x_2(t) -\text{abs. cont.},\right.\\
&\left. x'_1(t)\in L^2_{(-\infty,0]},\, x'_2(t)\in L^2_{[0,\infty)},\,x_1(0-)=0\right\}.\\
    \end{aligned}
\end{equation}
Similarly one finds
\begin{equation}\label{e-92-adj}
    \begin{aligned}
\dot A^* x&=i\frac{dx}{dt},\\
\dom(\dot A^*)&=\left\{x(t)=\left(
                       \begin{array}{c}
                         x_1(t) \\
                         x_2(t) \\
                       \end{array}
                     \right)
\,\Big|\,x_1(t),\,x_2(t) -\text{abs. cont.},\right.\\
&\left. x'_1(t)\in L^2_{(-\infty,0]},\, x'_2(t)\in L^2_{[0,\infty)}\right\}.\\
    \end{aligned}
\end{equation}
Then $\calH_+=\dom(\dA^\ast)=W^1_2(-\infty,0]\oplus W^1_2[0,\infty)$, where $W^1_2$ is a Sobolev space.  Construct a rigged Hilbert space
\begin{equation}\label{e-139-triple}
\begin{aligned}
&\calH_+ \subset \calH \subset\calH_-\\
&=W^1_2(-\infty,0]\oplus W^1_2[0,\infty)\subset L^2_{(-\infty,0]}\oplus L^2_{[0,\infty)}\subset (W^1_2(-\infty,0]\oplus W^1_2[0,\infty))_-
\end{aligned}
\end{equation}
and consider operators
\begin{equation}\label{e-93-bA}
\begin{aligned}
\bA x&=i\frac{dx}{dt}+i x(0+)\left[\delta(t+)-\delta(t-)\right],\\
\bA^\ast x&=i\frac{dx}{dt}+i x(0-)\left[\delta(t+)-\delta(t-)\right],
\end{aligned}
\end{equation}
where $x(t)\in W^1_2(-\infty,0]\oplus W^1_2[0,\infty)$, $\delta(t+)$, $\delta(t-)$ are delta-functions and elements of $(W^1_2(-\infty,0]\oplus W^1_2[0,\infty))_-=(W^1_2(-\infty,0])_-\oplus (W^1_2[0,\infty))_-$ such that
$$
\delta(t+)=\left(
                       \begin{array}{c}
                         0 \\
                         \delta_2(t+) \\
                       \end{array}
                     \right), \qquad \delta(t-)=\left(
                       \begin{array}{c}
                        \delta_1(t-)  \\
                          0
                       \end{array}
                     \right),
$$
and generate functionals by the formulas
$$x(0+)=(x,\delta(t+))=(x_1,0)+(x_2,\delta_2(t+))=x_2(0+),$$
and
$$
x(0-)=(x,\delta(t-))=(x_1,\delta_1(t-))+(x_2,0)=x_1(0-).
$$
It is easy to see
that
$\bA\supset T\supset \dA$, $\bA^\ast\supset T^\ast\supset \dA,$
and
\begin{equation}\label{e-140-RbA}
\RE\bA x=i\frac{dx}{dt}+\frac{i }{2}(x(0+)+x(0-))\left[\delta(t+)-\delta(t-)\right].
\end{equation}
Clearly, $\RE\bA$ has its quasi-kernel equal to $A$ in \eqref{e-89-ext}. Moreover,
$$
\IM\bA =\left(\cdot,\frac{1}{\sqrt 2}[\delta(t+)-\delta(t-)]\right) \frac{1}{\sqrt 2}[\delta(t+)-\delta(t-)]=(\cdot,\chi)\chi,
$$
where $\chi=\frac{1}{\sqrt 2}[\delta(t+)-\delta(t-)]$.
Now we can build
\begin{equation}\label{e6-125-mom}
\Theta_0=
\begin{pmatrix}
\bA &K &1\\
&&\\
\calH_+ \subset \calH \subset\calH_- &{ } &\dC
\end{pmatrix},
\end{equation}
that is an L-system with $\calH_+ \subset \calH \subset\calH_-$ of the form \eqref{e-139-triple},
\begin{equation}\label{e7-62-new}
\begin{aligned}
Kc&=c\cdot \chi=c\cdot \frac{1}{\sqrt 2}[\delta(t+)-\delta(t-)], \quad (c\in \dC),\\
K^\ast x&=(x,\chi)=\left(x,  \frac{1}{\sqrt
2}[\delta(t+)-\delta(t-)]\right)=\frac{1}{\sqrt
2}[x(0+)-x(0-)],\\
\end{aligned}
\end{equation}
and $x(t)\in \calH_+= W^1_2(-\infty,0]\oplus W^1_2[0,\infty)$.
It was shown in \cite{BMkT-2} that  $s(\dA,A)(z)\equiv0$ for all $z$ in $\dC_+$ and $V_{\Theta_0}(z)=i$ for all $z\in\dC_+$. Thus $V_{\Theta_0}(z)$ is a constant function of the class $\sM$.

Now let us consider
\begin{equation}\label{e-001-ex}
    V(z)=1+V_{\Theta_0}(z)=1+i, \quad z\in\dC_+.
\end{equation}
Clearly, by construction $V(z)\in\sM^1$. In order to construct a perturbed L-system $\Theta$ that realizes $V(z)$ we will rely on the method described in  the proof of Theorem \ref{t-18-M-q} in order to preserve the same symmetric operator $\dA$ and state-space as in L-system $\Theta_0$. Taking $Q=1$ in \eqref{e-53-kappa'} 
 we  obtain
\begin{equation}\label{e-002-ex'}
\kappa=\frac{1}{\sqrt5}.
\end{equation}
Then applying \eqref{e-54-U-M-q} yields
\begin{equation}\label{e-158-U}
    U=\frac{-1+2i}{\sqrt5}.
\end{equation}
We are going  to  construct an L-system $\Theta$ out of the L-system $\Theta_0$ such that $V_\Theta(z)=V(z)\equiv 1+i$, $(z\in\dC_+)$.  Our construction is going to  be based on the elements of the unperturbed L-system $\Theta_0$ in \eqref{e6-125-mom} with the values of $\kappa$ and $U$ given by \eqref{e-002-ex'} and \eqref{e-158-U}. In order to utilize this construction we use formulas \eqref{e-212} and \eqref{e-214} of Appendix \ref{A2}. To apply these formulas, first we  must find  $\varphi=\calR^{-1}(g_+)$ and $\psi=\calR^{-1}(g_-)$, where $\calR$ is a  Riesz-Berezansky   operator where $g_+$ and $g_-$ are defined by \eqref{e-88-def}. Observe that by the properties of rigged Hilbert triplets we have
$$
(g_+,\varphi)=(g_+,\calR^{-1}(g_+))=(g_+,g_+)_+=1,
$$
where $(\cdot,\cdot)_+$ is the inner product in $\calH_+= W^1_2(-\infty,0]\oplus W^1_2[0,\infty)$. On the other hand, it follows from the definition of $\delta(t+)$ and $\delta(t-)$ that
\begin{equation}\label{e-154-Riesz}
(g_+,\delta(t+))=0,\quad (g_+,\delta(t-))=1,\quad (g_-,\delta(t+))=1,\quad (g_-,\delta(t-))=0.
\end{equation}
Since $\varphi$, $\psi$ and  $\delta(t-)$, $\delta(t+)$ are  elements of the same two-dimensional subspace of  $(W^1_2(-\infty,0])_-\oplus (W^1_2[0,\infty))_-$ that is $(\cdot)$-orthogonal to $\DdA$, we get
$$
\varphi=c_{11}\delta(t-)+c_{12}\delta(t+),\quad \psi=c_{21}\delta(t-)+c_{22}\delta(t+).
$$
In order to find the constants $c_{ij}$, ($i,j=1,2$) above we use the above derivations and write the linear system
$$
\left\{\begin{array}{ll}
        (g_+,\varphi)=1 & \hbox{} \\
         (g_+,\psi)=0 & \hbox{}
      \end{array} \right.
      \quad\textrm{ or }\quad \left\{\begin{array}{ll}
        (g_+,c_{11}\delta(t-)+c_{12}\delta(t+))=1 & \hbox{} \\
         (g_+,c_{21}\delta(t-)+c_{22}\delta(t+))=0. & \hbox{}
      \end{array} \right.
$$
Expanding the left sides of both equations and taking into account \eqref{e-154-Riesz} we obtain $c_{11}=1$ and $c_{21}=0$. Then
$$
\left\{\begin{array}{ll}
        (g_-,\varphi)=0 & \hbox{} \\
         (g_-,\psi)=1 & \hbox{}
      \end{array} \right.
      \quad\textrm{ or }\quad \left\{\begin{array}{ll}
        (g_+,\delta(t-)+c_{12}\delta(t+))=0 & \hbox{} \\
         (g_+,c_{22}\delta(t+))=1, & \hbox{}
      \end{array} \right.
$$
yields $c_{12}=0$ and $c_{21}=1$. Therefore, we can conclude that
\begin{equation}\label{e-155-phi-psi}
    \varphi=\calR^{-1}(g_+)=\delta(t-) \quad\textrm{ and }\quad \psi=\calR^{-1}(g_-)=\delta(t+).
\end{equation}

To continue with construction of the perturbed L-system we are moving to the modified deficiency pair $g_\pm^\alpha$ of the form  \eqref{e-76-def}.
The unimodular constant responsible for changing the deficiency vectors of the form \eqref{e-65-B} is the following
\begin{equation}\label{e-159-e-b}
    -e^{2i\alpha}=\frac{1-2i}{\sqrt{5}}.
\end{equation}
We have
\begin{equation}\label{e-160-def-alpha}
    g_+^\alpha=g_+=\left(
    \begin{array}{c}
      e^t  \\
                0 \\
      \end{array}
     \right),\, (t<0),\quad
g_-^\alpha=(-e^{2i\alpha})g_-=\frac{1-2i}{\sqrt{5}}\left(
\begin{array}{c}
               0 \\
      e^{-t} \\
       \end{array}
     \right),\,(t>0).
\end{equation}
Introduce a new operator $T_1$ as follows
\begin{equation}\label{e-90-T1}
    \begin{aligned}
T_1 x&=i\frac{dx}{dt},\\
\dom(T_1)&=\left\{x(t)=\left(
                       \begin{array}{c}
                         x_1(t) \\
                         x_2(t) \\
                       \end{array}
                     \right)
\,\Big|\,x_1(t),\,x_2(t) -\text{abs. cont.},\right.\\
&\left. x'_1(t)\in L^2_{(-\infty,0]},\, x'_2(t)\in L^2_{[0,\infty)},\,5 x_2(0+)=(-1+2i)x_1(0-)\right\}.\\
    \end{aligned}
\end{equation}
Now $\kappa=\frac{1}{\sqrt5}$ is the von Neumann parameter of $T_1$ corresponding to the deficiency vectors $g_+^\alpha$ and $g_-^\alpha$. Using the direct check we obtain
\begin{equation}\label{e-91-T1-star}
    \begin{aligned}
T^*_1 x&=i\frac{dx}{dt},\\
\dom(T^*_1)&=\left\{x(t)=\left(
                       \begin{array}{c}
                         x_1(t) \\
                         x_2(t) \\
                       \end{array}
                     \right)
\,\Big|\,x_1(t),\,x_2(t) -\text{abs. cont.},\right.\\
&\left. x'_1(t)\in L^2_{(-\infty,0]},\, x'_2(t)\in L^2_{[0,\infty)},\, (1+2i)x_2(0+)=-5 x_1(0-)\right\}.\\
    \end{aligned}
\end{equation}
Consider also,
\begin{equation}\label{e-161-ext}
    \begin{aligned}
A_1 x&=i\frac{dx}{dt},\\
\dom(A_1)&=\left\{x(t)=\left(
                       \begin{array}{c}
                         x_1(t) \\
                         x_2(t) \\
                       \end{array}
                     \right)
\,\Big|\,x_1(t),\,x_2(t) -\text{abs. cont.},\right.\\
&\left. x'_1(t)\in L^2_{(-\infty,0]},\, x'_2(t)\in L^2_{[0,\infty)},\,{5}x_2(0+)=(3+4i)x_1(0-)\right\}.\\
    \end{aligned}
\end{equation}
It is easy to check that $g_+^\alpha+U g_-^\alpha\in\dom(A_1)$, where $U$ is given by \eqref{e-158-U}. Operator $A_1$ is a self-adjoint extension of $\dA$.
Taking into account \eqref{e-155-phi-psi} and making adjustments for the new deficiency vectors \eqref{e-160-def-alpha} we have
$$
\varphi^\alpha=\varphi=\delta(t-)\quad\textrm{ and }\quad \psi^\alpha=(-e^{2i\alpha})\psi=\frac{1-2i}{\sqrt{5}}\delta(t+).
$$
Then we construct operator $\bA$  according to the formulas \eqref{e-212}--\eqref{e-205-A} as described in the Appendix \ref{A2}. We  have
\begin{equation}\label{e-160-chi1}
\chi_1=\frac{1+i}{\sqrt{2}}\delta(t-)+\frac{7+i}{5\sqrt{2}}\delta(t+),
\end{equation}
$$
\RE\bA_1 x=i\frac{dx}{dt}-\frac{i}{\sqrt{2}}\Big({5} x(0+)-(3+4i) x(0-)\Big)\chi_1,
$$
and  
$$
\begin{aligned}
\bA_1 x&=i\frac{dx}{dt}+\frac{\sqrt2 i(\kappa+\bar U)}{|1+\kappa U|\sqrt{1-\kappa^2}}\Big(x, \kappa\varphi^\alpha+\psi^\alpha\Big)\chi_1\\
&=i\frac{dx}{dt}-\frac{i}{2\sqrt5}\Big({5} x(0+)+(1-2i) x(0-)\Big)\big((5+5i)\delta(t-)+(7+i)\delta(t+)\big),
\end{aligned}
$$
where  all other components are described above. Now we can compose an L-system
\begin{equation*}
\Theta= 
\begin{pmatrix}
\bA_1&K_1 &1\\
&&\\
\calH_+ \subset \calH \subset\calH_- &{ } &\dC
\end{pmatrix},
\end{equation*}
where $\calH_+ \subset \calH \subset\calH_-$ is of the form \eqref{e-139-triple}, $K_1 c=c\cdot \chi_1$, $(c\in \dC)$, $K^*_1 x=(x,\chi_1)$, and $x(t)\in \calH_+$.
In order to confirm that the L-system $\Theta$ above realizes our function $V(z)=1+i$, ($z\in\dC_+$) we will find the impedance function $V_\Theta(z)$. To do that we compute the resolvent of $A_1$ first. Consider
$$
(A_1-zI)x(t)=g(t), \quad x(t)=\left[
                       \begin{array}{c}
                         x_1(t) \\
                         x_2(t) \\
                       \end{array}
                     \right]\in\dom(A_1),
$$
and solve linear differential equation to obtain
\begin{equation}\label{e-164-res}
    (A_1-zI)^{-1}\left[
                       \begin{array}{c}
                         g_1(t) \\
                         g_2(t) \\
                       \end{array}
                     \right]=e^{-izt}\left[
                       \begin{array}{c}
                         x_1(0-)+i\int_x^0 g_1(v)e^{izv}dv \\
                        \\
                         x_2(0+)- i\int_0^x g_2(v)e^{izv}dv \\
                       \end{array}
                     \right],
\end{equation}
where  (see \eqref{e-161-ext})  $5x_2(0+)=(3+4i)x_1(0-)$. 
Then we extend the resolvent \eqref{e-164-res} to $\calH_-$ by $(-,\cdot)$-continuity as explained in \cite[Section 4.5]{ABT} to obtain the extended resolvent $\overline{(A_1-zI)^{-1}}$. Furthermore, (see \cite{ABT})
\begin{equation}\label{e-res-152}
    \begin{aligned}
V_\Theta(z)&=\Big((\RE\bA_1-zI)^{-1}\chi_1,\chi_1 \Big)=\Big(\overline{(A_1-zI)^{-1}}\chi_1,\chi_1 \Big)\\
&=\left(e^{-izt}\left[
                       \begin{array}{c}
                         x_1(0-)+id \\
                        \\
                         x_2(0+)- ic \\
                       \end{array}
                     \right],\left[
                       \begin{array}{c}
                         d\delta(t-) \\
                        \\
                         c\delta(t+) \\
                       \end{array}
                     \right]\right),
    \end{aligned}
\end{equation}
where $c= \frac{7+i}{5\sqrt{2}}$ and $d=\frac{1+i}{\sqrt2}$ are the coefficients of the vector $\chi_1$ in \eqref{e-160-chi1}. It is also known (see \cite{ABT}) that if $\IM z>0$, then $\overline{(A_1-zI)^{-1}}\chi_1\in\sN_z$ and hence \eqref{e-res-152} implies that $x_2(0+)- ic=0$. Using this together with the fact that $5x_2(0+)=(3+4i)x_1(0-)$ (see \eqref{e-161-ext}),  we perform necessary calculations in \eqref{e-res-152} to obtain
$$
V_\Theta(z)=i\bar U c \bar d+id\bar d=1+i\equiv V(z), \quad z\in\dC_+.
$$
Thus, our L-system $\Theta$ realizes function $V(z)$.

\vskip.5cm


\noindent
\textbf{Example 2.}
In this Example  we are going to illustrate the method of construction described in Section \ref{s10} where the universal model was developed. We will also show how to construct a ``perturbed" L-system based on a given one. We will rely on some objects presented in Example 1 but with certain changes. Consider an L-system
\begin{equation}\label{e-154-mom_0}
\Xi=
\begin{pmatrix}
\bA_0 &K_0 &1\\
&&\\
\calH_+ \subset \calH \subset\calH_- &{ } &\dC
\end{pmatrix}.
\end{equation}
The state space of $\Xi$ is $\calH_+ \subset \calH \subset\calH_-$ of the form \eqref{e-139-triple}  and its symmetric operator $\dA$ is given by \eqref{e-87-sym} as in Example 1.
The main operator $T_0$ of $\Xi$ is defined as follows
\begin{equation}\label{e-149-T}
    \begin{aligned}
T_0 x&=i\frac{dx}{dt},\\
\dom(T_0)&=\left\{x(t)=\left(
                       \begin{array}{c}
                         x_1(t) \\
                         x_2(t) \\
                       \end{array}
                     \right)
\,\Big|\,x_1(t),\,x_2(t) -\text{abs. cont.},\right.\\
&\left. x'_1(t)\in L^2_{(-\infty,0]},\, x'_2(t)\in L^2_{[0,\infty)},\,3 x_2(0+)=-x_1(0-)\right\}.\\
    \end{aligned}
\end{equation}
It follows from \eqref{e-88-def} that $g_+-\frac{1}{3}g_-\in\dom(T_0)$ and hence $\kappa_0=\frac{1}{3}$ is the von Neumann parameter of $T_0$ corresponding to the deficiency vectors \eqref{e-88-def}. Its adjoint operator $T_0^*$ is given by
\begin{equation}\label{e-150-T-star}
    \begin{aligned}
T^*_0 x&=i\frac{dx}{dt},\\
\dom(T^*_0)&=\left\{x(t)=\left(
                       \begin{array}{c}
                         x_1(t) \\
                         x_2(t) \\
                       \end{array}
                     \right)
\,\Big|\,x_1(t),\,x_2(t) -\text{abs. cont.},\right.\\
&\left. x'_1(t)\in L^2_{(-\infty,0]},\, x'_2(t)\in L^2_{[0,\infty)},\, x_2(0+)=-3 x_1(0-)\right\}.\\
    \end{aligned}
\end{equation}
The state-space operator of $\Xi$ in the rigged Hilbert space \eqref{e-139-triple} is
\begin{equation}\label{e-152-bA}
\begin{aligned}
\bA_0 x&=i\frac{dx}{dt}+\frac{i}{2} (3 x(0+)+x(0-))\left[\delta(t+)-\delta(t-)\right],\\
\bA^*_0 x&=i\frac{dx}{dt}+\frac{i}{2}  ( x(0+)+3 x(0-))\left[\delta(t+)-\delta(t-)\right],
\end{aligned}
\end{equation}
where all the components are defined in Example 1. Direct check (see also \eqref{e-17-real}) reveals that 
$$
\RE\bA_0 x=i\frac{dx}{dt}+i(x(0+)+x(0-))\left[\delta(t+)-\delta(t-)\right],
$$
and
$$
\IM\bA_0 x=\frac{1}{2}\left(\cdot,[\delta(t+)-\delta(t-)]\right) [\delta(t+)-\delta(t-)]=(\cdot,\chi_0)\chi_0,
$$
where $\chi_0={\frac{1}{\sqrt2}}\big(\delta(t+)-\delta(t-)\big)$.
Finally, the channel operator of L-system $\Xi$ is
\begin{equation*}\label{e-155-new_0}
\begin{aligned}
K_0c&=c\cdot \chi_0=c\cdot {\frac{1}{\sqrt2}}[\delta(t+)-\delta(t-)], \quad (c\in \dC),\\
K^*_0 x&=(x,\chi_0)={\frac{1}{\sqrt2}}(x(0+)-x(0-)),\\
\end{aligned}
\end{equation*}
and $x(t)\in \calH_+= W^1_2(-\infty,0]\oplus W^1_2[0,\infty)$. As we have mentioned in Example 1,  $s(\dA,A)(z)\equiv0$ for all $z$ in $\dC_+$, where $A$ is the quasi-kernel of $\RE\bA_0$ defined by \eqref{e-89-ext}. Consequently, (see \cite{MT-S}, \cite{BMkT}) the characteristic function $S(\dA,T_0,A)(z)$ of the form \eqref{e-42-Liv} is
 $$S(\dA,T_0,A)(z)\equiv\kappa_0={1}/{3}$$
 for all $z\in\dC_+$. Furthermore, applying \cite[Theorem 7]{BMkT}, we have that $$W_{\Xi}(z)={1}/{S(\dA,T_0,A)(z)}\equiv 3,\quad z\in\dC_+.$$
Applying \eqref{e6-3-6} we get
$$
V_{\Xi}(z)\equiv \frac{3-1}{3+1}i=\frac{1}{2}i,\quad z\in\dC_+.
$$
Observe that $V_{\Xi}$ belongs to the class $\sM_{\kappa_0}$, where $\kappa_0=\frac{1}{3}$ and $a$ given by \eqref{e-66-L} is
$$
a=\frac{1-\kappa_0}{1+\kappa_0}=\frac{1}{2}.
$$
Now we are going 
to  construct a perturbed L-system $\Theta$ out of the elements of L-system $\Xi$ such that
$$V_\Theta(z)=V(z)\equiv 1+\frac{1}{2}i,\quad z\in\dC_+.$$
Clearly, by construction $V(z)\in\sM^1_{1/3}$. 
We  rely on the method described in Section  \ref{s10} in order to preserve the same symmetric operator $\dA$ and state-space as in L-system $\Xi$. Using the values $Q=1$ and $a=1/2$ in \eqref{e-112-k-u} we find
\begin{equation}\label{e-167-ex-k-U}
    \ti\kappa=\frac{1}{\sqrt{2}}\quad \textrm{and} \quad\ti U=\frac{-1+ i}{\sqrt{2}}.
\end{equation}
Substituting the above values in \eqref{e-115-kappa} we obtain
\begin{equation}\label{e-168-kappa}
\kappa=\frac{-2-4i}{-\sqrt2+5\sqrt2 i}=\frac{\sqrt2}{26}(11-3i).
\end{equation}
To continue with construction of the perturbed L-system we are moving to the modified deficiency pair $g_\pm^\alpha$ of the form  \eqref{e-76-def}.
The unimodular constant responsible for changing the deficiency vectors of the form \eqref{e-101-phase} is the following
\begin{equation}\label{e-169-e-b}
    -e^{2i\alpha}=\frac{1-i}{\sqrt{2}}.
\end{equation}
We have
\begin{equation}\label{e-170-def-alpha}
    g_+^\alpha=g_+=\left(
    \begin{array}{c}
      e^t  \\
                0 \\
      \end{array}
     \right),\, (t<0),\quad
g_-^\alpha=(-e^{2i\alpha})g_-=\frac{1-i}{\sqrt{2}}\left(
\begin{array}{c}
               0 \\
      e^{-t} \\
       \end{array}
     \right),\,(t>0).
\end{equation}
Also,
$
\kappa(-e^{2i\alpha})=\frac{4-7i}{13}.
$
Introduce an operator $T_1$ as follows
\begin{equation}\label{e-156-T}
    \begin{aligned}
&\quad\quad T_1 x=i\frac{dx}{dt},\\
&\dom(T_1)=\left\{x(t)=\left[
                       \begin{array}{c}
                         x_1(t) \\
                         x_2(t) \\
                       \end{array}
                     \right]
\,\Big|\,x_1(t),\,x_2(t) -\text{abs. cont.}, x'_1(t)\in L^2_{(-\infty,0]},\right.\\
&\left.  x'_2(t)\in L^2_{[0,\infty)},\,13 x_2(0+)=(-4+7i)\,x_1(0-)\right\}.\\
    \end{aligned}
\end{equation}
Also,
\begin{equation}\label{e-157-T-star}
    \begin{aligned}
&\quad\quad T^*_1 x=i\frac{dx}{dt},\\
&\dom(T^*_1)=\left\{x(t)=\left[
                       \begin{array}{c}
                         x_1(t) \\
                         x_2(t) \\
                       \end{array}
                     \right]
\,\Big|\,x_1(t),\,x_2(t) -\text{abs. cont.}, x'_1(t)\in L^2_{(-\infty,0]},\right.\\
&\left.  x'_2(t)\in L^2_{[0,\infty)}, (4+7i) \,x_2(0+)=-13 x_1(0-)\right\}.\\
    \end{aligned}
\end{equation}
Consider also,
\begin{equation}\label{e-173-ext}
    \begin{aligned}
A_1 x&=i\frac{dx}{dt},\\
\dom(A_1)&=\left\{x(t)=\left[
                       \begin{array}{c}
                         x_1(t) \\
                         x_2(t) \\
                       \end{array}
                     \right]
\,\Big|\,x_1(t),\,x_2(t) -\text{abs. cont.},\right.\\
&\left. x'_1(t)\in L^2_{(-\infty,0]},\, x'_2(t)\in L^2_{[0,\infty)},\,x_2(0+)=x_1(0-)\right\}.\\
    \end{aligned}
\end{equation}
It is easy to check that $g_+^\alpha+\ti U g_-^\alpha\in\dom(A_1)$, where $\ti U$ is given by \eqref{e-167-ex-k-U}. Operator $A_1$ is a self-adjoint extension of $\dA$.
Taking into account \eqref{e-155-phi-psi} and making adjustments for the new deficiency vectors \eqref{e-170-def-alpha} we have 
$$
\varphi^\alpha=\varphi=\delta(t-)\quad\textrm{ and }\quad \psi^\alpha=(-e^{2i\alpha})\psi=\frac{1-i}{\sqrt{2}}\delta(t+).
$$
Then we construct the operator $\bA$  according to the formulas \eqref{e-113-chi} and \eqref{e-114-bA} and the method described in the Section \ref{s10}.
We have
\begin{equation}\label{e-171-chi}
\chi_1=\frac{1}{2}\Big((2+i) \delta(t-)+ (1+i)\delta(t+)\Big),
\end{equation}
$$
\RE\bA_1 x=i\frac{dx}{dt}-\frac{i}{\sqrt{2}}\Big( x(0)-x(0+)\Big)\chi_1,
$$
and  
$$
\begin{aligned}
\bA_1 x&=\RE\bA_1x+i(x,\chi_1)\chi_1\\
&=i\frac{dx}{dt}-\frac{i}{\sqrt{2}}\Big( 13x(0+)+(4-7i)x(0-)\Big)\Big((2+i) \delta(t-)+ (1+i)\delta(t+)\Big),
\end{aligned}
$$
where all the components are described above. Now we can compose an L-system
\begin{equation*}
\Theta_1= 
\begin{pmatrix}
\bA_1&K_1 &1\\
&&\\
\calH_+ \subset \calH \subset\calH_- &{ } &\dC
\end{pmatrix},
\end{equation*}
where $\calH_+ \subset \calH \subset\calH_-$ is of the form \eqref{e-139-triple}, $K_1 c=c\cdot \chi_1$, $(c\in \dC)$, $K^\ast_1 x=(x,\chi_1)$ and $x(t)\in \calH_+$.
According to  Section \ref{s10} this L-system $\Theta_1$ realizes our function $V(z)$, that is
$$
V_\Theta(z)=V(z)\equiv 1+\frac{1}{2}i,\quad z\in\dC_+.
$$

Alternatively, we can construct a realization of $V(z)=1+\frac{1}{2}i$, ($z\in\dC_+$) based on the result of Corollary \ref{c-26}. We are going to use the techniques of Theorem \ref{t-18} and the original set of deficiency vectors \eqref{e-88-def}. This will require a positive value of $\kappa$ of the form \eqref{e-53-kappa-prime} and $U$ of the form \eqref{e-75-U} to yield
\begin{equation}\label{e-161-k-U}
    \kappa=\frac{\sqrt{65}}{13}\quad \textrm{and} \quad U=\frac{-7+4i}{\sqrt{65}}.
\end{equation}
Note that  $\kappa$ in \eqref{e-161-k-U} equals (see \eqref{e-120-kappas}) the absolute value of the complex $\kappa$ from \eqref{e-168-kappa}. Then
\begin{equation}\label{e-162-T1}
    \begin{aligned}
&\quad\quad T_{11} x=i\frac{dx}{dt},\\
&\dom(T_{11})=\left\{x(t)=\left[
                       \begin{array}{c}
                         x_1(t) \\
                         x_2(t) \\
                       \end{array}
                     \right]
\,\Big|\,x_1(t),\,x_2(t) -\text{abs. cont.}, x'_1(t)\in L^2_{(-\infty,0]},\right.\\
&\left.  x'_2(t)\in L^2_{[0,\infty)},\,\sqrt{65} x_2(0+)=-13\,x_1(0-)\right\},\\
    \end{aligned}
\end{equation}
and
\begin{equation}\label{e-163-T1-star}
    \begin{aligned}
&\quad\quad T^*_{11} x=i\frac{dx}{dt},\\
&\dom(T^*_{11})=\left\{x(t)=\left[
                       \begin{array}{c}
                         x_1(t) \\
                         x_2(t) \\
                       \end{array}
                     \right]
\,\Big|\,x_1(t),\,x_2(t) -\text{abs. cont.}, x'_1(t)\in L^2_{(-\infty,0]},\right.\\
&\left.  x'_2(t)\in L^2_{[0,\infty)},  13\,x_2(0+)=-\sqrt{65} x_1(0-)\right\}.\\
    \end{aligned}
\end{equation}
Following formulas \eqref{e-212}, \eqref{e-214}, and \eqref{e-205-A} of Appendix \ref{A2} we construct
\begin{equation}\label{e-164-chi1}
\chi_{11}=\frac{1}{2\sqrt{65}}\left(\sqrt{65}(1+2i) \delta(t-)+ (1+18 i)\delta(t+)\right),
\end{equation}
$$
\RE\bA_{11} x=i\frac{dx}{dt}-\frac{i}{\sqrt{65}}\Big(\sqrt{65} x(0-)+(7-4i)x(0+)\Big)\chi_1,
$$
and 
$$
\begin{aligned}
\bA_{11} x&=i\frac{dx}{dt}-\frac{i}{20}\Big(\sqrt{65} x(0+)+13x(0-)\Big)\left(\sqrt{65}(4+3i) \delta(t-)+ (20+35 i)\delta(t+)\right).
\end{aligned}
$$
Now we can compose an L-system
\begin{equation}\label{e6-125-11}
\Theta_{11}= 
\begin{pmatrix}
\bA_{11}&K_{11} &1\\
&&\\
\calH_+ \subset \calH \subset\calH_- &{ } &\dC
\end{pmatrix},
\end{equation}
where $\calH_+ \subset \calH \subset\calH_-$ is of the form \eqref{e-139-triple}, $K_{11} c=c\cdot \chi_1$, $(c\in \dC)$, $K^\ast_{11} x=(x,\chi_1)$, $\chi_1$ is given by \eqref{e-164-chi1}, and $x(t)\in \calH_+$. This L-system $\Theta_{11}$ is yet another realization of our
function $V(z)=1+\frac{1}{2}i$, ($z\in\dC_+$). According to \cite[Theorem 6.6.10]{ABT} the L-systems $\Theta_{1}$ and $\Theta_{11}$ are bi-unitarily equivalent. Also, using \eqref{e6-3-6} we get
\begin{equation}\label{e-180-W}
W_{\Theta_{11}}(z)=W_{\Theta_{1}}(z)=-\frac{1+8i}{5}, \quad z\in\dC_+.
\end{equation}

\vskip.5cm


\noindent
\textbf{Example 3.} This example is meant to illustrate the concept of unimodular transformation and the results of Theorem \ref{t-22-Q}. As before we will rely on the objects constructed in Examples 1 and 2. Let
$$
V(z)=\frac{1}{2}i\quad\textrm{ and }\quad V_0(z)=i,\quad z\in\dC_+.
$$
As we have shown this in Examples 1 and 2 these functions $V(z)$ and $V_0(z)$ are realized by L-systems $\Xi$ and $\Theta_0$ of the forms \eqref{e-154-mom_0} and \eqref{e6-125-mom}, respectively, with all the components completely described in Examples 1 and 2. As we have proved it in Theorem \ref{t-22-Q} part (2), L-system $\Xi$ cannot be a unimodular transformation of $\Theta_0$. On the other hand, part (3) of Theorem \ref{t-22-Q} claims that for $Q=1$ there is a value of perturbing parameter $Q_0$ such that $Q_0+V_0(z)$ can be realized by an L-system $\Theta^{Q_0}$ that is a unimodular transformation of $\Theta_{11}$ of the form \eqref{e6-125-11} that realizes $1+V(z)=1+\frac{1}{2}i$. First, we use \eqref{e-92-Q-0} to find one value of $Q_0$. Using the value of $\kappa$ from \eqref{e-161-k-U} we obtain
$$
Q_0=\frac{2\frac{\sqrt{65}}{13}}{\sqrt{1-\frac{65}{169}}}=\sqrt{\frac{5}{2}}.
$$
Now we construct an L-system $\Theta$ that realizes the function
$$
Q_0+V_0(z)=\sqrt{\frac{5}{2}}+i,\quad z\in\dC_+,
$$
and is a unimodular transformation of $\Theta_{11}$ of the form \eqref{e6-125-11}. As we have shown  in the proof of part (3) of Theorem \ref{t-22-Q}, the L-system $\Theta$ we seek shares the state space $\calH_+ \subset \calH \subset\calH_-$  of the form \eqref{e-139-triple}, symmetric operator $\dA$ of the form \eqref{e-87-sym}, and main operator of the form \eqref{e-162-T1} with L-system $\Theta_{11}$. In order to find the value of von Neumann's parameter $U$ we use \eqref{e-54-U-M-q}
\begin{equation}\label{e-180-U}
    U=\frac{-Q_0+2i}{\sqrt{Q^2_0+4}}=\frac{-\sqrt{\frac{5}{2}}+2i}{\sqrt{\frac{5}{2}+4}}=\frac{-\sqrt5+2\sqrt2 i}{\sqrt{13}}.
\end{equation}
Following formulas \eqref{e-212} and \eqref{e-205-A} of Appendix \ref{A2} with values for $\kappa$ and $U$ from \eqref{e-161-k-U} and \eqref{e-180-U} we construct
\begin{equation}\label{e-164-chi2}
\chi_{2}=\frac{1}{2\sqrt{13}}\left((\sqrt{26}+\sqrt{65} i) \delta(t-)+ (\sqrt{10}+9i)\delta(t+)\right),
\end{equation}
and 
$$
\begin{aligned}
\bA_{2} x&=i\frac{dx}{dt}-\frac{i}{4}\Big(\sqrt{65} x(0+)+13x(0-)\Big)\left((\sqrt{26}+\sqrt{65} i) \delta(t-)+ (\sqrt{10}+9i)\delta(t+)\right).
\end{aligned}
$$
Now we can compose an L-system
\begin{equation}\label{e6-125-2}
\Theta= 
\begin{pmatrix}
\bA_{2}&K_{2} &1\\
&&\\
\calH_+ \subset \calH \subset\calH_- &{ } &\dC
\end{pmatrix},
\end{equation}
where $\calH_+ \subset \calH \subset\calH_-$ is of the form \eqref{e-139-triple}, $K_{2} c=c\cdot \chi_2$, $(c\in \dC)$, $K^\ast_{2} x=(x,\chi_2)$, $\chi_2$ is given by \eqref{e-164-chi2}, and $x(t)\in \calH_+$.
Its transfer function is
$$
W_\Theta(z)=-\frac{5+2\sqrt{10}\,i}{5},\quad z\in\dC_+
$$
and it is related to the transfer function  $W_{\Theta_{11}}$ in \eqref{e-180-W} as follows
$$
W_\Theta(z)=W_{\Theta_{11}}(z)\cdot(-e^{2i\alpha})=\frac{-1-8i}{5}\cdot(-e^{2i\alpha}),\quad z\in\dC_+.
$$
This allows us to find the unimodular factor responsible for the unimodular transformation
$$
-e^{2i\alpha}=\frac{5+2\sqrt{10}i}{1+8i}.
$$
Clearly, $\Theta$ is a unimodular transformation of $\Theta_{11}$.

\appendix

\section{$(*)$-extensions as state-space operators of L-systems}\label{A2}

Here we provide an explicit construction of  an L-system based upon a given $(*)$-extension that becomes the state-space operator of an obtained system.
We will also demonstrate the case when the corresponding operators of this L-system  satisfy the conditions of Hypotheses \ref{setup} or \ref{setup-1}.  This construction can be found in \cite{T69} and its detailed treatment in \cite[Section 4]{BMkT-2}.

Let $\dA$ be a densely defined closed symmetric operator with finite deficiency indices $(1,1)$ and $(+)$-normalized deficiency vectors $g_+$ and $g_-$. Let $T$ be a dissipative quasi-self-adjoint extension of $\dA$ parameterized (see \eqref{parpar}) with the von Neumann parameter $\kappa$, ($0\le\kappa<1$) and $A$ be a self-adjoint extension of $\dA$ whose von Neumann parameter in \eqref{DOMHAT} is $U$. Let $S_{\bA}$ and $S_{\bA^*}$ be $(2\times2)$-matrices of the form
\begin{equation}\label{e4-62}
    S_\bA=\left(
            \begin{array}{cc}
              H\kappa & H \\
              \kappa^2 H+i\kappa & i+\kappa H \\
            \end{array}
          \right),\quad
S_{\bA^*}=\left(
            \begin{array}{cc}
              \kappa \bar H -i & \kappa^2 \bar H -i\kappa \\
              \bar H  & \bar H \kappa \\
            \end{array}
          \right),
\end{equation}
where
\begin{equation}\label{e3-39-new}
H=\frac{i}{1-\kappa^2}\left(\frac{\kappa+\bar U}{1+\kappa\bar U}+\kappa \right).
\end{equation}
 Then  any ($*$)-extension $\bA$ of $T$ takes a form
\begin{equation}\label{e3-40}
    \bA=\dA^*+\left[p(\cdot,\varphi)+q(\cdot,\psi) \right]\varphi     +\left[v(\cdot,\varphi)+w(\cdot,\psi) \right]\psi,
\end{equation}
where $S_{\bA}=\left(\begin{array}{cc}
                                        p & q \\
                                        v & w \\
                                      \end{array}
                                    \right)$ is defined by \eqref{e4-62}. Here $\varphi=\calR^{-1}(g_+)$ and $\psi=\calR^{-1}(g_-)$, where $\calR$ is a  Riesz-Berezansky   operator.  Similarly we write
\begin{equation}\label{e-21-star}
    \bA^*=\dA^*+\left[p^\times(\cdot,\varphi)+q^\times(\cdot,\psi) \right]\varphi
    +\left[v^\times(\cdot,\varphi)+w^\times(\cdot,\psi) \right]\psi,
\end{equation}
where $S_{\bA^*}=\left(\begin{array}{cc}
                                        p^\times & q^\times \\
                                        v^\times & w^\times \\
                                      \end{array}
                                    \right)$
is also defined by \eqref{e4-62}.

Any ($*$)-extension $\bA$ in \eqref{e3-40} can be included in an L-system of the form \eqref{e-62}. For the sake of simplicity we are going to illustrate this inclusion process for the case of Hypotheses \ref{setup} or \ref{setup-1}. First, (see \cite{BMkT}) let us assume Hypotheses \ref{setup} yielding $U=-1$. Then $\bA=\bA_1$  becomes
\begin{equation}\label{e-H1}
H=\frac{i}{1-\kappa^2}[(\kappa-1)(1-\kappa)^{-1}+\kappa]=\frac{-i}{1+\kappa},
\end{equation}
and
\begin{equation}\label{e-15}
    \begin{aligned}
    S_{\bA_1}&=\left(
            \begin{array}{cc}
              -\frac{i\kappa}{1+\kappa} & \frac{-i}{1+\kappa} \\
              \frac{i\kappa }{1+\kappa} & \frac{i}{1+\kappa} \\
            \end{array}
          \right)=\frac{i}{1+\kappa}\left(
                                            \begin{array}{cc}
                                              -\kappa & -1 \\
                                              \kappa & 1 \\
                                            \end{array}
                                          \right),\\
S_{\bA^*_1}&=\left(
            \begin{array}{cc}
              \frac{-i}{1+\kappa} & -\frac{i\kappa}{1+\kappa}  \\
              \frac{i}{1+\kappa} & \frac{i\kappa}{1+\kappa} \\
            \end{array}
          \right)=\frac{i}{1+\kappa}\left(
                                            \begin{array}{cc}
                                              -1 & -\kappa \\
                                              1 & \kappa \\
                                            \end{array}
                                          \right).
          \end{aligned}
\end{equation}
Then using \eqref{e3-40} and \eqref{e-21-star} with \eqref{e-15} one obtains (see \cite{BMkT})
\begin{equation}\label{e-17}
    \begin{aligned}
    \IM\bA_1&=\left(\frac{1}{2}\right)\frac{1-\kappa}{1+\kappa}(\cdot,\varphi-\psi)(\varphi- \psi)=(\cdot,\chi_1)\chi_1,
       \end{aligned}
\end{equation}
where
\begin{equation}\label{e-18}
    \chi_1=\sqrt{\frac{1-\kappa}{2+2\kappa}}\,(\varphi- \psi)=\sqrt{\frac{1-\kappa}{1+\kappa}}\left(\frac{1}{\sqrt2}\,\varphi- \frac{1}{\sqrt2}\,\psi\right).
\end{equation}
Also,
\begin{equation}\label{e-17-real}
    \begin{aligned}
    \RE\bA_1&=\dA^*+\frac{i}{2}(\cdot,\varphi+\psi)(\varphi-\psi).
       \end{aligned}
\end{equation}
As one can see from \eqref{e-17-real}, the domain $\dom(\hat A_1)$ of the quasi-kernel $\hat A_1$ of $\RE\bA_1$ consists of such vectors $f\in\calH_+$ that are orthogonal to $(\varphi+\psi)$.
The $(*)$-extension $\bA_1$  that we have just described  can be included in an L-system
\begin{equation}\label{e-62-1-1}
\Theta_1= \begin{pmatrix} \bA_1&K_1&\ 1\cr \calH_+ \subset \calH \subset
\calH_-& &\dC\cr \end{pmatrix}
\end{equation}
with $K_1c=c\cdot\chi_1$, $(c\in\dC)$.

Now let us assume (see \cite{BMkT-2}) the case of Hypotheses \ref{setup-1} with $U=1$ and describe a $(*)$-extension $\bA=\bA_2$.
Then formula \eqref{e3-39-new} yields
\begin{equation}\label{e-H2}
H=\frac{i}{1-\kappa^2}[(\kappa+1)(1+\kappa)^{-1}+\kappa]=\frac{i}{1-\kappa}.
\end{equation}
Similarly to the above, we substitute this value of $H$ into \eqref{e4-62} and obtain
$$
   \begin{aligned}
    S_{\bA_2}&=\left(
            \begin{array}{cc}
              \frac{i\kappa}{1-\kappa} & \frac{i}{1-\kappa} \\
              \frac{i\kappa}{1-\kappa} & \frac{i}{1-\kappa} \\
            \end{array}
          \right)=\frac{i}{1-\kappa}\left(
                                            \begin{array}{cc}
                                              \kappa & 1 \\
                                              \kappa & 1 \\
                                            \end{array}
                                          \right),\\
S_{\bA^*_2}&=\left(
            \begin{array}{cc}
              -\frac{i}{1-\kappa} & \frac{-i\kappa}{1-\kappa}  \\
              -\frac{i}{1-\kappa} & \frac{-i\kappa}{1-\kappa} \\
            \end{array}
          \right)=\frac{i}{1-\kappa}\left(
                                            \begin{array}{cc}
                                              -1 & -\kappa \\
                                              -1 & -\kappa \\
                                            \end{array}
                                          \right).
          \end{aligned}
$$
Furthermore,
\begin{equation}\label{e-17-1}
       \IM\bA_2= \left(\frac{1}{2}\right)\frac{1+\kappa}{1-\kappa}\Big((\cdot,\varphi+\psi)(\varphi+\psi)\Big)=(\cdot,\chi_2)\chi_2,
    \end{equation}
where
\begin{equation}\label{e-18-1}
    \chi_2=\sqrt{\frac{1+\kappa}{2-2\kappa}}\,(\varphi+ \psi)=\sqrt{\frac{1+\kappa}{1-\kappa}}\left(\frac{1}{\sqrt2}\,\varphi+ \frac{1}{\sqrt2}\,\psi\right).
\end{equation}
Also,
\begin{equation}\label{e-32-real}
    \begin{aligned}
    \RE\bA_2&=\dA^*-\frac{i}{2}(\cdot,\varphi-\psi)(\varphi+ \psi).
    \end{aligned}
\end{equation}
As one can see from \eqref{e-32-real}, the domain $\dom(\hat A_2)$ of the quasi-kernel $\hat A_2$ of $\RE\bA_2$ consists of such vectors $f\in\calH_+$ that are orthogonal to $(\varphi- \psi)$.
Again we include $\bA_2$  into an L-system
\begin{equation}\label{e-62-1-3}
\Theta_2= \begin{pmatrix} \bA_2&K_2&\ 1\cr \calH_+ \subset \calH \subset
\calH_-& &\dC\cr \end{pmatrix}
\end{equation}
with $K_2c=c\cdot\chi_2$, $(c\in\dC)$.

Note that two L-systems $\Theta_1$ and $\Theta_2$ in \eqref{e-62-1-1} and \eqref{e-62-1-3} are constructed in a way that the quasi-kernels $\hat A_1$ of $\RE\bA_1$ and $\hat A_2$ of $\RE\bA_2$ satisfy the conditions of Hypotheses \ref{setup} and \ref{setup-1}, respectively, as it follows from \eqref{e-17-real} and \eqref{e-32-real}.

Now we make a similar construction for an arbitrary parameter $U$. It follows directly from \eqref{e4-62} (see also \cite{ABT}) that
\begin{equation}\label{e-203}
    \frac{S_{\bA}+S_{\bA^*}}{2}=\frac{1}{2}\left(
            \begin{array}{cc}
              H\kappa+\bar H\kappa-i & H+\kappa^2\bar H-i\kappa \\
              \kappa^2 H+\bar H+i\kappa & i+\kappa H+\kappa\bar H \\
            \end{array}
          \right)
\end{equation}
and
\begin{equation}\label{e-204}
    \frac{S_{\bA}-S_{\bA^*}}{2i}=\frac{1}{2i}\left(
            \begin{array}{cc}
              H\kappa-\bar H\kappa+i & H-\kappa^2\bar H+i\kappa \\
              \kappa^2 H-\bar H+i\kappa & i+\kappa H-\kappa\bar H \\
            \end{array}
          \right).
\end{equation}
To simplify calculations we write the matrix in  \eqref{e-203} as
\begin{equation}\label{e-205}
 \frac{S_{\bA}+S_{\bA^*}}{2}=\left(
            \begin{array}{cc}
              \calA & \calB \\
              \calC & \calD \\
            \end{array}
          \right),
\end{equation}
where $\calA$, $\calB$, $\calC$, and $\calD$ are the corresponding entries in the right hand side of \eqref{e-203}. It can be shown then (see also \cite[Theorem 6.3.7]{ABT}) that
\begin{equation}\label{e-206}
    \calB=-\bar U\calA, \quad \calD=-\bar U\calC.
\end{equation}
Using \eqref{e3-40} and \eqref{e-21-star} in conjunction with \eqref{e-205}  and \eqref{e-206} we obtain
\begin{equation}\label{e-207}
        \RE\bA=\dA^*+(\cdot,\varphi-U \psi)(\calA\varphi+ \calC\psi).
\end{equation}
It can be shown (see also \cite{TSh1}) that the matrix
\begin{equation}\label{e-208}
    \Delta=\frac{1}{\kappa^2-1}\left(
            \begin{array}{cc}
              \kappa^2+1 & 2\kappa \\
              -2\kappa& -\kappa^2-1 \\
            \end{array}
          \right)
\end{equation}
is such that $\Delta^2=I$ and
\begin{equation}\label{e-209}
    \frac{S_{\bA}+S_{\bA^*}}{2}=i\left(\frac{S_{\bA}-S_{\bA^*}}{2i}\cdot\Delta \right).
\end{equation}
Using \eqref{e-205} this yields
$$
\begin{aligned}
\frac{S_{\bA}-S_{\bA^*}}{2i}&=(-i)\frac{S_{\bA}+S_{\bA^*}}{2}\cdot\Delta\\&=\frac{-i}{\kappa^2-1}\left(
            \begin{array}{cc}
              \calA(\kappa^2+1+2\kappa\bar U) & \calA(2\kappa+\bar U(\kappa^2+1) \\
              \calC(\kappa^2+1+2\kappa\bar U)&\calC(2\kappa+\bar U(\kappa^2+1) \\
            \end{array}
          \right).
\end{aligned}
$$
Applying \eqref{e3-40} and \eqref{e-21-star} with \eqref{e-205} and performing straightforward calculations we obtain
\begin{equation}\label{e-210}
    \begin{aligned}
        \IM\bA&=\left(\cdot,i\frac{\kappa^2+1+2\kappa U}{\kappa^2-1}\varphi+i\frac{2\kappa+ U(\kappa^2+1)}{\kappa^2-1}\psi\right)(\calA\varphi+ \calC\psi)\\
        &=\frac{-i}{\kappa^2-1}\left(\cdot,({\kappa^2+1+2\kappa U})\varphi+({2\kappa+ U(\kappa^2+1)})\psi\right)(\calA\varphi+ \calC\psi).
    \end{aligned}
\end{equation}
We recall that
$$
\calA= \frac{1}{2}(H\kappa+\bar H\kappa-i)\quad\textrm{ and }\quad\calC=\frac{1}{2}(\kappa^2 H+\bar H+i\kappa),
$$
and use \eqref{e-H2} to substitute the value for $H$  and obtain
\begin{equation}\label{e-211}
     \begin{aligned}
    \calA&= \left(-\frac{i U}{2}\right)\frac{\kappa^2+1+2\kappa U}{(U+\kappa)(1+\kappa U)}=\left(-\frac{i}{2}\right)\frac{\kappa^2+1+2\kappa U}{|1+\kappa U|^2},\\
    \calC&= \left(-\frac{i U}{2}\right)\frac{\kappa^2 U+2\kappa+ U}{(U+\kappa)(1+\kappa U)}=\left(-\frac{i}{2}\right)\frac{\kappa^2 U+2\kappa+ U}{|1+\kappa U|^2}.
     \end{aligned}
\end{equation}
Substituting \eqref{e-211} into \eqref{e-210} gives
$$
\begin{aligned}
        \IM\bA&=\frac{-i}{\kappa^2-1}\left(\cdot,({\kappa^2+1+2\kappa U})\varphi+({2\kappa+ U(\kappa^2+1)})\psi\right)(\calA\varphi+ \calC\psi)\\
        &=\frac{1}{2}\left(\cdot,  \frac{\kappa^2+1+2\kappa U}{|1+\kappa U|\sqrt{1-\kappa^2}}\varphi+ \frac{\kappa^2 U+2\kappa+ U}{|1+\kappa U|\sqrt{1-\kappa^2}}\psi\right)\\
        &\times\left(\frac{\kappa^2+1+2\kappa U}{|1+\kappa U|\sqrt{1-\kappa^2}}\varphi+ \frac{\kappa^2 U+2\kappa+ U}{|1+\kappa U|\sqrt{1-\kappa^2}}\psi\right)=(\cdot,\chi)\chi,
    \end{aligned}
$$
where
\begin{equation}\label{e-212}
    \chi=\frac{\kappa^2+1+2\kappa U}{\sqrt2|1+\kappa U|\sqrt{1-\kappa^2}}\varphi+ \frac{\kappa^2 U+2\kappa+ U}{\sqrt2|1+\kappa U|\sqrt{1-\kappa^2}}\psi.
\end{equation}
Also, \eqref{e-207} implies
\begin{equation}\label{e-214}
    \begin{aligned}
    \RE\bA&=\dA^*-\frac{i(\cdot,\varphi-U\psi)}{2|1+\kappa U|^2}\left((\kappa^2+1+2\kappa U)\varphi+ (\kappa^2 U+2\kappa+ U)\psi\right)\\
&=\dA^*-\frac{i\sqrt{1-\kappa^2}}{\sqrt2|1+\kappa U|}(\cdot,\varphi-U\psi)\chi.
    \end{aligned}
\end{equation}
Using the above relations we get
$$
 \begin{aligned}
\bA&=\RE\bA+i\IM\bA=\dA^*-\left(\cdot,\frac{-i\sqrt{1-\kappa^2}}{\sqrt2|1+\kappa U|}(\cdot,\varphi-U\psi)\right)\chi+i(\cdot,\chi)\chi\\
&=\dA^*-\left(\cdot,\frac{-i\sqrt{1-\kappa^2}}{\sqrt{2}|1+\kappa U|}\varphi+ \frac{iU\sqrt{1-\kappa^2}}{\sqrt{2}|1+\kappa U|}\psi\right)\chi\\
&\quad+{\frac{1}{\sqrt2}}\left(\cdot,-i\frac{\kappa^2+1+2\kappa U}{|1+\kappa U|\sqrt{1-\kappa^2}}\varphi-i\frac{\kappa^2 U+2\kappa+U}{|1+\kappa U|\sqrt{1-\kappa^2}}\psi\right)\chi\\
&=\dA^*+\left(\cdot,\left[\frac{i\sqrt{1-\kappa^2}}{\sqrt{2}|1+\kappa U|}-\frac{i(\kappa^2+1+2\kappa U)}{\sqrt2|1+\kappa U|\sqrt{1-\kappa^2}}\right]\varphi\right.\\
&\quad+\left.\left[\frac{-iU\sqrt{1-\kappa^2}}{\sqrt{2}|1+\kappa U|}-\frac{i(\kappa^2 U+2\kappa +U)}{\sqrt2|1+\kappa U|\sqrt{1-\kappa^2}}\right]\psi\right)\chi.\\
    \end{aligned}
$$
Simplifying the two sets of brackets above yields
\begin{equation}\label{e-205-A}
\begin{aligned}
\bA&=\dA^*+\left(\cdot, \frac{-\sqrt2 i\kappa(\kappa+U)}{|1+\kappa U|\sqrt{1-\kappa^2}}\varphi+\frac{-\sqrt2 i(\kappa+U)}{|1+\kappa U|\sqrt{1-\kappa^2}}\psi\right)\chi\\
&=\dA^*+\frac{\sqrt2 i(\kappa+\bar U)}{|1+\kappa U|\sqrt{1-\kappa^2}}\Big(\cdot, \kappa\varphi+\psi\Big)\chi.
   \end{aligned}
\end{equation}
Again we include $\bA$  into an L-system
\begin{equation}\label{e-215}
\Theta= \begin{pmatrix} \bA&K&\ 1\cr \calH_+ \subset \calH \subset
\calH_-& &\dC\cr \end{pmatrix}
\end{equation}
with $K c=c\cdot\chi$, $(c\in\dC)$. To find  $V_\Theta(z)$  we  solve the equation $(\RE\bA-z I)f=\chi$. This equation has  (see \cite[Theorems 4.3.2 and 4.5.12]{ABT}) a unique solution $f\in\ker(\dA^*-zI)$ that  can be written via von Neumann's  decomposition as
$$
f=f_0(z)+a(z)g_++b(z)g_-,\quad f_0(z)\in\DdA,
$$
where $a(z)$ and $b(z)$ are some coefficients that depend on $z$. Applying \eqref{e-214} we get
$$
\begin{aligned}
(\RE\bA-z I)f&=(\dA^*-z I)f-\frac{i\sqrt{1-\kappa^2}}{\sqrt2|1+\kappa U|}(f,\varphi-U\psi)\chi\\
&=-\frac{i\sqrt{1-\kappa^2}}{\sqrt2|1+\kappa U|}(f_0(z)+a(z)g_++b(z)g_-,\varphi-U\psi)\chi=\chi.
\end{aligned}
$$
Therefore,
$$
-\frac{i\sqrt{1-\kappa^2}}{\sqrt2|1+\kappa U|}(a(z)-b(z)\bar U)\chi=\chi,
$$
and hence
$$
-\frac{i\sqrt{1-\kappa^2}}{\sqrt2|1+\kappa U|}\Big(a(z)-b(z)\bar U\Big)=1.
$$
On the other hand, we know that
$$
V_\Theta(z)=((\RE\bA-z I)^{-1}\chi,\chi )=a(z)\frac{\overline{\kappa^2+1+2\kappa U}}{\sqrt2|1+\kappa U|\sqrt{1-\kappa^2}}+b(z)\frac{\overline{\kappa^2 U+2\kappa+U}}{\sqrt2|1+\kappa U|\sqrt{1-\kappa^2}}.
$$
Combining the above leads to a linear system for every $z\in\dC_\pm$
\begin{equation}\label{e-216-ls}
    \left\{
      \begin{array}{ll}
        \frac{\overline{\kappa^2+1+2\kappa U}}{\sqrt2|1+\kappa U|\sqrt{1-\kappa^2}}\, a(z)+\frac{\overline{\kappa^2 U+2\kappa+U}}{\sqrt2|1+\kappa U|\sqrt{1-\kappa^2}}\,b(z)=V_\Theta(z) & \hbox{} \\
&\\
        -\frac{i\sqrt{1-\kappa^2}}{\sqrt2|1+\kappa U|}a(z)+\frac{i\sqrt{1-\kappa^2}\bar U}{\sqrt2|1+\kappa U|}b(z)=1. & \hbox{}
      \end{array}
    \right.
\end{equation}
This system is always solvable because the determinant of the coefficient matrix equals $i\bar U$ for every $z\in\dC_\pm$ and hence is never zero. Obviously, $a(-i)=b(i)=0$.

\section{Model L-system based on a prime  triple}\label{A1}

In this Appendix we are going to explain the construction of a  functional model for a prime  dissipative triple\footnote{We call a triple $(\dot A, \whA , A)$
 a \textit{prime triple} if $\dot A$ is a prime symmetric operator with deficiency indices $(1,1)$,  $A$ is its self-adjoint extension, and $T \ne T^*$ is its dissipative extension defined by \eqref{parpar}.} (see \cite{BMkT}) and construct a minimal L-system based on that triple.


Let $M(z)$ be a Herglotz-Nevanlinna function such that
$$
M(z)=\int_\bbR \left
(\frac{1}{\lambda-z}-\frac{\lambda}{1+\lambda^2}\right )
d\mu(\lambda), \quad z\in \bbC_+,
$$
for some infinite Borel measure with
$$
\int_\bbR\frac{d\mu(\lambda)}{1+\lambda^2}=1.
$$
In the Hilbert space $L^2(\bbR;d\mu)$ introduce  the multiplication (self-adjoint) operator by the  independent variable $\cB$  on
\begin{equation}\label{nacha1-ap}
\dom(\cB)=\left \{f\in \,L^2(\bbR;d\mu) \,\bigg | \, \int_\bbR
\lambda^2 | f(\lambda)|^2d \mu(\lambda)<\infty \right \},
\end{equation} denote by  $\dot \cB$  its  restriction on
\begin{equation}\label{nacha2-ap}
\dom(\dot \cB)=\left \{f\in \dom(\cB)\, \bigg | \, \int_\bbR
f(\lambda)d \mu(\lambda) =0\right \},
\end{equation}
and let  $\whB $ be   the dissipative restriction of the operator  $(\dot \cB)^*$  on
\begin{equation}\label{nacha3-ap}
\dom(\whB )=\dom (\dot \cB)\dot +\linspan\left
\{\,\frac{1}{\cdot -i}- \kappa\frac{1}{\cdot +i}\right \},
\end{equation}
where $\kappa$, ($|\kappa|<1$) is the von Neumann parameter  of $\whB$. The deficiency elements $g_z\in\ker(\dot\cB^*-zI)$, ($z\ne\bar z$) are given by (see \cite{MT-S})
\begin{equation}\label{e-52-def}
g_z(\lambda)=\frac{1}{\lambda-z}, \quad
\,\, \text{$\mu$-a.e. }.
\end{equation}
The Liv\u sic function $s(z)$, the Weyl-Titchmarsh function $M(z)$ (see \cite{D}) for the pair $(\dA,A)$, and characteristic function $S(z)$ for the triple $(\dot A,\whA ,A)$  and their relations were introduced in \cite{MT-S} for a symmetric operator $\dA$, its self-adjoint extension $A$, and its dissipative  quasi-selfadjoint extension $T$ in an abstract Hilbert space $\calH$. For operators $\dot \cB$,   $\whB$, $\cB$ these functions are
\begin{equation}\label{e-197-Liv}
s(z)=\frac{z-i}{z+i}\cdot \frac{(g_z, g_-)}{(g_z, g_+)},\quad z\in \bbC_+,
\end{equation}
\begin{equation*}\label{e-WT}
M(z)=\frac1i\cdot\frac{s(z)+1}{s(z)-1},\quad z\in \bbC_+,
\end{equation*}
\begin{equation}\label{e-198-char}
S(z)=\frac{s(z)-\kappa} {\overline{ \kappa }\,s(z)-1}, \quad z\in \bbC_+.
\end{equation}
 It was established in \cite{MT-S} that a prime triple $(\dot A,\whA ,A)$ satisfying  Hypothesis \ref{setup}  is unitarily equivalent to the model triple $(\dot \cB,   \whB ,\cB)$ in the Hilbert space $L^2(\bbR;d\mu)$ if their characteristic functions $S(z)$ match. We will refer to the triple  $(\dot \cB,   \whB ,\cB)$ as {\it  the model  triple } in the Hilbert space $L^2(\bbR;d\mu)$.

Suppose $0\le\kappa<1$. We can follow the steps described in Appendix \ref{A2} (see also \cite[Section 4]{BMkT-2}) to construct a model L-system
\begin{equation}\label{e-59'}
\Theta_1= \begin{pmatrix} \dB_1&K_1&\ 1\cr \calH_+ \subset \calH \subset
\calH_-& &\dC\cr \end{pmatrix},
\end{equation}
corresponding to our model triple $(\dot \cB,   \whB ,\cB)$.
Here $\dB_1\in[\calH_+,\calH_-]$ is a  $(*)$-extension of $\whB $ such that $\RE\dB_1\supset \cB=\cB^*$ (see \eqref{e-17-real} under an appropriate normalization condition). Below we describe the construction of $\dB_1$. Relation \eqref{ddoomm14} of Hypothesis \ref{setup} implies that $g_+- g_-\in \dom(\cB)$ and $g_+-\kappa g_-\in \dom (\whB )$.
As it was shown in \cite[Section 4]{BMkT-2}
\begin{equation}\label{e-57'}
       \IM\dB_1=(\cdot,\chi_1)\chi_1,\quad \chi_1=\sqrt{\frac{1-\kappa}{1+\kappa}}\left(\frac{1}{2}\,\calR^{-1}g_+- \frac{1}{2}\,\calR^{-1}g_-\right),
    \end{equation}
and  $K_1 c=c\cdot \chi_1$, $K_1^*f=(f,\chi_1)$, $(f\in\calH_+)$.

Similarly, one can construct a model L-system complying with Hypothesis \ref{setup-1}. To do that we need to replace the operator $\cB$ defined on \eqref{nacha1-ap} in the model triple $(\dot \cB,   \whB ,\cB)$ with operator $\cB_1$ whose domain is given by
\begin{equation}\label{nacha1-a}
\dom(\cB_1)=\dom (\dot \cB)\dot +\linspan\left\{\,\frac{2(\cdot)}{(\cdot)^2 +1}\right \}.
\end{equation}
Then, as it was shown in \cite[Section 5]{BMkT-2}, the new model triple $(\dot \cB,   \whB ,\cB_1)$ is consistent with Hypothesis \ref{setup-1}. The corresponding to this model L-system is shown in \cite[Section 4]{BMkT-2} as
\begin{equation}\label{e-59''}
\Theta_2= \begin{pmatrix} \dB_2&K_2&\ 1\cr \calH_+ \subset \calH \subset
\calH_-& &\dC\cr \end{pmatrix}.
\end{equation}
Here $\dB_2\in[\calH_+,\calH_-]$ is another  $(*)$-extension of the same operator $\whB$  such that $\RE\dB_2\supset \cB_1=\cB_1^*$ (see \eqref{e-32-real}).  Hypothesis \ref{setup-1} implies that $g_++ g_-\in \dom(\cB_1)$. Consequently,
\begin{equation}\label{e-57''}
       \IM\dB_2=(\cdot,\chi_2)\chi_2,\quad \chi_2=\sqrt{\frac{1+\kappa}{1-\kappa}}\left(\frac{1}{2}\,\calR^{-1}g_++ \frac{1}{2}\,\calR^{-1}g_-\right),
    \end{equation}
and  $K_2 c=c\cdot \chi_2$, $K_2^*f=(f,\chi_2)$, $(f\in\calH_+)$.

A model L-system can be constructed for a model operator $\cB_U$ defined on
\begin{equation}\label{e-187-new}
\dom(\cB_U)=\dom (\dot \cB)\dot +\linspan\left\{\,\frac{1}{\cdot -i}+ U\frac{1}{\cdot +i}\right \},
\end{equation}
where $U$ is a complex number such that $|U|=1$. Hence we have that $g_++U g_-\in \dom(\cB_U)$. The model L-system then takes a form
\begin{equation}\label{e-59}
\Theta= \begin{pmatrix} \dB&K&\ 1\cr \calH_+ \subset \calH \subset
\calH_-& &\dC\cr \end{pmatrix}.
\end{equation}
Here $\dB\in[\calH_+,\calH_-]$ is a  $(*)$-extension of $\whB $ such that $\RE\dB\supset \cB_U=\cB_U^*$.
Construction of $\dB$ is governed by formulas \eqref{e3-40} and \eqref{e-21-star} of Appendix \ref{A2}.

If the value of $\kappa$ of $\whB$ in \eqref{nacha3-ap} is complex and $\kappa=|\kappa|e^{i\theta}$, then we can change the deficiency basis $g_\pm$ to a new one $g_+$ and $(e^{i\theta})g_-$ as explained in Remark \ref{r-12}. Then, following the steps above developed for the real value of $\kappa$,  we can construct L-system $\Theta$ of the form \eqref{e-59} such that $\dB$ is a ($*$)-extension of $\whB$ and $\cB_U$ is the quasi-kernel of $\RE\dB$ for some $U$.

We note that all model L-systems constructed above are minimal since the symmetric operator $\dot \cB$ of the form \eqref{nacha2-ap} is prime \cite{MT-S}.



\begin{thebibliography}{1}


\bibitem{AG93}
N.I.~Akhiezer, I.M.~Glazman, \textit{Theory of linear operators.} \newblock Pitman {A}dvanced {P}ublishing {P}rogram, 1981.


\bibitem{AlTs2}
D.~Alpay,  E.~Tsekanovski\u{\i}, \textit{Interpolation theory in sectorial Stieltjes classes and explicit system solutions.} Lin. Alg. Appl., \textbf{314} (2000), 91--136.

\bibitem{AlTs3}
D.~Alpay,  E.~Tsekanovski\u{\i},  \textit{Subclasses of Herglotz-Nevanlinna matrix-valued functions and linear systems}. In: J. Du and S. Hu (ed)
{Dynamical systems and differential equations}, An added volume to \textit{Discrete and continuous dynamical systems},  (2001), 1--14.

\bibitem{ABT}
{Yu.~Arlinski\u{\i}, S.~Belyi, E.~Tsekanovski\u{\i}},
\textit{Conservative Realizations  of Herglotz-Nevanlinna functions}.  {Oper. Theory Adv. Appl.}, {Vol. 217},
Birkh\"auser/Springer Basel AG, Basel, 2011, 528 pp.


\bibitem{ArTs03}
Yu.~Arlinski\u{\i}, E.~Tsekanovski\u{\i}, \textit{Constant $J$-unitary factor and operator-valued transfer functions}. In: Dynamical systems and
differential equations, Discrete Contin. Dyn. Syst., Wilmington, NC, (2003), 48--56.

\bibitem{AtDerTsek81}
Yu.~Arlinski\u{\i}, V.~Derkach, E.~Tsekanovski\u{\i}, \textit{Unitarily equivalent quasi-Hermitian extensions of Hermitian operators}. (Russian) Matematicheskaya Fizika,   Akademiya Nauk Ukrainskoi SSR., Institut Matematiki, \textbf{29}, (1981), 72-–77.


\bibitem{BT3}
{S.~Belyi, E.~Tsekanovski\u{\i}}, \textit{Realization theorems for operator-valued $R$-functions}. {Oper. Theory Adv. Appl.}, {Vol. 98}, Birkh\"auser Verlag Basel, {(1997)}, {55--91}.


\bibitem{BT4}
{S.~Belyi, E.~Tsekanovski\u{\i}}, \textit{On classes of realizable operator-valued $R$-functions}.  {Oper. Theory Adv. Appl.}, {\bf 115}, Birkh\"auser Verlag Basel, (2000), 85--112.



\bibitem{BMkT}
 S.~Belyi, K.~A.~ Makarov, E.~Tsekanovski\u{\i},  \textit{Conservative  L-systems and the Liv\v{s}ic function}. Methods of Functional Analysis and Topology,  \textbf{21}, no. 2,  (2015),  104--133.

\bibitem{BMkT-2}
S.~Belyi,  K.A.~Makarov, E.~Tsekanovski\u{\i}, \textit{A system coupling and Donoghue classes of Herglotz-Nevanlinna functions},  Complex Analysis and Operator Theory,  \textbf{10} (4), (2016), 835-880.

\bibitem{BMkT-4}
S.~Belyi,  K.A.~Makarov, E.~Tsekanovski\u{\i}, \textit{On unimodular transformations of conservative L-systems}, {Oper. Theory Adv. Appl.},  {Vol. 263}, (2018), 191--215.  ArXiv: http://arxiv.org/abs/1608.08583.

\bibitem{Ber}
{Yu.~Berezansky},  \textit{Expansion in eigenfunctions of self-adjoint operators}. {Vol. 17}, {Transl. Math. Monographs}, {AMS}, {Providence}, {1968}.


\bibitem{Bro}
M.~Brodskii, \textit{Triangular and Jordan representations of linear operators}. Translations of Mathematical Monographs,  Vol. 32. American Mathematical Society, Providence, R.I., 1971.




\bibitem {D}
 W.F.~Donoghue,
\textit{On perturbation of spectra}. Commun. Pure and Appl. Math. {\bf 18} (1965),  559--579.


\bibitem{GT}
{F.~Gesztesy, E.~Tsekanovski\u i}, \textit{On Matrix-Valued Herglotz Functions}. Math. Nachr. {\bf 218} (2000), 61--138.



\bibitem{GMT97} F.~Gesztesy, K.A.~Makarov, E.~Tsekanovski\u i,  \textit{An addendum to Krein's formula}, {J. Math. Anal. Appl.} {\bf 222}, (1998), 594--606.


\bibitem{KK74}
{I.S.~Kac, M.G.~Krein}, \textit{$R$-functions -- analytic functions mapping the upper halfplane into itself}. {Amer. Math. Soc. Transl.}, {Vol. 2}, {\textbf 103}  (1974), {1--18}.



\bibitem{L}
M.~Liv\v{s}ic, \textit{On a class of linear operators in Hilbert space}. Mat. Sbornik  (2), {\bf 19} (1946), 239--262 (Russian); English transl.:  Amer. Math. Soc. Transl., (2), {\bf 13}, (1960), 61--83.





\bibitem{Lv2}
M.~Liv\v{s}ic, \textit{Operators, oscillations, waves}. Moscow, Nauka, 1966 (Russian).

\bibitem{LiYa}
{ M.~Liv\v{s}ic,  A.~Yantsevich}, \textit{Operator colligations in Hilbert spaces}. Winston \& Sons, 1979.


\bibitem{MT-S}
{K.~A.~Makarov, E.~Tsekanovski\u i}, \textit{On the Weyl-Titchmarsh and Liv\v{s}ic  functions}. Proceedings of Symposia in Pure Mathematics, Vol. 87, American Mathematical Society, (2013), 291--313.


\bibitem{MT10}
{K.~A.~Makarov, E.~Tsekanovski\u i}, \textit{On the addition and  multiplication theorems}.  	
 {Oper. Theory Adv. Appl.}, {Vol. 244} (2015), 315--339.



\bibitem{T69}
 E.~Tsekanovski\u i, \textit{The description and the uniqueness of generalized extensions of quasi-Hermitian operators}. (Russian) Funkcional. Anal. i Prilozen., \textbf{3}, no. 1,  (1969), 95--96.

\bibitem{TSh1}
{E.~Tsekanovski\u i, Yu.~\u Smuljan,}   \textit{The theory of bi-extensions of operators on rigged Hilbert spaces. Unbounded operator colligations and characteristic functions}.  Russ. Math. Surv. {\bf 32}  {(1977)}, {73--131}.

\bibitem{Zol}
V.~Zolotarev, \textit{Analytic methods in spectral representations of non-self-adjoint and non-unitary operators}. (Russian), Kharkov National University, Kharkov, 2003, 342 pp.

\end{thebibliography}
\end{document}